\documentclass[a4paper,11pt]{article}

\usepackage{amsfonts,qtree,stmaryrd,textcomp}
\usepackage{pifont,amssymb,amsmath,amsthm,tabularx,graphicx}
\usepackage{tree-dvips,pictexwd,dcpic}
\usepackage{bbm}
\usepackage{bm}

\usepackage{stmaryrd}

\usepackage[T1]{fontenc}
\usepackage[latin1]{inputenc}
\usepackage[text={17cm,26cm},centering]{geometry}

\usepackage{etex}
\usepackage{fancyhdr}
\usepackage{graphicx}
\usepackage{enumitem}
\usepackage{amsmath}
\usepackage[bbgreekl]{mathbbol}% order sensitive; it is before amssymb
\usepackage{amssymb}
\usepackage{amsthm}
\usepackage[all]{xy}

\usepackage{epigraph}
\RequirePackage{bussproofs}

\usepackage{mathabx}

\usepackage{scalerel,stackengine}

\usepackage{framed}
\usepackage{mathtools}
\usepackage{url}
\usepackage{proof}
\usepackage{xspace}
\usepackage{listings}
\usepackage{wrapfig}
\usepackage{tikz}
\usepackage{tikz-cd}

\usepackage{relsize}

\usetikzlibrary{positioning}
\usetikzlibrary{shapes}

\usetikzlibrary{arrows}
\usetikzlibrary{calc}
\usepackage{tikz-3dplot}
\usetikzlibrary{decorations.pathmorphing}

\usepackage{amsfonts,amssymb,amsmath,qtree,stmaryrd,textcomp, amsthm}

\usetikzlibrary{decorations.pathreplacing}
\tikzset{axis/.style={&lt;-&gt;}}

\stackMath
\newcommand\reallywidehat[1]{%
\savestack{\tmpbox}{\stretchto{%
  \scaleto{%
    \scalerel*[\widthof{\ensuremath{#1}}]{\kern-.6pt\bigwedge\kern-.6pt}%
    {\rule[-\textheight/2]{1ex}{\textheight}}%WIDTH-LIMITED BIG WEDGE
  }{\textheight}% 
}{0.5ex}}%
\stackon[1pt]{#1}{\tmpbox}%
}
 \definecolor{MyBlue}{rgb}{0.05, 0.25, 0.65}
 
 \definecolor{MyRed}{rgb}{0.90, 0.05, 0.05}
\definecolor{MyGreen}{rgb}{0.05, 0.90, 0.05}

\newcommand{\B}{\boldsymbol}
\newcommand{\C}[1]{\mathcal{#1}}
\newcommand{\X}[1]{\mathbb{#1}}
\usepackage{ mathrsfs }

\newtheorem{theorem}{Theorem}[section]
\newtheorem{proposition}[theorem]{Proposition}
\newtheorem{lemma}[theorem]{Lemma}
\newtheorem{corollary}[theorem]{Corollary}
\newtheorem{remark}[theorem]{Remark}

\newtheorem{example}[theorem]{Example}

\newtheorem{definition}[theorem]{Definition}

\newcommand{\Nat}{{\mathbb N}}
\newcommand{\Real}{{\mathbb R}}

\newcommand{\id}{\mathrm{id}}

\newcommand{\Bic}{\mathrm{Bic}}

\newcommand{\BISH}{\mathrm{BISH}}

\newcommand{\dom}{\mathrm{dom}}

\newcommand{\MP}{\mathrm{MP}}
\newcommand{\Mod}{\mathrm{Mod}}

\newcommand{\SN}{\mathrm{SN}}

\newcommand{\LPO}{\mathrm{LPO}}

\newcommand{\Top}{\mathrm{\mathbf{Top}}}

\newcommand{\INT}{\mathrm{INT}}
\newcommand{\CLASS}{\mathrm{CLASS}}

\newcommand{\TOT}{\Leftrightarrow}

\newcommand{\To}{\Rightarrow}

\newcommand{\sto}{\rightsquigarrow}

\newcommand{\MLTT}{\mathrm{MLTT}}

\newcommand{\CM}{\mathrm{CM}}

\newcommand{\pr}{\textnormal{\texttt{pr}}}

\newcommand{\BST}{\mathrm{BST}}

\newcommand{\Disj}{\B ) \B (}

\newcommand{\Set}{\mathrm{\mathbf{Set}}}

\newcommand{\eto}{\hookrightarrow}

\newcommand{\op}{\mathrm{op}}

\newcommand{\Fun}{\textnormal{\textbf{Fun}}}

\newcommand{\emptys}{\slash{\hspace{-2.05mm}}\square}

\newcommand{\MIN}{\mathrm{MIN}}

\newcommand{\Ineq}{\textnormal{\texttt{Ineq}}}

\newcommand{\SetIneq}{\textnormal{\textbf{SetIneq}}}

\newcommand{\SetExtIneq}{\textnormal{\textbf{SetExtIneq}}}

\newcommand{\fXY}{f \colon X \to Y}

\newcommand{\Eisj}{\B ] \B [}

\newcommand{\swapa}{\textnormal{\texttt{swapa}}}

\newcommand{\ti}{\mathrm{I}}

\newcommand{\tii}{\mathrm{II}}

\newcommand{\cTop}{\textnormal{\texttt{csTop}}}

\newcommand{\cin}{ \ \varepsilon  \ } 
\newcommand{\ccin}{ \varepsilon } 

\newcommand{\ncin}{\slash{\hspace{-3.15mm}} \cin}

\newcommand{\nncin}{ \ \slash{\hspace{-3.15mm}} \cin}

\newcommand{\Clop}{\mathrm{Clop}} 

\newcommand{\Point}{\mathrm{Point}}

\newcommand{\TotClop}{\mathrm{TotClop}}

\newcommand{\cPoint}{\mathrm{cPoint}}

\newcommand{\csTop}{\textnormal{\textbf{csTop}}}

\newcommand{\one}{\X 1}

\newcommand{\two}{\X 2}

\newcommand{\EEmpty}{\mathrm{Empty}} 

\newcommand{\coEmpty}{\mathrm{coEmpty}}

\newcommand{\cs}{\mathrm{cs}}

\newcommand{\CI}{\mathrm{CI}}

\newcommand{\StrExtFun}{\textnormal{\textbf{StrExtFun}}}

\newcommand{\Emb}{\textnormal{\textbf{Emb}}}

\newcommand{\Inj}{\textnormal{\textbf{Inj}}}

\newcommand{\StrInj}{\textnormal{\textbf{StrInj}}}

\newcommand{\RightInv}{\textnormal{\textbf{RightInv}}}

\newcommand{\LeftInv}{\textnormal{\textbf{LeftInv}}}

\newcommand{\Inv}{\textnormal{\textbf{Inv}}}

\newcommand{\Adj}{\textnormal{\textbf{Adj}}}

\newcommand{\Surj}{\textnormal{\textbf{Surj}}}

\newcommand{\No}[1]{#1^\mathbf{N}}
\newcommand{\NN}[1]{#1^\mathbf{NN}}

\newcommand{\Cont}{\textnormal{\textbf{Cont}}}

\newcommand{\csCont}{\textnormal{\textbf{csCont}}}

\newcommand{\csC}{\textnormal{csC}}

\newcommand{\Sto}{\Rrightarrow}

\newcommand{\Metr}{\textnormal{\textbf{Metr}}}

\newcommand{\pContMod}{\textnormal{\textbf{pContMod}}}

\newcommand{\uContMod}{\textnormal{\textbf{uContMod}}}

\newcommand{\Base}{\textnormal{\texttt{Base}}}

\newcommand{\Op}{\textnormal{Op}}

\newcommand{\csb}{\textnormal{csb}}

\newcommand{\csbTop}{\textnormal{\textbf{csbTop}}}

\newcommand{\csbC}{\textnormal{csbC}}

\newcommand{\csbCont}{\textnormal{\textbf{csbCont}}}

\newcommand{\csbPointCont}{\textnormal{\textbf{csbPointCont}}}

\newcommand{\csbUnifCont}{\textnormal{\textbf{csbUnifCont}}}

\newcommand{\csbPointC}{\textnormal{csbPointC}}

\newcommand{\csbUnifC}{\textnormal{csbUnifC}}

\newcommand{\Si}{\textnormal{Si}}

\begin{document}

\date{}

\title{\textbf{Topologies of open complemented subsets}}
%\titlecomment{{\lsuper*}This paper is a major extension of~\cite{Pe17}.}

\author{Iosif Petrakis\\	%required
Department of Computer Science, University of Verona\\
iosif.petrakis@univr.it}  %optional
%\thanks{thanks 1, optional.}	%optional

% \author[B.~Name2]{Bob Name2}	%optional
% \address{address2; addresses should initially be duplicated, even if
%   authors share an affiliation}	%optional
% \email{name2@email2; ditto for email addresses}  %optional
% \thanks{thanks 2, optional.}	%optional
% 
% \author[C.~Name3]{Carla Name3}	%optional
% \address{address 3}	%optional
% \urladdr{name3@url3\quad\rm{(optionally, a web-page can be specified)}}  %optional
% \thanks{thanks 3, optional.}	%optional

%% etc.

%% required for running head on odd and even pages, use suitable
%% abbreviations in case of long titles and many authors:

%%%%%%%%%%%%%%%%%%%%%%%%%%%%%%%%%%%%%%%%%%%%%%%%%%%%%%%%%%%%%%%%%%%%%%%%%%%

%% the abstract has to PRECEDE the command \maketitle:
%% be sure not to issue the \maketitle command twice!

%\pagecolor{gray}

\maketitle

\begin{abstract}
	\noindent 
	We introduce $\cs$-topologies, or topologies of open complemented subsets, as a new approach to constructive topology that preserves the duality between open and closed subsets of classical topology. Complemented subsets were used successfully by Bishop in his constructive formulation of the Daniell approach to measure and integration. Here we use complemented subsets in topology, in order to describe simultaneously an open set, the first-component of an open complemented subset, together with its given complement as a closed set, the second component of an open complemented subset. We analyse the canonical $\cs$-topology induced by a metric, and we introduce the notion of a modulus of openness for a $\cs$-open subset of a metric space. Pointwise and uniform continuity of functions between metric spaces are formulated with respect to the way these functions inverse open complemented subsets together with their moduli of openness. The addition of moduli of openness in the concept of a complemented open subset, given a base for the $\cs$-topology, makes possible to define the notions of pointwise-like and uniform-like continuity of functions between $\csb$-spaces, that is $\cs$-spaces with a given base. 
	Constructions and facts from the classical theory of topologies of open subsets are translated to the constructive theory of topologies of open complemented subsets. \\[2mm]
\end{abstract}

\section{Introduction}
\label{sec: intro}

Roughly speaking, there are two main ways to constructivise the classical notion of a topological space. The first is to ``mimic'' the definition of a topological space as a starting point, and then to develop topology according to the chosen constructive requirements. Certainly one will use intuitionisitic logic, one may also want to use a point-free approach, or start with a base of a topology, instead of a topology, and study neighborhood spaces. 
The second way is to start from a notion of space other than that of a topological space and generate one a posteriori. E.g., one can start from a Bishop space, a constructive notion of $C(X)$, or from a constructive uniform space. For a comprehensive account of the various approaches to constructive topology see~\cite{Pe23a}. 

In a series of papers (\cite{Pe15}-\cite{Pe19b}, \cite{Pe20b, Pe21}, and~\cite{Pe24, Pe23a}) we have elaborated the theory of Bishop spaces. A  motivation for the study of Bishop spaces, an approach within the second way, was the realisation that the classical duality between open and closed subsets cannot be preserved constructively. The standard, weak complement $F^c$ of a closed subset $F$ need not be open constructively (take e.g., $F := \{0\} \subset \Real$). If one starts from a topology of open subsets and define a closed subset to be the strong complement $F^{\neq_X} := \{x \in X \mid \forall_{y \in F}(x \neq_X y)\}$, where $x \neq_X x{'}$ is an extensional inequality on a set $(X, =_X)$, then the finite intersection of closed subsets may not be provable constructively to be closed.
%(see...Waaldijk, PM23...).  
All approaches to constructive topology within the first way start from the properties of the open subsets, which, in our view, is a ``historical accident''. Classically, one could start equivalently from a topology of closed subsets and define the open subsets as certain complements of closed subsets. If one would start constructively from the properties of closed subsets and define the open subsets as weak or strong complements of closed ones, then some classical properties of open subsets wouldn't hold constructively.
% (does the finite union of opens fail to be open in general???). 
Although classically the two starting points lead to the same theory, constructively this is not the case, but at the same time there is no conceptual reason, other than the power of the classical tradition of beginning with a topology of open subsets, to base constructive topology to open or closed subsets. While within classical logic the difference between open and closed subsets ``collapses'', within constructive logic this is not the case.

Here we explain how to develop constructive topology within the aforementioned first way and preserve the classical duality between open and closed subsets. The distinctive feature of our approach here is that we propose a shift from a topology of open (or closed) subsets to a topology of open \textit{complemented} subsets. Moreover, we reformulate the notion of a point in a complemented subsets-way and define \textit{complemented points} as certain complemented subsets. Topology in classical mathematics $\CLASS$ is usually described as point-set topology. Here we introduce an approach to topology within constructive mathematics $(\CM)$ that can be described as complemented point-complemented set topology.
$$\frac{\mbox{(Point-Set)  Topology}}{\CLASS} \   \sim \    \frac{\mbox{(Complemented  Point-Complemented  Set)  Topology}}{\CM}$$
Bishop's central tool to the development of constructive measure theory was the use of complemented subsets, in order to recover constructively the classical duality between subsets and boolean-valued  total functions, a fundamental duality in the classical Daniell approach to measure theory.
He developed what we call \textit{Bishop Measure Theory} (BMT) in~\cite{Bi67}, and together with Cheng in~\cite{BC72}, what we call \textit{Bishop-Cheng Measure Theory} (BCMT). In both theories complemented subsets are first-class citizens. 
Recently, a programme of a predicative reformulation of BCMT is undertaken in~\cite{Pe20, Ze19, PZ22, Gr22, GP23}. Some connections between 
topology of Bishop spaces and constructive measure theory, especially BMT, are studied in~\cite{Pe19a, Pe20a}.
In order to avoid the use of the principle of the excluded middle (PEM) in the definition of the characteristic function $\chi_A$ of a subset $A$ of $X$, Bishop employed \emph{complemented subsets} \(\boldsymbol{A} \coloneqq (A^1, A^0)\) of \(X\), where $A^1, A^0$ are disjoint subsets of $X$ in a strong sense: $a^1 \neq a^0$, for every $a^1 \in A^1$ and $a^0 \in A^0$. The characteristic function $\chi_{\boldsymbol{A}} \colon A^1 \cup A^0 \to \{0, 1\}$ of a complemented subset \(\boldsymbol{A}\), defined by 
	$$\chi_{\B A}(x) := \left\{ \begin{array}{ll}
	1   &\mbox{, $x \in A^1$}\\
	0             &\mbox{, $x \in A^0$,}
\end{array}
\right. $$
 %$\chi_{\boldsymbol{A}}(a) \coloneqq 1$, if $a \in A^1$, and $\chi_{\boldsymbol{A}}(a) \coloneqq 0$, if $a \in A^0$, 
 is a partial, boolean-valued function on $X$, as it is defined on the subset $A^1 \cup A^0$ of $X$, without though, the use of PEM! The totality of complemented subsets of a set with an inequality $X$ is a \textit{swap algebra}, a generalisation of a Boolean algebra (see Definition~\ref{def: swapalgebra} in this paper), while the totality of boolean-valued, partial functions on $X$ is a \textit{swap ring}, a generalisation of a Boolean ring (see~\cite{PW22, MWP23}). The main category of sets in constructive measure theory is the category\footnote{We denote categories as pairs (\textbf{Objects},\textbf{ Arrows}), as we define later various subcategories of the category $(\Set, \Fun)$ of sets and functions.} $(\SetExtIneq, \StrExtFun)$ of sets with an extensional inequality and strongly extensional functions (see Definition~\ref{def: apartness}). For such sets the extensional empty set $\emptys_X$ of an object $X$ is positively defined. 

Here we introduce the notions of a $\cs$-\textit{topological space} and $\cs$-\textit{continuous function} between them, as constructive counterparts to the notions of a topological space and continuous function between them\footnote{The letters $\cs$ refer to complemented subsets.}. The introduced $\cs$-spaces are sets in $\SetExtIneq$ together with a so-called \textit{topology of open complemented subsets}, or a $\cs$-topology. The main advantage of these concepts is the preservation of the dualities found in classical point-set topology. The category $(\csTop, \csCont)$ of $\cs$-spaces and $\cs$-continuous functions is the constructive counterpart to the category $(\Top, \Cont)$ of topological spaces and continuous functions and a subcategory of $(\SetExtIneq, \StrExtFun)$.

The study of the canonical $\cs$-topology induced by a metric on a set $X$ in section~\ref{sec: csmetric} reveales the importance of ``concept-mining'' in constructive mathematics. By that we mean the accommodation of the definition of many classical concepts, such as the notion of a base for a topology, with ``witnesses'' of the related properties. These witnesses are moduli of continuity, of convergence etc., that instantiate what usually we get classically from the axiom of choice. The notion of an open set in the $\cs$-topology induced by a metric (see Definition~\ref{def: canonicalcs}) comes with a modulus of openness that witnesses the main property of a set being open. This enriched concept of open set in a metric space allows more informative proofs, from the computational point of view, of facts on metric topologies. Moreover, the presentation of pointwise and uniform continuity between metric spaces in section~\ref{sec: csmetriccont}, is strongly influenced by the presence of the moduli of openness associated with open sets. Pointwise continuous functions between metric spaces inverse open sets together with their moduli of openness in a certain ``pointwise'' way, while uniformly continuous functions between metric spaces inverse open sets together with their moduli of openness in a certain ``uniform'' way. Generalising these concepts within the theory of $\cs$-topological spaces with a given base, the so-called here $\csb$-spaces, gives us a way to translate pointwise and uniform continuity of functions between metric spaces to abstract $\csb$-spaces. In this way, the theory of $\cs$-topological spaces, or of $\csb$-spaces, is not only a constructive counterpart to classical topology that reveals its computational content, but also a new theory that is  interesting also from the classical point of view.
Pointwise-like and uniform-like continuity  of functions between topological spaces can be formulated  for standard topologies of open subsets with a base, in a way similar to that found in section~\ref{sec: csbcont}.
The theory of $\cs$-spaces introduced here is another major application of Bishop's deep idea to work with complemented subsets instead of subsets. Moreover, it is a non-trivial example of a theory that has its origin to constructivism and at the same time reveals new concepts, useful to classical mathematics too. In our view, another such example is the theory of swap algebras and swap rings, which is also rooted to Bishop's concept of a complemented subset and generalises in a new way the notion of a Boolean algebra.

We structure this paper as follows:
%\vspace{-3mm}
\begin{itemize}
	
	\item In section~\ref{sec: setineq} we include a brief introduction to $\BST$, the theory of sets that accommodates Bishop-style constructive mathematics $(\BISH)$. We also introduce the basic categories of sets with an extensional inequality and strongly extensional functions that accommodate the theory of $\cs$-topologies. 
	
	\item In section~\ref{sec: sub} we develop the basic theory of (extensional) subsets of a given set with an extensional inequality. Here we restrict our attention to extensional subsets i.e., subsets of a set $X$ defined by separation through an extensional property on $X$. The use of extensional subsets of $X$ only, and not categorical ones (i.e., pairs $(A, i_A)$, where $A$ is a set and $i _A \colon A \eto X$ is an embedding), allows a smoother constructive treatment of the empty subset and a calculus of subsets closer to the classical one. Moreover, we introduce various notions of tightness for an extensional subset of a set with an extensional inequality $(X, =_X, \neq_X)$. Tight subsets are formed by connecting (weak) negation of an equality and an inequality with the given inequality and equality.

	\item In section~\ref{sec: csub} we present the basic theory of (extensional) complemented subsets of a given set  with an extensional inequality. The basic operations on them are defined, and their abstraction, the notion of a swap algebra of type (I), is included.

	\item In section~\ref{sec: cp} we introduce the notion of a complemented point, which is the complemented version of a singleton, and the notion of a canonical complemented point. The latter is the notion of complemented point that suits better to the formulation of local versions of topological notions related to $\cs$-spaces.

\item In section~\ref{sec: csmetric} we define the canonical $\cs$-topology induced by a metric on a set $X$, and we introduce moduli of openness for the open subsets of $X$ with respect to its canonical base of complemented balls. Moduli of openness are going to be generalised to $\cs$-topological spaces with a given base and are crucial to our study of continuity.

\item In section~\ref{sec: csmetriccont} we explain how pointwise and uniformly continuous functions between metric spaces inverse open sets together with their given moduli of openness  (Proposition~\ref{prp: pcont1} and Proposition~\ref{prp: ucont1}, respectively). Conversely, the inversion of opens together with their moduli implies pointwise continuity (Proposition~\ref{prp: pcont2}), and under a certain condition, also uniform continuity (Proposition~\ref{prp: ucont2}). The corresponding constructions are inverse to each other (Proposition~\ref{prp: pcont3} and Proposition~\ref{prp: ucont3}, respectively).

\item In section~\ref{sec: cstop} we define the notions of a $\cs$-topological space, and of  a $\cs$-continous function between them, the arrow in the category $(\csTop, \csCont)$ of $\cs$-topological spaces. The clopen complemented subsets of a $\cs$-space form a swap algebra of type (I) (Corollary~\ref{cor: clopswap}). This fact is the complemented counterpart to the Boolean algebra-structure of the clopen subsets of a standard topological space. The relative $\cs$ topology and the quotient $\cs$-topology are presented. The formulation of the relative $\cs$-topology is not identical to the standard relative topology of open subsets, and the reason behind this slight difference is that a certain distributivity property of complemented subsets (this is property $(D_I)$, given in section~\ref{sec: csub}) cannot be accepted constructively. Instead, Bishop's distributivity property for complemented subsets is used to the proof of Proposition~\ref{prp: relcstop}.

\item In section~\ref{sec: csbtop} we define the notion of a $\cs$-topology induced by a base, in a way that generalises the canonical $\cs$-topology induced by a metric and the base of the complemented balls in section~\ref{sec: csmetric} (Remark~\ref{rem: metricbase}). A base for a $\cs$-topology is a proof-relevant version of the classical notion of a topology of open subsets, as the definitional properties of a base are witnessed by the corresponding base-moduli (Definition~\ref{def: csbase}). We call $\cs$-spaces with a base $\csb$-spaces, and these are the $\cs$-spaces that generalise metric spaces.

\item In section~\ref{sec: csbcont} we translate the notions of pointwise and uniform continuity of functions between metric spaces to pointwise-like and uniformly-like continuous functions between $\csb$-spaces. Uniform continuity between metric spaces has no direct formulation within topological spaces, and one has to define other structures, such as uniform spaces, in order to express uniform continuity. This is one of the reasons why Bishop didn't expect to find a useful constructive counterpart to classical topology (see also section~\ref{sec: concl}). Our constructive reformulation though, of a base for a topology and an open set with respect to that base, which
reveals all necessary witnessing data and avoids the axiom of choice, offers an alternative approach to this problem. We introduce $\csb$-continuous functions i.e., $\cs$-continuous functions that inverse also the corresponding moduli of openness, and we define pointwise-like and uniformly-like $\csb$-continuous functions similarly to pointwise and uniformly continuous functions between metric spaces, respectively (Definition~\ref{def: csbpcont}). 	
	
\item In section~\ref{sec: csbprod} we define the product of two $\csb$-spaces, and we explain how our choice of bases affects the pointwise-like or the uniform-like continuity of the projection-functions (Proposition~\ref{prp:  pcontproj}, and Proposition~\ref{prp:  ucontproj}). We  also define the weak $\csb$-topology and show that it is the least $\csb$-topology that turns all given functions into uniformly-like continuous ones (Proposition~\ref{prp: weak}).

\item In section~\ref{sec: concl} we provide a list of open questions and topics for future work within the theory of $\cs$-topological spaces.

\end{itemize}

We work within the extensional framework of \emph{Bishop Set Theory} (BST), a minimal extension of Bishop's theory of sets\footnote{The relation of BST to Bishop's original theory of sets is discussed in~\cite{Pe20}, section 1.2.} that behaves like a high-level programming language. In the extensional fragment of BST only extensional subsets are considered, which are defined by separation on a given set through an extensional property. Moreover, the totality of extensional subsets of a set, or the extensional powerset, is considered to be an impredicative set\footnote{The categorical subsets of a set $X$ i.e., arbitrary sets with an embedding in $X$, form the categorical powerset of $X$, which is a proper class, as its membership-condition requires quantification over the universe of predicative sets $\X V_0$. The relation between the two powersets involves Myhill's axiom of non-choice and it is studied in~\cite{MWP23}.}, in accordance to the (impredicative) \textit{Bridges-Richman Set Theory} (BRST) in~\cite{BR87, MRR88, BV06}.
The type-theoretic interpretation of Bishop sets as setoids is developed mainly by Palmgren (see e.g.,~\cite{Pa13}-\cite{Pa19}). Other formal systems for $\BISH$ are Myhill's system CST~\cite{My75} and Aczel's system CZF~\cite{AR10}.  Categorical approaches to Bishop sets are found e.g., in the work of Palmgren~\cite{Pa12} and Coquand~\cite{Co17}.
For all notions and results of BST that we use without explanation or proof we refer to~\cite{Pe20, Pe21, Pe22a, Pe24}. For all notions and facts from constructive analysis that we use without explanation or proof, we refer to~\cite{Bi67, BB85, BR87}. For all notions and facts from classical topology that we use without explanation or proof, we refer to~\cite{Du66}. In a proof of a theorem we may write $(\INT)$ to indicate that intuitionisitic logic i.e., Ex Falso quodlibet, is needed. If in a proof the Ex Falso quodlibet is not used, then our proof is within minimal logic $(\MIN)$ (see~\cite{SW12}).  If a proof employs Myhill's unique choice principle, a very weak form of choice, this is also indicated accordingly.

\section{Sets with an extensional inequality and strongly extensional functions}
\label{sec: setineq}

Bishop Set Theory (BST) is an informal, constructive theory of \emph{totalities} and \emph{assignment routines} between totalities that accommodates Bishop's informal system BISH and serves as an intermediate step between Bishop's original theory of sets and an adequate and faithful formalisation of BISH in Feferman's sense~\cite{Fe79}. Totalities, apart from the basic, undefined set of natural numbers $\mathbb{N}$, are defined through a membership-condition. The \emph{universe} $\mathbb{V}_0$ of predicative sets is an open-ended totality, which is not considered a set itself, and every totality the membership-condition of which involves quantification over the universe is not considered a set, but a proper class. \emph{Sets} are totalities the membership-condition of which does not involve quantification over $\mathbb{V}_0$, and are equipped with an equality relation i.e., an equivalence relation. An equality relation $x =_X x{'}$ on a defined set $X$ is defined through a formula $E_X(x, x{'})$ (see Definition~\ref{def: formulas}). Assignment routines are of two kinds: \emph{non-dependent} ones and \emph{dependent} ones. 

\begin{definition}\label{def: function}
If $(X, =_X)$ and $(Y, =_Y)$ are sets, a function $\fXY$ is a non-dependent assignment routine $f \colon X \sto Y$ i.e., for every $x \in X$ we have that $f(x) \in Y$, such that $x =_X x{'} \To f(x) =_Y f(x{'})$, for every $x, x{'} \in X$. A function \(f:X\to Y\) is an embedding, if for all \(x,x'\) in \(X\) we have that \(f(x) =_Yf(x')\) implies \(x =_X x'\), and we write \(f:X \eto Y\), while it is a surjection, if $\forall_{y \in Y}\exists_{x \in X}(f(x) =_Y y )$. We denote by $\X F(X, Y)$ the set of functions from $X$ to $Y$, and by $\id_X$ the identity function on $X$. Let $(\Set, \Fun)$ be the category of sets and functions, $(\Set, \Emb)$ its subcategory with embeddings and $(\Set, \Surj)$ its subcategory with surjections. Two sets are equal in $\mathbb{V}_0$ if there is $g \colon Y \to X$, such that\footnote{With this equality the universe in $\BST$ can be called \emph{univalent} in the sense of Homotopy Type Theory~\cite{Ri22, Ho13}, as, by definition, an equivalence between sets is an equality. Similarly, the type-theoretic function-extensionality axiom is incorporated in BST as the canonical equality of the function space.} $g \circ f =_{\X F(X, X)} \id_X$ and $f\circ g =_{\X F(Y, Y)} \id_Y$.
\end{definition}

In $\BST$ a set is usually equipped not only with a given equality, but also with a given inequality relation, 
$x \neq_X x{'}$ defined through a formula $I_X(x, x{'})$ (see Definition~\ref{def: formulas}). 
The primitive set of natural numbers $\Nat$ has a given equality $=_{\Nat}$ and inequality $\neq_{\Nat}$.
We consider $s \neq_{\Nat} t$ as primitive, rather than defined, as one could, through weak negation, as we want to avoid the use of weak negation in the foundations of BISH. For every basic or defined set in $\BST$ we have variables of that set. Next we define formulas within $\BST$. In the unpublished work~\cite{Bi69}  of Bishop formulas of $\BISH$ are translated in his fofrmal system $\Sigma$ in almost exactly the same way.
%$0 \neq_{\Nat} 1$ in a primitive way.

\begin{definition}[Formulas in $\BST$]\label{def: formulas}
	Prime formulas: $s =_{\Nat} t$, $s \neq_{\Nat} t$, where $s, t \in \Nat$.\\[1mm] 
	Complex formulas: If $A, B$ are formulas, then $A \vee B, A \wedge B, A \To B$ are 
	formulas, and if $S$ is a set and $\phi(x)$ is a  formula, for every variable $x$ of set $S$,
	then $\exists_{x \in S} \big( \phi(x)\big)$ and $\forall_{x \in S} \big(\phi(x)\big)$ are formulas. 
	Let $\top := 0 \neq_{\Nat} 1, \bot := 0 =_{\Nat} 1$ and $\neg A := A \To \bot$. 
\end{definition}

Constructively, $\neg A$ has a weak behaviour, 
as the negation of a conjunction  
$\neg(A \wedge B)$ does not generally imply the disjunction $\neg A \vee \neg B$, and similarly the negation of
a universal
formula $\neg\forall_x A(x)$ does not generally imply the existential formula $\exists_x \neg A(x)$.
Even if $A$ and  $B$ are stable i.e., $\neg \neg A \to A$ and $\neg \neg B \to B$, we only get  
$\neg(A \wedge B) \vdash \neg \neg (\neg A \vee \neg B)$ and $\neg\forall_x A(x) 
\vdash \neg \neg \exists_x \neg A(x)$, where $\vdash$ is the derivability relation in minimal logic. 
For this reason, $\neg A$ is called the \textit{weak} negation of $A$, which, according to Rasiowa~\cite{Ra74}, p.~276, 
is not constructive, exactly due to its aforementioned behaviour. The critique of weak negation in constructive mathematics $(\CM)$ goes back to Griss~\cite{Gr46}. Despite the use of weak negation in Brouwer's intuitionistic mathematics and Bishop's system $\BISH$, both, Brouwer and Bishop, developed a positive and strong approach to many classically negatively defined concepts.
Recently, Shulman~\cite{Sh22} showed
that most of the strong concepts of $\BISH$ ``arise automatically from an ``antithesis'' translation
of affine logic into intuitionistic logic via a ``Chu/Dialectica construction''.
Motivated by this work of Shulman, and in relation to our reconstruction of the theory of sets underlying 
$\BISH$ (see~\cite{Pe20, Pe21, Pe22a} and~\cite{PW22, MWP23}),
we develop in~\cite{Pe23b} a treatment of strong negation within
$\BST$, arriving in a heuristic method for the definition of strong concepts 
in $\BISH$, similar to Shulman's. For that, the equivalences occurring in the axiom schemata related to strong negation in formal theories of strong negation
become the definitional clauses of the recursive definition of $\No{A}$, with respect to the above
inductive definition of formulas $A$ in $\BST$. 
In~\cite{KP23}, strong negation is incorporated in the theory of computable functionals TCF, extending the following recursive definition of strong negation of formulas to inductively and coinductively defined predicates.

\begin{definition}[Strong negation in $\BST$]\label{def: strongneg}
	Strong negation $\No{A}$ of a formula $A$ in $\BST$ is defined by the following clauses:\\[1mm]
	$(\SN_1)$ For prime formulas related to the primitive set $(\Nat, =_{\Nat}, \neq_{\Nat})$ let
	%	and for every $s, t \in \Nat$ we define
	$$\No{\big(s =_{\Nat} t\big)} := s\neq_{\Nat} t \ \ \& \ \ \No{\big(s \neq_{\Nat} t\big)} := s =_{\Nat} t.$$
	If $A, B$ are complex formulas in $\BST$, let\\[1mm]
	$(\SN_2) \ \ \ \ \ \ \ \ \ \ \ \ \ \ \ \ \ \ \ \ \ \ \ \ \ \ \ \ \ \ \ \ \ \ \ \ \ 
	\No{(A \vee B)} := \No{A} \wedge \No{B}$,\\[1mm]
	$(\SN_3) \ \ \ \ \ \ \ \ \ \ \ \ \ \ \ \ \ \ \ \ \ \ \ \ \ \ \ \ \ \ \ \ \ \ \ \ \ 
	\No{(A \wedge B)} :=  \No{A} \vee \No{B}$,\\[1mm]
	$(\SN_4) \ \ \ \ \ \ \ \ \ \ \ \ \ \ \ \ \ \ \ \ \ \ \ \ \ \ \ \ \ \ \ \ \ \ \ \ \ 
	\No{(A \To B)} := A \wedge \No{B}$.\\[1mm]
	If $S$ is a set and $\phi(x)$ is a formula in $\BST$, where $x$ is a  varriable of $S$, let
	$$(\SN_5) \ \ \ \ \ \ \ \ \ \ \ \ \ \ \ \ \ \ \ \ \ \ \ \ \ \ \ \ \ \ \ \ \ \ \ \ \ 
	\No{\big(\exists_{x \in S} \phi(x)\big)} := \forall_{x \in S}\big(\No{\phi(x)}\big), \ \ \ \ \ \ \ \ \ \ \ \ \ 
	\ \ \ \ \ \ \ \ \ \ \ \ \ \ \ \ \ \ \ \ \ \ \ \ \ \ \ \ \ \ \ $$
	$$(\SN_6) \ \ \ \ \ \ \ \ \ \ \ \ \ \ \ \ \ \ \ \ \ \ \ \ \ \ \ \ \ \ \ \ \ \ \ \ \ \  
	\No{\big(\forall_{x \in S} \phi(x)\big)} := \exists_{x \in S}\big(\No{\phi(x)}\big). \ \ \ \ \ \ \ 
	\ \ \ \ \ \ \ \ \ \ \ \ \ \ \ \ \ \ \ \ \ \ \ \ \ \ \ \ \ \ \ \ \ \ \ \ \ \ $$
	We also define $\NN{A}  := \No{(\No{A})}$. If $A, B$ are formulas in $\BST$, their strong implication\footnote{Strong implication is defined by Rasiowa in~\cite{Ra74} within a formal system of constructive logic with strong negation.} is defined by
	$$ A \Sto B := (A \To B) \wedge (\No{B} \To \No{A}).$$
\end{definition}

 %which is a tight apartness relation, a concept defined next.
%on $\Nat$.

\begin{definition}\label{def: apartness}
	If $(X, =_X)$ is a set, let the following formulas with respect to a relation $x \neq_X y:$\\[1mm]
	$(\Ineq_1) \  \forall_{x, y \in X}\big(x =_X y \ \& \ x \neq_X y \To \bot \big)$,\\[1mm]
	$(\Ineq_2) \ \forall_{x, x{'}, y, y{'} \in X}\big(x =_X x{'} \ \& \ 
	y =_X y{'} \ \& \ x \neq_X y \To x{'} \neq_X y{'}\big)$,\\[1mm]
	$(\Ineq_3) \ \forall_{x, y \in X}\big(\neg(x \neq_X y) \To x =_X y\big)$,\\[1mm]
	$(\Ineq_4)  \forall_{x, y \in X}\big(x \neq_X y \To y \neq_X x\big)$,\\[1mm]
	$(\Ineq_5) \ \forall_{x, y \in X}\big(x \neq_X y \To \forall_{z \in X}(z \neq_X x \ \vee \ z \neq_X y)\big)$,\\[1mm]
	$(\Ineq_6) \ \forall_{x, y \in X}\big(x =_X y \vee x \neq_X y\big)$.\\[1mm]
	If $\Ineq_1$ is satisfied, we call $\neq_X$ an 
	%$($abstract$)$ 
	inequality on $X$, and the 
	structure $\C X := (X, =_X, \neq_X)$ a set with an inequality. 
	%We also write $|\B X| := X$.
	If $(\Ineq_6)$ is satisfied, then $\C X$ is 
	%set with inequality $(X,=_X, \neq_X)$ is 
	\textit{discrete}, if $(\Ineq_2)$ holds, $\neq_X$ is 
	called \textit{extensional}, and  if $(\Ineq_3)$ holds, it is \textit{tight}. If it satisfies $(\Ineq_4)$ and $(\Ineq_5)$, it is called an \textit{apartness relation} on $X$.
	\end{definition}
	
	The primitive inequality $\neq_{\Nat}$ is a discrete, tight apartness relation. In many cases,
	 %(see~\cite{Pe23d}), 
	 the inequality on a set is equivalent to the strong negation of its equality and vice versa.
	
	\begin{definition}\label{def: setineq}
	If $\C X := (X, =_X, \neq_X)$ and $\C Y := (Y, =_Y, \neq_Y)$ are sets with inequality, a function $f \colon X \to Y$ is \textit{strongly extensional},
	if $f(x) \neq_Y f(x{'}) \To x \neq_X x{'}$, for every $x, x{'} \in X$, and it is an injection, if the converse implication (also\footnote{Here we consider injections that are already strongly extensional functions.}) holds i.e., $x \neq_X x{'} \To f(x) \neq_Y f(x{'})$, for every $x, x{'} \in X$. A \textit{strong injection} is an injection that is also an embedding\footnote{Classically, the negation of equality $\neg(x =_X x{'})$ is used instead of $x \neq_X x{'}$, and the corresponding implications are equivalent with the use of double negation elimination. Constructively, we need to treat these notions separately. Notice though, that each (constructive) implication follows from the other by using contraposition with respect to \textit{strong negation}, if the inequality on $X$ is the strong negation of its equality and vice versa.} i.e., $x \neq_X x{'} \To f(x) \neq_Y f(x{'})$ and $f(x) =_Y f(x{'}) \To x =_X x{'}$, for every $x, x{'} \in X$. 
	Let $(\SetIneq, \StrExtFun)$ be the category of sets with an inequality and strongly extensional functions, let	$(\SetExtIneq, \StrExtFun)$ be its subcategory with sets with an extensional inequality, and $(\SetIneq, \Inj)$ and $(\SetExtIneq, \Inj)$ the corresponding subcategories with $($strongly extensional$)$ injections.
	Moreover, let $(\SetIneq, \StrInj)$ and $(\SetExtIneq, \StrInj)$ be the corresponding subcategories with strong injections.
	\end{definition}

	\begin{definition}\label{def: canonicalineq}
	Let $\C X := (X, =_X, \neq_X)$ and $\C Y := (Y, =_Y, \neq_Y)$ be sets with inequalities.
	The canonical inequalities on the product $X \times Y$ and the function space
	%\footnote{If $X, Y$ are totalities 
	%	an \textit{assignment routine} $f$
	%	from $X$ to $Y$ is denoted by $f \colon X \sto Y$. If $X, Y$ are sets, $f \colon X \sto Y$
	%	is a \textit{function}, if it respects their equalities. A function is an embedding, if it is an injection.}
	$\X F(X, Y)$ are given, respectively, by
	$$(x, y) \neq_{X \times Y} (x{'}, y{'}) :\TOT x \neq_X x{'} \vee y \neq_Y y{'},$$
	$$f \neq_{\X F(X, Y)} g :\TOT \exists_{x \in X} \big[f(x) \neq_Y g(x)\big].$$
	\end{definition}
	% -inequality structure of a
	% set\index{equality-inequality structure of a set} $X$.
%\end{definition}
Clearly, the canonical inequalities on $X \times Y$ and $\X F(X, Y)$ are formed by the strong negation of their canonical equalities, respectively.
The projections $\pr_X, \pr_Y$ associated to $X \times Y$ 
are strongly extensional functions.
It is not easy to give examples of non strongly extensional functions, although we cannot accept
in $\BISH$ that all functions are strongly extensional. E.g., the strong extensionality of 
all functions from a metric space to itself is equivalent to Markov's principle (see~\cite{Di20}, p.~40).
Even to show that a constant function between sets with an inequality is strongly extensional, one needs 
intuitionistic,  and not minimal, logic. An apartness relation $\neq_X$ on a set $(X, =_X)$ is always
extensional (see~\cite{Pe20}, Remark 2.2.6, p.~11).
%There are two notions of subsets in $\BST$, categorical and extensional subsets, the equality of which dependes on Myhill's axiom of non-choice (see~\cite{Pe20, MWP23}). 

\begin{definition}\label{def: canonical}
	The canonical inequality on the set of reals $\Real$ is given by 
	$a \neq_{\Real} b :\TOT |a - b| > 0 \TOT a > b \vee a < b$, which is  
	a special case of the canonical inequality on a metric space $(Z, d)$, given by 
	$z \neq_{(Z,d)} z{'} :\TOT d(z, z{'}) > 0$. 
	If $(X, =_X)$ is a set and $F \subseteq \X F(X)$, the set of real-valued functions on $X$,
	the inequality on $X$ induced by $F$ is defined by
	$x \neq_{(X, F)} x{'} :\TOT \exists_{f \in F}\big(f(x) \neq_{\Real} f(x{'})\big).$
\end{definition}

The inequalities $a \neq_{\Real} b$ and $z \neq_{(Z,d)} z{'}$ are tight apartness relations. E.g., cotransitivity of $z \neq_{(Z,d)} z{'}$ is shown by the fact that $d(z, z{''}) + d(z{''}, z{'}) \geq d(z, z{'}) > 0$, and if $a + b > 0$ in $\Real$, then $a > 0$ or $b > 0$ (see~\cite{BB85}, p.~26). The inequality 
$x \neq_{(X, F)} y$ is an apartness relation, which is tight if and only if 
%we have that 
$$\forall_{f \in F}\big(f(x) =_{\Real} f(x{'})\big) \To x =_X x{'},$$
for every $x, x{'} \in X$.
No negation of some sort is used in the definition of the inequality $x \neq_{(X, F)} x{'}$, in 
the proof of its extensionality, while the last implication above can be seen as completely positive formulation of its tightness. 
If $F$ is a Bishop topology of functions 
(see~\cite{Pe15,Pe20b,Pe21,Pe22b, Pe23a}), then $x \neq_{(X, F)} x{'}$ is the canonical inequality of a Bishop space.
The inequality $a \neq_{\Real} b$ is equivalent to such an inequality induced by functions, as
$a \neq_{\Real} b \TOT a \neq_{(\Real, \Bic(\Real))} b$,
where $\Bic(\Real)$ is the topology of Bishop-continuous functions of type $\Real \to \Real$ 
(see~\cite{Pe15}, Proposition 5.1.2.). Notice that all Bishop-continuous functions of type $\Real \to \Real$, actually all pointwise continuous functions of this type, are
strongly extensional (see~\cite{Pe15}, 2.3.12(i), p.~26), a fact that is generalised in Remark~\ref{rem: secont}.
By its extensionality,
%Using Proposition~\ref{prp: apartness1}, 
an inequality induced by functions
provides a (fully) positive definition of the 
extensional empty subset, and of the extensional complement
%\footnote{Notice that it is not clear how to define a 
	%	complement of a categorical subset as a categorical subset, and an arbitrary categorical complemented subset
	%	determines such a complement through its $0$-component.} 
$X_P^{\neq_X}$ of an extensional subset $X_P$ of $X$ (see Definition~\ref{def: extempty} and Definition~\ref{def: extunion}). The following facts are used in section~\ref{sec: cp} and are straightforward to prove.

\begin{proposition}\label{prp: leftinv}
	Let $\C X, \C Y, \fXY$ in $(\SetExtIneq, \StrExtFun)$ with a left inverse $g_l \colon Y \to X$ 
	%i.e., the following diagram commutes
	\begin{center}
		\begin{tikzpicture}
			
			\node (E) at (0,0) {$Y$};
			\node[right=of E] (X) {$X$};
			\node[right=of X] (Y) {$Y$.};

			\draw[->] (E)--(X) node [midway,above] {$g_l$};
			\draw[->] (X)--(Y) node [midway,above] {$f $};
			\draw[->,bend right] (E) to node [midway,below] {$\id_Y$} (Y);
			
		\end{tikzpicture}
	\end{center}
\normalfont (i)
\itshape 
$f$ is a surjection,
% i.e., $\forall_{y \in Y}\exists_{x \in X}(f(x) =_Y y)$, 
and $g_l$ is a strong injection.\\[1mm]
\normalfont (ii)
\itshape  If $f$ is an injection, then $g_l$ is strongly extensional.\\[1mm]
\normalfont (iii)
\itshape If $g_l$ is strongly extensional and surjective, then $f$ is an injection.
	
\end{proposition}

%\begin{proof}
%	(i) 
%	
%	
%\end{proof}

\begin{proposition}\label{prp: rightinv}
	Let $\C X, \C Y, \fXY$ in $(\SetExtIneq, \StrExtFun)$ with a right inverse $g_r \colon Y \to X$ 
	%i.e., the following diagram commutes
	\begin{center}
		\begin{tikzpicture}
			
			\node (E) at (0,0) {$X$};
			\node[right=of E] (X) {$Y$};
			\node[right=of X] (Y) {$X$.};

			\draw[->] (E)--(X) node [midway,above] {$f$};
			\draw[->] (X)--(Y) node [midway,above] {$g_r $};
			\draw[->,bend right] (E) to node [midway,below] {$\id_X$} (Y);
			
		\end{tikzpicture}
	\end{center}
	\normalfont (i)
	\itshape 
	$g_r$ is a surjection, and if $f$ is a surjection, then $g_r$ is an embedding.\\[1mm]
	\normalfont (ii)
	\itshape If $g_r$ is strongly extensional, then $f$ is an injection.\\[1mm]
	\normalfont (iii)
	\itshape If $f$ is a surjection and an injection, then $g_r$ is strongly extensional.
	
\end{proposition}

If a function $f$ has a left $g_l$ and a right inverse $g_r$, then it is immediate to show that $g_l =_{\X F(Y, X)} g_r$.
% and hence without loss of generality we assume that they are identical.

\begin{proposition}\label{prp: inv}
	Let $\C X, \C Y, \fXY$ in $(\SetExtIneq, \StrExtFun)$ with an inverse $g \colon Y \to X$ 
	%i.e., the following diagram commutes
	\begin{center}
		\begin{tikzpicture}
			
			\node (E) at (0,0) {$Y$};
			\node[right=of E] (X) {$X$};
			\node[right=of X] (Y) {$Y$};
			\node[right=of Y] (Z) {$X$.};
			
			\draw[->] (E)--(X) node [midway,above] {$g$};
			\draw[->] (X)--(Y) node [midway,above] {$f $};
			\draw[->] (Y)--(Z) node [midway,above] {$g$};
			\draw[->,bend right] (E) to node [midway,below] {$\id_Y$} (Y);
			\draw[->,bend left=40] (X) to node [midway,above] {$\id_X$} (Z);
			
		\end{tikzpicture}
	\end{center}
	\normalfont (i)
	\itshape 
	$f$ is a surjection and $g$ is a surjective strong injection.\\[1mm]
	\normalfont (ii)
	\itshape  $f$ is an injection if and only if $g$ is strongly extensional.
\end{proposition}

\begin{proposition}\label{prp: adj}
	Let $\C X, \C Y, \fXY$ in $(\SetExtIneq, \Fun)$ with a left inequality-adjoint $g \colon Y \to X$, in symbols $g \dashv f$ i.e., for every $x \in X$ and $y \in Y$ the following condition is satisfied:
	$$g(y) \neq_X x \TOT y \neq_Y f(x).$$
		\normalfont (i)
	\itshape 
	If $f$ is a surjection and an injection, then $g$ is an injection, and, dually, if $g$ is a surjection and an injection, then $f$ is an injection.\\[1mm]
	\normalfont (ii)
	\itshape  If $f$ is a strongly extensional surjection, then  $g$ is strongly extensional, and, dually, if $g$ is a strongly extensional surjection, then $f$ is strongly extensional.
\end{proposition}

Clearly, if $\fXY$ has a left (right) inverse one of the two implications of the adjointness-condition holds, and if $f$ has an inverse $g$, then $g \dashv f$. We can organise the above notions of functions into certain subcategories of $(\SetIneq, \Fun)$.

\begin{definition}\label{def: invcats}
Let $(\Set, \LeftInv)$ be the category of sets with a left inverse, $(\Set, \RightInv)$ the category of sets with functions with a right inverse, and $(\Set, \Inv)$ the category of sets with invertible functions. The categories 
$(\SetIneq, \LeftInv)$, $(\SetIneq, \RightInv)$, $(\SetIneq, \Inv)$, and their subcategories $(\SetExtIneq, \LeftInv)$, $(\SetExtIneq, \RightInv)$, and $(\SetExtIneq, \Inv)$ are defined accordingly. Let $(\SetIneq, \Adj)$ be the category of sets with inequality and strongly extensional functions with a left inequality-adjoint. The category $(\SetExtIneq, \Adj)$ is defined similarly. 
\end{definition}

\section{Extensional subsets}
\label{sec: sub}

Here we work only with extensional subsets, as in Bridges-Richman Set Theory. The use of extensional subsets only allows a smoother treatment of the (extensional) empty subset of a set, and hence a calculus of subsets closer to that of subsets within classical logic. The role of Myhill's axiom of non-choice to the relation between extensional and categorical subsets is studied in~\cite{MWP23}.

\begin{definition}[Extensional subsets]\label{def: extsubset}
	A formula $P(x)$, where $x$ is a variable of set $(X, =_X)$,
	is an \textit{extensional property}\footnote{Notice that in $\MLTT$ a property on a 
		type $X$ is described as a type family $P \colon X \to \C U$ over $X$, where $\C U$ is a universe and $X \colon \C U$, and it
		is always extensional through the corresponding transport-term (see~\cite{Ho13}, section 2.3).}
	on $X$\index{extensional property on a set}, if 
	$$\forall_{x, y \in X}\big([x =_{X } y \ \& \ P(x)] \To P(y)\big).$$
	The corresponding totality $X_P$\index{$P_X$} defined by the separation scheme is called an \textit{extensional subset} of $X$, if 
	its equality is inherited from $=_{X}$. If $\neq_X$ is a given inequality on $X$, then the canonical inequality
	$\neq_{X_P}$ on $X_P$ is inherited from $\neq_{X}$. The totality of extensional subsets of $X$ is denoted by $\C E(X)$.
	If $X_P, X_Q \in \C E(X)$, let $X_P \subseteq X_Q :\TOT \forall_{x \in X}\big(P(x) \To Q(x)\big)$ and $X_P =_{\C E(X)} X_Q :\TOT
	X_P \subseteq X_Q \ \& \ X_Q \subseteq X_P$. We may avoid writing the extensional property $P(x)$ determining an extensional
	subset $A$ of $X$, and in this case we simply write $A \subseteq X$. 
%	We define an inequality on $\C E (X)$ as follows:
%	$$X_P \neq_{\mathsmaller{\C E(X)}} X_Q :\TOT \exists_{x \in X}\big((P(x) \wedge \neg {Q(x)}) \vee (Q(x) \wedge \neg{P(x)})\big).$$
\end{definition}

Notice that no quantification over the universe $\X V_0$ is required in the membership condition of 
$\C E(X)$. In accordance to the fundamental distinction between \textit{judgments} and \textit{propositions} due to Martin-L\"of, 
if $A, B \in \C E(X)$, the inclusion $A \subseteq B$ is a formula that does
not employ the membership symbol 
i.e., we do not say that for all $x \in X$ if $x \in A$, then $x \in B$, as such a definition would interpret the judgments
$x \in A, x \in B$ as formulas. If $A := X_P$ and $B := X_Q$, we only need to prove $P(x) \To Q(x)$, for every $x \in X$. 
There are non-extensional properties on a set i.e., formulas $P(x)$, where $x$ is a variable of the set $X$, 
such that $P(x)$ and $x =_X y$ do not imply $P(y)$. 
For example, let $n \in \Nat$, $q \in \X Q$ and $P_q(x)$, where $x$ is a variable of set $\Real$, defined by
$P_q(x) := x_n =_{\X Q} q.$ If $y \in \Real$ such that $y =_{\Real} x$, then it is not necessary
that $y_n =_{\X Q} q$, if $ x_n =_{\X Q} q$.
A set $X$ is an extensional subset of itself, in the sense that it is equal
%\footnote{Two (predicative) 
%	sets $X, Y$ are equal in $\X V_0$ if 
%	there is a bijection (or better a pair of bijections)
%	between them, and in this case we write $X =_{\X V_0} Y$. The bijection between $X$ and $X_{=_{X}}$ is
%	defined by the identity rule.} 
in $\X V_0$ to the extensional subset 
$X_{=_{X}} := \{x \in X \mid x =_X x\}.$
The sets $\one$ and $\two$ are defined as extensional subsets 
of the primitive set  $\Nat$:
$$\one := \{x \in \Nat \mid x =_{\Nat} 0\} =: \{0\}, \ \ \ \ \ 
\two := \{x \in \Nat \mid x =_{\Nat} 0 \ \vee x =_{\Nat} 1\} =: \{0, 1\}$$
As an equality $=_X$ is extensional on $X \times X$, the diagonal $D(X, =_X) := \{(x, y) \in X \times X \mid x =_X y\}$
of $X$ is an extensional subset of $X \times X$. In the following definition of the extensional empty subset of a set $X$ we do not require $X$ to be inhabited, as Bishop does in~\cite{Bi67}, p.~65, for the categorical empty subset. We only require that $X$ is equipped with an extensional inequality.

\begin{definition}[Extensional empty subset]\label{def: extempty}
	If $\C X := (X, =_{X}, \neq_{X})$ in in $\SetExtIneq$, then 
	$$\emptys_X := \{x \in X \mid x \neq_X x\}$$
	 is the extensional empty subset of $\C X$. 
	We call	$X_P \in \C E(X)$ \textit{empty}, if $X_P \subseteq \emptys_X$.
\end{definition}

Notice that within $\INT$ the inclusion $\ \emptys_X \subseteq X_P$ holds, for every extensional property $P(x)$ on $X$. The extensionality of an equality $=_X$ implies that the inequality $\neg(x =_X y)$ on $(X, =_X)$ is extensional, with associated \textit{weak}
empty subset the extensional subset 
$$\{x \in X \mid \neg(x =_X x)\}.$$ Although this definition employs negation, we can define inequalities on a set, as we saw in Definition~\ref{def: apartness}, and hence associated empty subsets, in a completely positive way. 
%Through this property we can define a stronger inequality on $\C E (X)$ as follows:
%$$X_P \stackrel{s} \neq_{\mathsmaller{\C E(X)}} X_Q :\TOT \exists_{x \in X}\big((P(x) \wedge [Q^{\neq_X}](x)) \vee (Q(x) \wedge [P^{\neq_X}](x))\big).$$
%
%\begin{remark}\label{rem: emptywhole2}
%	If $(X, a, =_X, \neq_X)$ is an inhabited set with an inequality satisfying $(\Ineq_5)$, then $\ \emptys_X \stackrel{s} \neq_{\mathsmaller{\C E(X)}} X$.
%\end{remark}	
%
%\begin{proof}
%	Clearly, $a =_X a$ and if $x \neq_X x$, then by cotransitivity we get $a \neq_X x$.
%\end{proof}	
Next we define the basic operations on extensional subsets\footnote{In~\cite{MWP23} a different notation is used for these operations, in order to be distinguished from the corresponding operations on categorical subsets. As here we employ only extensional subsets, we use for them the standard notation for the opeartions on subsets.}.

\begin{definition}\label{def: extunion}
	If $(X, =_X), (Y, =_Y) \in \Set$, $A := X_P, B := X_Q \in \C E(X)$, $C := Y_R \in \C E(Y)$ and $f \colon X \to Y$, let the following operations on extensional subsets:
	$$A \cup B := \{x \in X \mid P(x) \vee Q(x)\}, \ \ \ \ \ A \cap B := \{x \in X \mid P(x) \wedge Q(x)\},$$
	$$f(A) := \big\{y \in Y \mid \exists_{x \in X_P}\big(f(x) =_Y y\big)\big\}, \ \ \ \ \ 
	f^{-1}(C) := \big\{x \in X \mid \exists_{y \in Y_R}\big(f(x) =_Y y\big)\big\},$$
	$$A \times C := \big\{((x{'}, y{'}) \in X \times Y \mid \exists_{x \in X_P}\exists_{y \in Y_R}\big(x{'} =_X x 
	\wedge y{'} =_Y y\big)\big\}.$$
	%$$X_P^{\neq_X} := \big\{x \in X \mid \forall_{y \in X_P}\big(x \neq_X y)\big)\big\}.$$
	If $(I, =_I)$ is a set, a family of extensional subsets of $X$ indexed by $I$ is an assignment routine 
	$\lambda_0 : I \sto \X V_0$, such that for every $i, j \in I$ we have that $\lambda_0(i) := X_{P_i}$, where
	$P_i(x)$ is an extensional property on $X$, and if $i =_I j$, then $X_{P_i} =_{\C E(X)} X_{P_j}$. For simplicity we 
	denote such a family by $\big(X_{P_i}\big)_{i \in I}$. Its union and intersection are given by
	%defined, respectively, as follows:
	$$\bigcup_{i \in I}\lambda_0(i) := \big\{x \in X \mid \exists_{i \in I}\big(P_i(x)\big)\big\}, \ \ \ \ \ 
	\bigcap_{i \in I}\lambda_0(i) := \big\{x \in X \mid \forall_{i \in I}\big(P_i(x)\big)\big\}.$$
Alternatively, if $F(X_S)$ is an extensional property on $\C E(X)$ and
	$\C F := \{X_S \in \C E(X) \mid F(X_S)\},$
	let
	$$\bigcup \C F := \big\{x \in X \mid \exists_{X_S \in \C F}\big(S(x)\big)\big\}, \ \ \ \ \ 
	\bigcap \C F := \big\{x \in X \mid \forall_{X_S \in \C F}\big(S(x)\big)\big\}.$$
	If $\neq_X$ is an extensional inequality on $X$, let the extensional complement $A^{\neq_X}$ of $A$, defined by
	$$A^{\neq_X} := \big\{x \in X \mid \forall_{y \in X_P}\big(x \neq_X y)\big)\big\}.$$
\end{definition}

Notice that the extensionality of $P(x), Q(x)$ and $R(y)$, together with the preservation 
of equalities by a function, imply the extensionality of the formulas 
$$(P \vee Q)(x) := P(x) \vee Q(x), \ \ \ \ \ 
(P \wedge Q)(x) := P(x) \wedge Q(x),$$
$$[f[P]](y) := \exists_{x \in X_P}\big(f(x) =_Y y\big), \ \ \ \ \ 
[f^{-1}[R]](x) := \exists_{y \in Y_R}\big(f(x) =_Y y\big),$$
$$[P \times R](x, y) := \exists_{x{'} \in X_P}\exists_{y{'} \in Y_R}\big(x{'} =_X x 
\wedge y{'} =_Y y\big).$$
The extensionality of $P_i(x)$, for every $i \in I$, implies the extensionality of the formulas 
$$\bigg(\bigcup_{i \in I}P_i\bigg)(x) := \exists_{i \in I}\big(P_i(x)\big), \ \ \ \ \ 
\bigg(\bigcap_{i \in I}P_i\bigg)(x) := \forall_{i \in I}\big(P_i(x)\big).$$
Similalrly, the extensionality of $S$, for every $X_S \in \C F$, implies the extensionality of the formulas
$$\bigg(\bigcup \C F\bigg)(x) := \exists_{X_S \in \C F}\big(S(x)\big),   \ \ \ \ \ 
\bigg(\bigcap \C F \bigg)(x) := \forall_{X_S \in \C F}\big(S(x)\big).$$
Clearly, the extensionality of $\neq_X$ implies the extensionality of the property 
$$[P^{\neq_X}](x) := \forall_{y \in X_P}\big(x \neq_X y)\big).$$
The usual properties of the inverse and direct image of subsets hold. Next we mention only a few.
%, and for the rest see~\cite{Pe20}, section 2.6.

\begin{proposition}\label{prp: extempty1}
	Let $\C X, \C Y, \C Z$ in $\SetExtIneq$, $A, B \subseteq X$, $C, D \subseteq Y$, $\fXY$, and $g \colon Y \to Z$.\\[1mm]
	\normalfont (i) 
	\itshape If $f$ is strongly extensional, then $f^{-1}(\emptys_Y) =_{\C E(X)} \emptys_X$. \\[1mm]
	\normalfont (ii) 
	\itshape If $f$ is an injection, then $f(\emptys_X) =_{\C E(Y)} \emptys_Y$.\\[1mm]
	\normalfont (iii) 
	\itshape $f^{-1}(C \cup D) =_{\C E(X)} f^{-1}(C) \cup f^{-1}(D)$.\\[1mm]
	\normalfont (iv) 
	\itshape $f^{-1}(C \cap D) =_{\C E(X)} f^{-1}(C) \cap f^{-1}(D)$.\\[1mm]
	\normalfont (v)
	\itshape If $ {\C H}$ is an extensional subset of $\C E(Y)$ and $ f^{-1}({\C H}) := \{f^{-1}(H) \mid H \in {\C H}\}$, then 
	$$f^{-1}\bigg(\bigcup {\C H}\bigg) =_{\mathsmaller{\C E (X)}} \bigcup f^{-1}({\C H}), \ \ \ \ 
	f^{-1}\bigg(\bigcap {\C H}\bigg) =_{\mathsmaller{\C E (X)}} \bigcap f^{-1}({\C H}).$$
	\normalfont (vi) 
	\itshape $f(A \cup B) =_{\C E(Y)} f(A) \cup f(B)$.\\[1mm]
	\normalfont (vii) 
	\itshape $f(A \cap B) \subseteq f(A) \cap f(B)$.\\[1mm]
\normalfont (viii)
\itshape If $ {\C G}$ is an extensional subset of $\C E(X)$ and $ f({\C G}) := \{f(G) \mid G \in {\C G}\}$, then 
$$f\bigg(\bigcup {\C G}\bigg) =_{\mathsmaller{\C E (Y)}} \bigcup f({\C G}), \ \ \ \ 
f\bigg(\bigcap {\C G}\bigg) \subseteq \bigcap f({\C G}).$$
	\normalfont (ix) 
	\itshape $A \subseteq f^{-1}(f(A))$ and $f(f^{-1}(C)) \subseteq C$.\\[1mm]
	\normalfont (x) 
	\itshape $f(f^{-1}(C) \cap A) =_{\C E(Y)} C \cap f(A)$, and $f(f^{-1}(C)) =_{\C E(Y)} C \cap f(X)$.\\[1mm]
	\normalfont (xi) 
	\itshape  $A \times (C \cup D) =_{\C E(X \times Y)} (A \times C) \cup (A \times D)$.\\[1mm]
	\normalfont (xii) 
	\itshape  $A \times (C \cap D) =_{\C E(X \times Y)} (A \times C) \cap (A \cap D)$.\\[1mm]
	\normalfont (xiii) 
	\itshape  $A \times (C - D) =_{\C E(X \times Y)} (A \times C) - (A \cap D)$.\\[1mm]
		\normalfont (xiv) 
	\itshape $A =_{\C E(X)} f^{-1}(f(A))$, for every $A \in \C E(X)$, if and only if $f$ is an embedding.\\[1mm]
		\normalfont (xv) 
	\itshape $C =_{\C E(Y)} f(f^{-1}(C))$, for every $C \in \C E(Y)$,  if and only if $f$ is a surjection.\\[1mm]
		\normalfont (xvi) 
	\itshape $f(A \cap B) =_{\C E(Y)} f(A) \cap f(B)$, for every $A, B \in \C E(X)$, if and only if $f$ is an embedding.\\[1mm]
	\normalfont (xvii) 
	\itshape If $f$ is an embedding and $ {\C G}$ is an extensional subset of $\C E(X)$, then 
	$f\big(\bigcap {\C G}\big) =_{\C E(Y)} \bigcap f({\C G})$.\\[1mm]
		\normalfont (xviii) 
	\itshape If $H \subseteq Z$, then $(g \circ f)^{-1}(H) =_{\C E(X)} f^{-1}(g^{-1}(H))$.\\[1mm]
		\normalfont (xix) 
	\itshape If $(g \circ f)(A) =_{\C E(Z)} g(f(A))$.
	\end{proposition}

\begin{proof}
	We show only (i) and (ii), and the rest follow in a straightforward way.\\
	(i) Clearly, within $\INT$ we have that $\emptys_X \subseteq f^{-1}(\emptys_Y) $. Actually, if $f$ is an injection this inclusion is shown trivially within $\MIN$. The converse inclusion is shown within $\MIN$ as follows: if $x \in f^{-1}(\emptys_Y)$, let $y \in Y$ such that $y \neq_Y y$ and $f(x) =_Y y$. By the extensionality of $\neq_Y$ we get $f(x) \neq_Y f(x)$, and by strong extensionality of $f$ we conclude that $x \neq_X x$ i.e., $x \in \emptys_X$.\\
	(ii) Within $\INT$ we have that $\emptys_Y \subseteq f(\emptys_X)$. Actually, if $f$ is only a strongly extensional surjection, then this inclusion is shown easily within $\MIN$. The converse inclusion is shown within $\MIN$ as follows: by definition $f(\emptys_X) := \{y \in Y \mid \exists_{x \in \  {\emptys_X}}(f(x) =_Y y)\}$. Let $y \in Y$ and $x \in \emptys_X$, such that $f(x) =_Y y$. As $x \neq_X x$ and $f$ is an injection, we get $f(x) \neq_Y f(x)$, and by the extensionality of $\neq_Y$ we get $y \neq_Y y$.
\end{proof}

The inclusions in Proposition~\ref{prp: extempty2}(i, ii) hold trivially within $\INT$, but under certain conditions on $\neq_X$ they are also provable within $\MIN$.

\begin{proposition}\label{prp: extempty2}
	Let $\C X, \C Y$ in $\SetExtIneq$, $A, B, A_i \subseteq X$, for every $i \in I$, $C, D \subseteq Y$, and $\fXY$.\\[1mm]
	\normalfont (i) 
	\itshape $(\MIN)$ $X^{\neq_X} \subseteq \emptys_X$, and if $\neq_X$ is an apartness relation, or discrete,  then $\emptys_X \subseteq X^{\neq_X}$. \\[1mm]
	\normalfont (ii) 
	\itshape $(\MIN)$ $\emptys_X^{\neq_X} \subseteq X$, and if $\neq_X$ is an apartness relation, or discrete,  then $X \subseteq \emptys_X^{\neq_X}$.\\[1mm]
	\normalfont (iii) 
	\itshape $\big(\bigcup_{i \in I}A_i\big)^{\neq_X} =_{\C E(X)} \bigcap_{i \in I}A_i^{\neq_X}$. \\[1mm]
	\normalfont (iv) 
	\itshape $\big(\bigcap_{i \in I}A_i\big)^{\neq_X} \supseteq_{\C E(X)} \bigcup_{i \in I}A_i^{\neq_X}$.\\[1mm]
	\normalfont (v)
	\itshape $A \subseteq B^{\neq_X} \TOT B \subseteq A^{\neq_X}$.\\[1mm]
	\normalfont (vi) 
	\itshape $A \subseteq (A^{\neq_X})^{\neq_X}$ and $A^{\neq_X} =_{\C E(X)} [(A^{\neq_X})^{\neq_X}]^{\neq_X}$.\\[1mm]
	\normalfont (vii) 
	\itshape If $f$ is strongly extensional, then $f^{-1}\big(C^{\neq_X}\big) \subseteq [f^{-1}(C)]^{\neq_X}$.\\[1mm]
	\normalfont (viii)
	\itshape If $f$ is an injection and a surjection, then $[f^{-1}(C)]^{\neq_X} \subseteq f^{-1}\big(C^{\neq_X}\big)$.\\[1mm]
	\normalfont (ix) 
	\itshape If $f$ is an injection, then $f\big(A^{\neq_X}\big) \subseteq (f(A))^{\neq_Y}$.\\[1mm]
	\normalfont (x) 
	\itshape If $f$ is a strongly extensional surjection, then $(f(A))^{\neq_Y} \subseteq f\big(A^{\neq_X}\big)$.
	%\\[1mm]
%	\normalfont (xi) 
%	\itshape  $A \times (C \cup D) =_{\C E(X \times Y)} (A \times C) \cup (A \times D)$.\\[1mm]
%	\normalfont (xii) 
%	\itshape  $A \times (C \cap D) =_{\C E(X \times Y)} (A \times C) \cap (A \cap D)$.\\[1mm]
%	\normalfont (xiii) 
%	\itshape  $A \times (C - D) =_{\C E(X \times Y)} (A \times C) - (A \cap D)$.\\[1mm]
%	\normalfont (xiv) 
%	\itshape $A =_{\C E(X)} f^{-1}(f(A))$, for every $A \in \C E(X)$, if and only if $f$ is an embedding.\\[1mm]
%	\normalfont (xv) 
%	\itshape $C =_{\C E(Y)} f(f^{-1}(C))$, for every $C \in \C E(Y)$,  if and only if $f$ is a surjection.\\[1mm]
%	\normalfont (xvi) 
%	\itshape $f(A \cap B) =_{\C E(Y)} f(A) \cap f(B)$, for every $A, B \in \C E(X)$, if and only if $f$ is an embedding.\\[1mm]
%	\normalfont (xvii) 
%	\itshape If $f$ is an embedding and $ {\C G}$ is an extensional subset of $\C E(X)$, then 
%	$f\big(\bigcap {\C G}\big) =_{\C E(Y)} \bigcap f({\C G})$.\\[1mm]
%	\normalfont (xviii) 
%	\itshape If $H \subseteq Z$, then $(g \circ f)^{-1}(H) =_{\C E(X)} f^{-1}(g^{-1}(H))$.\\[1mm]
%	\normalfont (xix) 
%	\itshape If $(g \circ f)(A) =_{\C E(Z)} g(f(A))$.
\end{proposition}

\begin{proof}
	We show only (i), and for (ii) we work similarly. The rest cases follow in a straightforward way.
	If $x{'} \in X^{\neq_X} := \{x{'} \in X \mid \forall_{x \in X}(x{'} \neq_X x)\}$, then $x{'} \neq_X x{'}$ too i.e., $x{'} \in \emptys_X$. To show $\emptys_X \subseteq X^{\neq_X}$ under the apartness assumption for $\neq_X$, let $x{'} \in X$ with $x{'} \neq_X x{'}$. If $x \in X$, by cotransitivity$(\Ineq_5)$ we get $x{'} \neq_X x$ or $x \neq_X x{'}$. If $\neq_X$ is discrete, then it suffices to treat the case $x{'} =_X x$. By the extensionality of $\neq_X$ and the hypothesis $x{'} \neq_X x{'}$ we conclude that $x{'} \neq_X x$.
\end{proof}

%
%
%
%\section{Tight subsets}
%\label{sec: tight}

Next we introduce various notions of tight (extensional) subsets of a set with an extensional inequality $(X, =_X, \neq_X)$. Clearly, the \textit{logical} or \textit{weak complement} of an extensional subset $A$ of $X$
$$A^c := \{x \in X \mid \forall_{a \in A}(\neg(x =_X a))\}$$
includes the $\neq_X$-complement $A^{\neq_X}$. If the converse inclusion holds, we call $A$ $1$-tight.

\begin{definition}
	If $\C X := (X, =_X, \neq_X) \in \SetExtIneq$, and $A \subseteq X$, we call $A$:\\[1mm]
	\normalfont (a)
	\itshape $1$-tight, if $A^c \subseteq A^{\neq_X}$,\\[1mm]
	\normalfont (b)
	\itshape $0$-tight, if $(A^{\neq_X})^c \subseteq A$,\\[1mm]
	\normalfont (c)
\itshape $\neq_X$-stable, if $A^{\neq_X \neq_X} \subseteq A$,\\[1mm]
	\normalfont (d)
\itshape $\neg$-stable, if $A^{c c} \subseteq A$,\\[1mm]
	\normalfont (d)
\itshape left-cotight, if $A \subseteq (A^c)^{\neq_X}$,\\[1mm]
	\normalfont (e)
\itshape right-cotight, if $(A^c)^{\neq_X} \subseteq A$.
\end{definition}

The term tight subset is explained as follows. If $x \neq_X y := \No{(x=_X y)}$, then tightness of $\neq_X$ (condition $(\Ineq_3)$ of Definition~\ref{def: apartness}) takes the form 
$$\neg\big[\No{(x=_X y)}\big] \To x =_X y,$$
which is an instance of the implication 
$$\neg\big[\No{A}\big] \To A.$$
We call the formulas of $\BST$ that satisfy the above implication \textit{tight} (see also~\cite{KP23}). There are formulas of $\BST$ for which we cannot show to be tight. For example, let the implication $\neg(x \leq y) \To x > y$, where $x, y \in \Real$, which is equivalent to Markov's Principle $(\MP)$, and hence not provable in $\BISH$ (see~\cite{BR87}). As we explain in Remark~\ref{rem: order1} we can show in $(\MIN)$ that $x \leq y \TOT \No{(x > y)}$, and hence the above implication in $\Real$ is written as 
$$\neg\big(\No{(x > y)}\big) \To x > y.$$ 
Consequently, the tightness of the formula corresponding to $x > y$ cannot be shown in $\BISH$.
The dual implication $\No{(\neg A)} \To A$ holds, as by Definition~\ref{def: strongneg} we have that $\No{(\neg A)} := \No{(A \to \bot)} := A \wedge \top \TOT A$. Now, as $A^{\neq_X} := \{x \in X \mid \forall_{a \in A}(x \neq_X a)\},$ and as 
$$A =_{\C E(X)} \{x \in X \mid \exists_{a \in A}(x =_X a)\},$$
the inclusion $A^c \subseteq A^{\neq_X}$ implies that for every $x \in A^c$ i.e., $\neg( \exists_{a \in A}(x =_X a))$, we have that $x \in A^{\neq_X}$ i.e., $\forall_{a \in A}(x \neq_X a)$, the strong negation 
of which is $\exists_{a \in A}\No{(x \neq_X a)}$. If we have that $\No{(x \neq_X a)} \To x =_X a$, a form of strong tightness, then the implication $x \in A^c \To x \in A^{\neq_X}$ expresses the following tightness 
$$\neg{\big[\No{\big(\forall_{a \in A}(x \neq_X a)\big)}\big]} \To \forall_{a \in A}(x \neq_X a).$$
A similar explanation can be offered for $0$-tightness as follows. The inclusion $(A^{\neq_X})^c \subseteq A$ implies for every $x \in X$ the implication $\neg{(x \in A^{\neq_X})} \To x \in A$ i.e., $\neg{(\forall_{a \in A}(x \neq_X a)} \To x \in A$. Again, as $x \in A \TOT \exists_{a \in A}(x =_X a)$, $0$-tightness of $A$ takes the form 
$$\neg{\big[\No{(x \in A)}\big]} \To x \in A.$$

The following facts are straightforward to show.

\begin{proposition}\label{prp: tight1}
	Let $\C X \in \SetExtIneq$, and $A, B, A_i$ subsets of $X$, for every $i \in I$.\\[1mm]
		\normalfont (i)
	\itshape $\emptys_X$ is $1$-tight.\\[1mm]
	\normalfont (ii) 
	\itshape If $\neq_X$ is discrete, then $X$ is $1$-tight.\\[1mm]
		\normalfont (iii)
	\itshape If $A_i$ is $1$-tight, for every $i \in I$, then $\bigcup_{i \in I}A_i$ is $1$-tight.\\[1mm]
		\normalfont (iv)
	\itshape If $A$ is $0$-tight, then $A$ is $\neg$-stable.\\[1mm]
		\normalfont (v)
	\itshape If $A$ is $1$-tight and $\neg$-stable, then $A$ is $0$-tight.\\[1mm]
		\normalfont (vi)
	\itshape If $A$ is $\neq_X$-stable, then $A$ is left-cotight.\\[1mm]
	\normalfont (vii)
	\itshape If $A$ is right-cotight and $1$-tight, then $A$ is $\neq_X$-stable.\\[1mm]
	\normalfont (viii)
	\itshape If $A$ is $1$-tight, then $A$ is left-cotight.

\end{proposition}

%\begin{proof}
%(i)
%	
%\end{proof}

Notice that te intersection of $1$-tight sets need not be $1$-tight. The open ball $[d_x < \epsilon] := \{y \in X \mid d(x, y) < \epsilon\}$ of a metric space $(X, d)$ is $1$-tight, but, in general, it is not $0$-tight.
%Inverse and direct image of tight subsets...??????
A plethora of $1$-tight subsets of a metric space is provided by Remark~\ref{rem: ball1}(ii).

\section{Complemented subsets}
\label{sec: csub}

Complemented subsets are pairs of positively disjoint subsets of a set with an inequality. As the definition of the empty subset requires that $X$ is equipped with an extensional inequality, we work within the category $\SetExtIneq$, although many related concepts are definable also within $\SetIneq$. 
%In this section, unless otherwise stated, 

\begin{definition}\label{def: extapartsubsets}
	If $\C X := (X, =_X, \neq_X)$ is in $\SetExtIneq$ and 
	$A := X_P, B := X_Q \in \C E(X)$, then the overlapping relation $A \between B$ is defined by
	$$A \between B :\TOT \exists_{x \in X_P}\exists_{x{'} \in X_Q}\big(x =_X x{'}\big).$$
	Strongly negating the above notion, we say that $A$ and $B$ are disjoint with respect to $\neq_X$,
	%in symbols $A \Disj_{\neq_X} B$, or simpler 	$A \Disj B$, 
	if
	$$A \Disj B :\TOT \forall_{x \in X_P}\forall_{x{'} \in X_Q}\big(x \neq_X x{'}\big).$$
	A \textit{complemented subset}\index{complemented subset} of $\C X$ is a pair
	$\B A := (A^1, A^0)$\index{$\B A := (A^1, A^0)$}, 
	where $A^1 := X_{P_1}, A^0 := X_{P_0} \in \C E(X)$ such that $A^1 \Disj A^0$. Its characteristic function %$\chi_{\B A}$ 	is the function 
	$\chi_{\B A} \colon \dom(\B A) := \{x \in X \mid P_1(x) \vee P_0(x)\} \to \two$ is defined by 
	$$\chi_{\B A}(x) := \left\{ \begin{array}{ll}
		1   &\mbox{, $P_1(x)$}\\
		0             &\mbox{, $P_0(x)$.}
	\end{array}
	\right. $$
	If $\B A, \B B$ are complemented extensional subsets of $X$, let
	$\B A \subseteq \B B : \TOT A^1 \subseteq B^1 \ \& \ B^0 \subseteq A^0.$ We say that $\B A$ is inhabited, if $A^1$ is inhabited, coinhabited, if $A^0$ is inhabited, and total, if $\dom(\B A) =_{\C E(X)} X$. We call $\B A$ $1$-tight, if for every $x \in X$, we have that 
	$\neg{(x \in A^1)} \To x \in A^0$, and we call $\B A$ $0$-tight, if for every $x \in X$, we have that 
	$\neg{(x \in A^0)} \To x \in A^1$. 	
	Let $\C E^{\Disj}(X)$\index{$\C P^{\Disj}(X)$}\index{$\B A \subseteq \B B$} be their totality\footnote{The totality $\C E^{\Disj}(X)$ can also be defined by separation as follows:
		$\C E^{\Disj}(X) := \{(A^1, A^0) \in \C E (X) \times \C E (X) \mid A^1 \Disj A^0\}$. The relation $R(A^1, A^0) := A^1 \Disj A^0$ is extensional on $\C E (X) \times \C E (X)$, and as $\C E (X)$ is an 
		impredicative set, $\C E (X) \times \C E(X)$ and $\C E^{\Disj}(X)$ are also impredicative sets.}, 
	equipped with the equality 
	$\B A =_{\C E^{\mathsmaller{\Disj}} (X)} \B B : \TOT \B A \subseteq \B B \ \& \ \B B \subseteq \B A$.  
\end{definition}

If $A := X_P \in \C E(X)$, then $\B {\mu} A := (A, A^{\neq_X})$ is in $\C E^{\Disj}(X)$, and 
%Moreover, notice that
if $\B A$ is $1$-tight, then the inclusions
$$(A^1)^c \subseteq A^0 \subseteq (A^1)^{\neq_X} \subseteq (A^1)^c $$
show that the $1$-tightness of $\B A$ is equivalent to the $1$-tightness of $A^1$.

\begin{proposition}\label{prp: isvoid}
	Let $\C X := (X, =_X, \neq_X)$ be in $\SetExtIneq$ and 
	$A := X_P, B := X_Q \in \C E(X)$.\\[1mm]
	\normalfont (i) 
	\itshape 
	$(\MIN)$ If $A \Disj B$, then $A \cap B \subseteq \emptys_X$, and $(\INT)$ $\emptys_X \subseteq A \cap B$. \\[1mm]
	\normalfont (ii) 
	\itshape
	$(\MIN)$ If $\neq_X$ is discrete and $A \cap B \subseteq \emptys_X$, then $A \Disj B $.\\[1mm]
	\normalfont (iii) 
	\itshape $(\MIN)$ If $\neq_X$ is cotransitive $(\Ineq_5)$ and symmetric $(\Ineq_4)$, then $\emptys_X \Disj A$.\\[1mm]
		\normalfont (iv) 
	\itshape $(\INT)$ $\emptys_X \Disj A$.
\end{proposition}

\begin{proof}
	(i) By the definition of $A \Disj B$ we get $\forall_{x \in X}\big(P(x) \wedge Q(x) \To x \neq_X x\big)$, which is the required inclusion. The converse inclusion holds trivisally by Ex Falso.\\
	(ii) If $P(a)$ and $Q(b)$, for some $a, b \in X$, then if $a =_X b$, by extensionality of $Q(x)$ we get $Q(a)$. Since
	than $P(a) \wedge Q(a)$, we get $a \neq_X a$. Hence, by discreteness of $\neq_X$, we conclude that $a \neq_X b$.\\
	(iii) If $x \in \emptys_X$, then $x \neq_X x$, and by cotransitivity $x \neq_X y$ or $y \neq_X x$, where $y \in A$. By symmetry of $\neq_X$ we get in both cases the required inequality $x \neq_X y$.\\
	(iv) If $x \in \emptys_X$, then $x \neq_X x$, and as $x =_X x$ by $(\Ineq_1)$ we get $\bot$ and by Ex Falso we conclude $x \neq_X y$, where $y \in A$.
\end{proof}

All operations between extensional complemented subsets that appear in this section
are determined by the corresponding definitions of operations between subsets in the previous section.
Here we shall use only the operations of the so-called in~\cite{PW22} \textit{first algebra of complemented subsets} $\C B_1^e(X) := \big(\C E^{\Disj}(X), \cup, \cap, -)$ of $X$, as the second algebra $\C B_2^e(X) := \big(\C E^{\Disj}(X), \vee, \wedge, -)$ does not suit well to topology (it suits well though, to Bishop-Cheng Measure Theory).
% we have 

\begin{definition}\label{def: opers}
Let $\C X := (X, =_X, \neq_X), \C Y := (Y, =_Y, \neq_Y) \in \SetIneq$.
If $\B A, \B B \in \C E^{\Disj}(X)$, $(\B \lambda(i))_{i \in I}$ is a family in $\C E^{\Disj}(X)$ indexed by some set $I$, and if $\B C \in \C E^{\Disj}(Y)$ , let 
$$\B A \cup \B B  := (A^1 \cup B^1, A^0 \cap B^0), \ \ \ \ 
\bigcup_{i \in I} \B \lambda (i) := \bigg(\bigcup_{i \in I}\lambda^1(i), \bigcap_{i \in I}\lambda^0(i)\bigg),
$$ 
$$\B A \cap \B B := (A^1 \cap B^1, A^0 \cup B^0), \ \ \ \ \ \  \bigcap_{i \in I} \B \lambda (i) := 
\bigg(\bigcap_{i \in I}\lambda^1(i), \bigcup_{i \in I}\lambda^0(i)\bigg),$$
$$ \ \ \ - \B A := (A^0, A^1), \ \ \ \ \ \ \B A - \B B := \B A \cap (-\B B),$$
$$ \B A \times \B C := \big(A^1 \times C^1, \ [A^0 \times Y] \cup [X \times C^0]\big).$$
\end{definition}

With the above notion of complement of a complemented subset we recover constructively properties of classical subsets that are based on classical logic. Namely, we have that 
$$\B A \subseteq \B B \TOT -\B B \subseteq - \B A, \ \ \ \ -(-\B A) := \B A, \ \ \ \ (X, \emptys_X) - \B A =_{\C E^{\Disj}(X)} \B A.$$
%Gather all basic properties: of a swap algebra .....
%It is immediate that 
If $\B \lambda(0) := \B A$ and $\B \lambda(1) := \B B$ defines a family of complemented subsets
over $\two$, then 
$$\bigcup_{i \in \two} \B \lambda (i) =_{\mathsmaller{\C E^{\Eisj}(X)}} \B A \cup \B B \ \ \ \& \ \ \ \bigcap_{i \in \two} \B \lambda (i) =_{\mathsmaller{\C E^{\Eisj}(X)}} \B A \cap \B B.$$
Alternatively, we can define the union and the intersection of a set of complemented subsets as follows.

\begin{definition}\label{def: union}
If $F(X_{P^1}, X_{P^0})$ is an extensional property on $\C E^{\Disj}(X)$, let 
$$\B {\C F} := \{(X_{P^1}, X_{P^0}) \in \C E^{\Disj}(X) \mid F(X_{P^1}, X_{P^0})\},$$
$$\C F^1 := \big\{ X_P \in \C E(X) \mid \exists_{X_Q \in \C E(X)}\big(F(X_P, X_Q)\big)\big\},$$
$$\C F^0 := \big\{ X_Q \in \C E(X) \mid \exists_{X_P \in \C E(X)}\big(F(X_P, X_Q)\big)\big\}.$$
We define
$$\bigcup \B {\C F} := \bigg(\bigcup \C F^1, \bigcap \C F^0\bigg), \ \ \ \ \ \bigcap \B {\C F} := \bigg(\bigcap \C F^1, \bigcup \C F^0\bigg).$$

\end{definition}

Notice that the extensionality of $F$ implies the extensionality of the properties
$$F^1(X_P) := \exists_{X_Q \in \C E(X)}\big(F(X_P, X_Q)\big), \ \ \ \ \ 
F^0(X_Q) := \exists_{X_P \in \C E(X)}\big(F(X_P, X_Q)\big).$$

The following facts are straightforward to show and are used later in the paper.

\begin{corollary}\label{cor: csoperations1}

	\normalfont (i)
	\itshape If $\B A \subseteq \B B$ and $\B A \subseteq \B C$, then $\B A \subseteq \B B \cap \B C$.\\[1mm]
	\normalfont (ii)
	\itshape If $\B B \subseteq \B A_i$, then $\B B \subseteq \bigcup_{i \in I}\B A_i$.
\end{corollary}

In contrast to the standard one-dimensional subset theory, where only one empty subset $\emptys_X$ of a set $X$ exists, in the two dimensional complemented subset theory there are many ``empty''
complemenetd subsets and equally many ``coempty'' subsets. The proof of subsequent Proposition~\ref{prp: isswapa1} is straightforward.

\begin{definition}\label{def: zeros}
If $\C X := (X, =_X, \neq_X)$ is in $\SetExtIneq$, and if $\B A \in \C E^{\Disj}(X)$, let 
$$0_{\B X} := ( \emptys_X, X), \ \ \ \ 1_{\B X} := (X,  \emptys_X), \ \ \ \ 
1_{\B A} :=  (\dom(\B A), \emptys_X), \ \ \ \  0_{\B A} := \big(\emptys_X, \dom(\B A)\big).$$
The complemented subsets $0_{\B X} $ and of the form $0_{\B A}$ are called empty, while the 
complemented subsets $1_{\B X} $ and of the form $1_{\B A}$ are called coempty. Let their corresponding sets
$$\EEmpty(\C X) := \{\B A := (A^1, A^0) \in \C E^{\Disj}(X) \mid A^1 =_{\C E(X)} \emptys_X\},$$
$$\coEmpty(\C X) := \{\B A := (A^1, A^0) \in \C E^{\Disj}(X) \mid A^0 =_{\C E(X)} \emptys_X\}.$$
\end{definition}

\begin{proposition}[$\INT$]\label{prp: isswapa1} The following hold:\\[1mm]
	%$\F B_1^e(X) := \big(\C E^{\Eisj}(X), \sqcup, \sqcap, -)$
	\normalfont(i)
	%\itshape 
	$0_{\B X}$ and $1_{\B X}$ are the bottom and top elements of $\C E^{\Disj}(X)$, respectively.\\[1mm]
	\normalfont(ii)
	%\itshape
	$-0_{\B X} := 1_{\B X}$ and $-0_{\B A} := 1_{\B A}$.\\[1mm]
	\normalfont(iii)
	%\itshape
	$\B A \cup (- \B A) =_{\mathsmaller{\C E^{\Disj}(X)}} 1_{\B A}$ and 
	$\B A \cap (- \B A) =_{\mathsmaller{\C E^{\Disj}(X)}} 0_{\B A}$.\\[1mm]
	\normalfont(iv)
	%\itshape
	$0_{\B X} \cup \B A =_{\mathsmaller{\C E^{\Disj}(X)}} \B A =_{\mathsmaller{\C E^{\Disj}(X)}} 
	0_{\B A} \cup \B A$.\\[1mm]
	\normalfont(v)
	%\itshape
	$0_{\B X} \cap \B A =_{\mathsmaller{\C E^{\Disj}(X)}} 0_{\B X}$ and
	$0_{\B A} \cap \B A =_{\mathsmaller{\C E^{\Disj}(X)}} 0_{\B A}$.\\[1mm]
	\normalfont(vi)
	%\itshape
	$1_{\B X} \cup \B A =_{\mathsmaller{\C E^{\Disj}(X)}} 1_{\B X}$ and $1_{\B A} \cup
	\B A =_{\mathsmaller{\C E^{\Disj}(X)}} \B 1_A$.\\[1mm]
	\normalfont(vii)
	%\itshape  
	$1_{\B X} \cap \B A =_{\mathsmaller{\C E^{\Disj}(X)}} \B A =_{\mathsmaller{\C E^{\Disj}(X)}} 1_{\B A} \cap 	\B A$.\\[1mm]
	\normalfont(viii)
	%\itshape 
	$0_{-\B A} =_{\mathsmaller{\C E^{\Disj}(X)}} 0_{\B A}$.\\[1mm]
	\normalfont(ix)
	%\itshape
	$0_{0_{\B X}} =_{\mathsmaller{\C E^{\Disj}(X)}} 0_{\B X}$ and
	$1_{1_{\B X}} =_{\mathsmaller{\C E^{\Disj}(X)}} 1_{\B X}$.\\[1mm]
	\normalfont(x)
	%\itshape
	$0_{0_{\B A}} =_{\mathsmaller{\C E^{\Disj}(X)}} 0_{\B A} =_{\mathsmaller{\C E^{\Disj}(X)}} 0_{1_{\B A}}$.
	%\\[1mm]
%	\normalfont(xi)
	%\itshape
%	If $X$ is inhabited, then $0_{\B X} \neq_{\mathsmaller{\C E^{\Disj}(X)}} 1_{\B X}$.
\end{proposition}

Proposition~\ref{prp: isswapa1}(iii) shows that the algebra of complemented subsets cannot be, in geberal boolean. Actually, the totality of complemented subsets of some $\C X := (X, =_X, \neq_X)$ in $\SetExtIneq$ is a \textit{swap algebra} of type (I) (see~\cite{MWP23} for all necessary details). The abstract version of the common properties of the two algebras of complemented subsets $\C B_1^e(X)$ and $\C B_2^e(X)$ forms the notion of a swap algebra, while the axiom $(\swapa_{\ti})$ is satisfied by $\C B_1^e(X)$ and the axiom $(\swapa_{\tii})$ is satisfied by $\C B_2^e(X)$ (see Definition~\ref{def: swapalgebra}). A swap algebra of type (I) does not satisfy, in general, condition $(\swapa_{\tii})$, and a swap algebra of type (II) does not satisfy, in general, condition $(\swapa_{\tii})$. A swap agebra is a generalisation of a Boolean algebra, as every Boolean algebra is a swap algebra. In this case the two types of swap algebras collapse to the structure of a Boolean algebra. 
Moreover, the total elements of a swap algebra from a Boolean algebra. Swap algebras of type (II) are better related to boolean-valued partial functions rather than swap algebras of type (I) (see~\cite{MWP23}). The abstract version of the boolean-valued partial functions on a set with an extensional inequality is the notion of the so-called \textit{swap ring}. The theory of swap algebras and swap rings, which is under current development (see also~\cite{MWP24b} for a classical treatment of the Stone representation theorem for swap rings and~\cite{MWP25} for a constructive treatment of a Stone representation theorem for separated swap algebras of type $(\tii)$), is expected to be a generalisation of the theory of Boolean algebras and Boolean rings that originates from constructive mathematics. Their theory is one of the few examples 
%in the history of modern mathematics 
of mathematical theories that have non-trivial applications to classical mathematics and are motivated by constructive mathematics!

\begin{definition}\label{def: swapalgebra}
	A \textit{swap algebra} is a structure $\C A := \big(A, \vee, \wedge, 0, 1, -, 0_{-}, 1_{-}\big)$,
	where $(A, =_A, \neq_A)$ is a set with an inequality, $0, 1 \in A$, $\vee, \wedge, \colon A \times A \to A$ and $0_{-}, 1_{-} \colon A \to A$
	%$[-,-]$
	are functions, such that the following conditions hold:\\[1mm]
	$(\swapa_1)$ \ $a \vee a =_{\mathsmaller{A}} a$.\\[1mm]
	$(\swapa_2)$ \ $a \vee b =_{\mathsmaller{A}} b \vee a$.\\[1mm]
	$(\swapa_3)$ \ $a \vee (b  \vee c) =_{\mathsmaller{A}} (a \vee b) \vee c$.\\[1mm]
	$(\swapa_4)$ \ $a \vee (b  \wedge c) =_{\mathsmaller{A}} (a \vee b) \wedge (a \vee c)$.\\[1mm]
	$(\swapa_5)$ $0 \neq_{\mathsmaller{A}} 1$.\\[1mm]
	$(\swapa_5)$ \ $-0 =_{\mathsmaller{A}} 1$.\\[1mm]
	$(\swapa_7)$ \ $-(-a) =_{\mathsmaller{A}} a$.\\[1mm]
	$(\swapa_8)$ \ $-(a \vee b) =_{\mathsmaller{A}} (-a) \wedge (-b)$.\\[1mm]
	$(\swapa_9)$ \ $a \vee (-a) =_{\mathsmaller{A}} 1_a$ and $a \wedge (-a) =_{\mathsmaller{A}} 0_a$.\\[1mm]
	$(\swapa_{10})$  \ $a  =_{\mathsmaller{A}} 0_{a} \vee a$.\\[1mm]
	A swap algebra $\C A$ is \textit{of type} $($I$)$, if the following condition holds:\\[1mm]
	$(\swapa_{\ti})$ \ $(a \vee b) \wedge a =_{\mathsmaller{A}} a$.\\[1mm]
	%and $(a \wedge b) \vee a =_{\mathsmaller{A}} a$\\[1mm]
	A swap algebra  $\C A$ is \textit{of type} $($II$)$, if the following condition holds:\\[1mm]
	$(\swapa_{\tii})$ \ $0_a \vee b =_{\mathsmaller{A}} 1_a \wedge b$.\\[1mm]
	An element $a $ of $A$ is called \textit{total}, if $1_a =_{\mathsmaller{A}} 1$.
	%, and it is called \textit{inhabited}, if $a \neq_{\mathsmaller{A}}  0_a$. 
	%Let the (subclasses) subsets
	%$$T(A) := \{ x \in A \mid a \ \mbox{is total}\},$$
	%$$Z(A) := \big\{ x \in A \mid \exists_{y \in A}\big(x = _{\mathsmaller{A}} %0_y\big)\big\}, \ \ \ \ \ 
	%O(A) := \big\{ x \in A \mid \exists_{y \in A}\big(x = _{\mathsmaller{A}} %1_y\big)\big\}.$$
\end{definition}
%$0_{\B f_P} \cdot  \B g_Q =_{\mathsmaller{\C F(X, \due)}} \B 0_{P \wedge Q}$

The complemented subsets studied here form a swap algebra of type (I). To show $0_{\B X} \neq_{\mathsmaller{\C E^{\Disj}(X)}} 1_{\B X}$, we need to fix the notion of inequality between complemented subsets.
% The deinition used in~\cite{MWP23} allows to show this inequality for an inhabited set $X$ with an inequality. 
Using Definition~\ref{def: csineq} that involves complemented points, introduced in section 5, the inequality $0_{\B X} \neq_{\mathsmaller{\C E^{\Disj}(X)}} 1_{\B X}$ follows easily for an inhabited set with an inequality (Corollary~\ref{cor: csineq1}(i)) . 
Next follow some properties of the \textit{complemented inverse image} of a strongly extensional function.

\begin{proposition}\label{prp: finvproperties}
	Let $\C X, \C Y, \C Z,  \fXY, g \colon Y \to Z \in (\SetExtIneq, \StrExtFun)$.\\[1mm]
%	\normalfont (i) 
%	\itshape 
%		$f^{-1}(- \B B) := - f^{-1}(\B B)$.\\[1mm]
	\normalfont (i) 
\itshape 
If $\B B := (B^1, B^0) \in \C E^{\Disj}(Y)$, then 
$f^{-1}(\B B) := \big(f^{-1}(B^1), f^{-1}(B^0)\big) \in \C E^{\Disj} (X)$.\\[1mm]
\normalfont (ii) 
\itshape 
The assignment routine 
$f^* \colon \C E^{\Disj}(Y) \sto \C E^{\Disj}(X)$, $\B B \mapsto f^{-1}(\B B)$, is a monotone function.\\[1mm]
	\normalfont (iii) 
	\itshape 
	If $\B B$ is total, then $f^{-1}(\B B)$ is total.\\[1mm]
\normalfont (iv) 
\itshape $f^{-1}(\B A \cup \B B) =_{\mathsmaller{\C E^{\Disj} (X)}}
f^{-1}(\B A) \cup f^{-1}(\B B)$.\\[1mm]
\normalfont (v) 
\itshape $f^{-1}(\B A \cap \B B) =_{\mathsmaller{\C E^{\Disj} (X)}} f^{-1}(\B A) \cap f^{-1}(\B B)$.\\[1mm]
\normalfont (vi)
\itshape $f^{-1}(- \B A) =_{\mathsmaller{\C E^{\Disj} (X)}} - f^{-1}(\B A)$.\\[1mm]
\normalfont (vii)
\itshape $f^{-1}(\B A - \B B) =_{\mathsmaller{\C E^{\Disj} (X)}} f^{-1}(\B A) - f^{-1}(\B B)$.\\[1mm]
\normalfont (viii)
\itshape $f^{-1}(Y, \emptys_Y) =_{\mathsmaller{\C E^{\Disj} (X)}} (X, \emptys_X)$ and 
$f^{-1}(\emptys_Y, Y) =_{\mathsmaller{\C E^{\Disj} (X)}} (\emptys_X, X)$.\\[1mm]
\normalfont (ix)
\itshape If $\B {\C G}$ is an extensional subset of $\C E^{\Disj}(Y)$ and $ f^{-1}(\B {\C G}) := \{f^{-1}(\B G) \mid \B G \in \B {\C G}\}$, then 
$$f^{-1}\bigg(\bigcup \B {\C G}\bigg) =_{\mathsmaller{\C E^{\Disj} (X)}} \bigcup f^{-1}(\B {\C G}), \ \ \ \ 
f^{-1}\bigg(\bigcap \B {\C G}\bigg) =_{\mathsmaller{\C E^{\Disj} (X)}} \bigcap f^{-1}(\B {\C G}).$$
%The corresponding properties for $g$......not with equalities....and intersections?? CHECK Habil 
\normalfont (x)
\itshape If $\B C \in \C E^{\Disj}(Z)$, then $(g \circ f)^{-1}(\B C) =_{\mathsmaller{\C E^{\Disj} (X)}} f^{-1}(g^{-1}(\B C))$.
\end{proposition}

\begin{proof}
(i) Let $x^1 \in X$ such that $f(x^1) =_Y b^1$, for some $b^1 \in B^1$, and let $x^0 \in X$ such that $f(x^0) =_Y b^0$, for some $b^0 \in B^0$. By extensionality of $\neq_Y$ we have that the inequality $b^1 \neq_Y b^0$ implies the inequality $f(x^1) \neq_Y f(x^0)$, hence by strong extensionality of $f$ we get %the required inequality 
$x^1 \neq_X x^0$.\\
(ii) Monotonicity of $f^*$
%the complemented inverse image 
follows from monotonicity of the inverse image.\\
% of a function.\\
(iii) $\dom(f^{-1}(\B B)) := f^{-1}(B^1) \cup f^{-1}(B^) =_{\C E(X)} f^{-1}(B^1 \cup B^0) 
=_{\C E(X)} f^{-1}(Y) =_{\C E(X)} X$.\\
(iv, v) We show only (iv) and for (v) we work similarly. We have that
%, and for (iii) we proceed similarly.
\begin{align*}
	f^{-1}(\B A \cup \B B) & := f^{-1}\big(A^1 \cup B^1, A^0 \cap B^0 \big)\\
	& := \big(f^{-1}(A^1 \cup B^1), f^{-1}(A^0 \cap B^0)\big)\\
	& =_{\mathsmaller{\C E^{\Disj} (X)}}  \big(f^{-1}(A^1) \cup f^{-1}(B^1), f^{-1}(A^0) \cap f^{-1}(B^0)\big)\\
	& =: f^{-1}(\B A) \cup f^{-1}(\B B).
\end{align*}
(vi) follows immediately, and (vii) follows from (v), (vi) and the definition of $\B A - \B B$.\\
(viii) follows from Proposition~\ref{prp: extempty1}(i), (ix) from Proposition~\ref{prp: extempty1}(v), and (x) from Proposition~\ref{prp: extempty1}(xviii).
\end{proof}

Next we prove the basic properties of the \textit{complemented direct image} of a strong injection.
% (see Definition~\ref{def: setineq}).
% (in this case the implication in the definition of an injection holds in the strong sense of Rasiowa ...).

\begin{proposition}\label{prp: fimageproperties}
	Let $\C X, \C Y, \C Z \in\SetExtIneq$, $\fXY$ and $g \colon Y \to Z$ strong injections.\\[1mm]
%	an injection and an embedding.\\[1mm]
	%	\normalfont (i) 
	%	\itshape 
	%		$f^{-1}(- \B B) := - f^{-1}(\B B)$.\\[1mm]
		\normalfont (i) 
	\itshape 
	If $\B A := (A^1, A^0) \in \C E^{\Disj}(X)$, then 
	$f(\B A) := \big(f(A^1), f(A^0)\big) \in \C E^{\Disj} (Y)$.\\[1mm]
	\normalfont (ii) 
	\itshape 
	The assignment routine 
	$f_* \colon \C E^{\Disj}(X) \sto \C E^{\Disj}(Y)$, $\B A \mapsto f(\B A)$, is a monotone function.\\[1mm]
	\normalfont (iii) 
	%	\itshape 
%	If $\B B$ is total, then $f^{-1}(\B B)$ is total.\\[1mm]
%	\normalfont (iv) 
	\itshape $f(\B A \cup \B B) =_{\mathsmaller{\C E^{\Disj} (Y)}}
	f(\B A) \cup f(\B B)$.\\[1mm]
	\normalfont (iv) 
	\itshape $f(\B A \cap \B B) =_{\mathsmaller{\C E^{\Disj} (Y)}} f(\B A) \cap f(\B B)$.\\[1mm]
	\normalfont (v)
	\itshape $f(- \B A) =_{\mathsmaller{\C E^{\Disj} (Y)}} - f(\B A)$.\\[1mm]
	\normalfont (vi)
	\itshape $f(\B A - \B B) =_{\mathsmaller{\C E^{\Disj} (Y)}} f(\B A) - f(\B B)$.\\[1mm]
	\normalfont (vii)
	\itshape $f(X, \emptys_X) =_{\mathsmaller{\C E^{\Disj} (Y)}} (f(X), \emptys_Y)$ and 
	$f(\emptys_X, X) =_{\mathsmaller{\C E^{\Disj} (Y)}} (\emptys_Y, f(X))$.\\[1mm]
	\normalfont (viii)
	\itshape If $\B {\C H}$ is an extensional subset of $\C E^{\Disj}(X)$ and $ f(\B {\C H}) := \{f(\B H) \mid \B H \in \B {\C H}\}$, then 
	$$f\bigg(\bigcup \B {\C H}\bigg) =_{\mathsmaller{\C E^{\Disj} (Y)}} \bigcup f(\B {\C H}), \ \ \ \ 
	f\bigg(\bigcap \B {\C H}\bigg) =_{\mathsmaller{\C E^{\Disj} (Y)}} \bigcap f(\B {\C H}).$$
		\normalfont (ix)
	\itshape $\B B \cap f(X, \emptys_X) \subseteq f(f^{-1}(\B B))$.\\[1mm]
		\normalfont (x)
	\itshape $(g \circ f)(\B A) =_{\C E^{\Disj}(Z)} g(f(\B A))$.
\end{proposition}

\begin{proof}
	(i) Let $y^1 \in Y$, such that $y^1 =_Y f(a^1)$, for some $a^1 \in A^1$, and let $y^0 \in Y$, such that $y^0 =_Y f(a^0)$, for some $a^0 \in A^0$. As $a^1 \neq_X a^0$, by injectivity of $f$ we have that $f(a^1) \neq_Y f(a^0)$, and by extensionality of $\neq_Y$ we get $y^1 \neq_Y y^0$.\\
	(ii) Monotonicity of the complemented direct image follows from monotonicity of the direct image.\\
	(iii, iv) We show only (iii) and for (iv) we work similarly. We have that
	%, and for (iii) we proceed similarly.
	\begin{align*}
		f(\B A \cup \B B) & := f\big(A^1 \cup B^1, A^0 \cap B^0 \big)\\
		& := \big(f(A^1 \cup B^1), f(A^0 \cap B^0)\big)\\
		& =_{\mathsmaller{\C E^{\Disj} (Y)}}  \big(f(A^1) \cup f(B^1), f(A^0) \cap f(B^0)\big)\\
		& =: f(\B A) \cup f(\B B).
	\end{align*} 
(v) follows immediately and (vi) follows from (iv) and (v).\\
(vii) follows from Proposition~\ref{prp: extempty1}(ii), and (viii) follows from Proposition~\ref{prp: extempty1}(viii) and (xvi).\\
(ix) By Proposition~\ref{prp: extempty1}(ix, x) we have that
%, and for (iii) we proceed similarly.
\begin{align*}
	\B B \cap f(X, \emptys_X) & := \big(B^1 \cap f(X), B^0 \cup {\emptys_Y} \big)\\
	& :=_{\mathsmaller{\C E^{\Disj} (Y)}} \big(B^1 \cap f(X), B^0 \big)\\
	& =_{\mathsmaller{\C E^{\Disj} (Y)}}  \big(f(f^{-1}(B^1)), B^0\big)\\
	& \subseteq   \big(f(f^{-1}(B^1)), f(f^{-1}(B^0))\big)\\
	& =: f(f^{-1}(\B B)).
\end{align*} 
(x) It follows from Proposition~\ref{prp: extempty1}(xix).
\end{proof}

Notice that for the proof of (i) and (ii) above we need only $f$ to be an injection, and that under our hypotheses on $f$ we get equality in Proposition~\ref{prp: fimageproperties}(iv) and (viii).
By the previous two propositions and by Proposition~\ref{prp: extempty1}(xiv, xv) we get the following corollary.

\begin{corollary}\label{cor: invimage1}
Let $\C X, \C Y \in  \SetExtIneq$ and $\fXY$ strongly extensional and a strong injection.\\[1mm]
%	\normalfont (i) 
%	\itshape 
%		$f^{-1}(- \B B) := - f^{-1}(\B B)$.\\[1mm]
\normalfont (i) 
\itshape $f^{-1}(f(\B A)) =_{\C E^{\Disj} (X)} \B A$, for every $\B A \in \C E^{\Disj} (X)$.\\[1mm]
\normalfont (ii) 
\itshape If $f$ is also a surjection, then $f(f^{-1}(\B B)) =_{\C E^{\Disj} (Y)} \B B$, for every $\B B \in \C E^{\Disj} (X)$.

\end{corollary}

In analogy to the characterisations of injectivity and surjectivity through Proposition~\ref{prp: extempty1}(xiv, xv), we may define $\fXY$ to be a $\cs$-\textit{injection}, if $f$ is strongly extensional and a strong injection, while it is a $\cs$-\textit{surjection}, if it is strongly extensional, (strong) injection and surjection. 

%(Monos epis in the corresponding categories???) \textbf{DIFF}. 
%
%Next corollary is used to the characterisation of a homeomorphism between $\cs$-topological spaces through open $\cs$-continuous functions (Proposition...).
%
%
%
%\begin{corollary}\label{cor: invimage2}
%	Let $\C X, \C Y \in  \SetExtIneq$ and $\fXY$ strongly extensional and a strong injection.\\[1mm]
%	%	\normalfont (i) 
%	%	\itshape 
%	%		$f^{-1}(- \B B) := - f^{-1}(\B B)$.\\[1mm]
%	\normalfont (i) 
%	\itshape for every .\\[1mm]
%	\normalfont (ii) 
%	\itshape If 
%	
%\end{corollary}

\begin{proposition}\label{prp: emptyprod}
If $\C X, \C Y \in  \SetExtIneq$, the following hold:\\[1mm]
	\normalfont (i) 
	\itshape $\emptys_{X \times Y} =_{\C E(X \times Y)} ({\emptys_X} \times Y) \cup (X \times \emptys_Y)$.\\[1mm]
	\normalfont (ii) 
	\itshape $(\INT)$ $\emptys_{X \times Y} =_{\C E(X \times Y)} {\emptys_X} \times  \emptys_Y$.\\[1mm]
	\normalfont (iii) 
	\itshape $(X, \emptys_X) \times (Y, \emptys_Y) =_{\C E^{\Disj}(X \times Y)} (X \times Y, \emptys_{X \times Y})$.\\[1mm]
	\normalfont (iv) 
	\itshape $(\INT)$ $(\emptys_X, X) \times (\emptys_Y, Y) =_{\C E^{\Disj}(X \times Y)} (\emptys_{X \times Y}, X \times Y)$.
\end{proposition}

\begin{proof}
	(i) By Definitions~\ref{def: extempty} and~\ref{def: canonicalineq} we have that
	$\emptys_{X \times Y} := \{(x, y) \in X \times Y \mid x \neq_X x \vee y \neq_Y y\}.$
	Clearly, $\emptys_{X \times Y} \subseteq ({\emptys_X} \times Y) \cup (X \times \emptys_Y)$. For the converse inclusion, let $x \neq_X x$ and $y \in Y$. Then $(x, y) \in \emptys_{X \times Y}$. Similarly, if
	$y \neq_Y y$ and $x \in X$, then $(x, y) \in \emptys_{X \times Y}$.\\
	(ii) By extensionality of $\neq_X$ and $\neq_Y$ we have that
	\begin{align*}
	{\emptys_X} \times  \emptys_Y & := \{(u, w) \times X \times Y \mid \exists_{x \in {\ \emptys_X}}\exists_{y \in {\ \emptys_Y}}(u =_X x \wedge w =_Y y)\}\\
	& = \{(u, w) \times X \times Y \mid u \neq_X u \wedge w \neq_Y w)\},
	\end{align*}
		hence ${\emptys_X} \times  \emptys_Y \subseteq \emptys_{X \times Y}$. The converse inclusion follows from Ex Falso.\\
		(iii) By case (i) we have that
		$$(X, \emptys_X) \times (Y, \emptys_Y) := (X \times Y,  ({\emptys_X} \times Y) \cup (X \times \emptys_Y)) =_{\C E(X \times Y)} (X \times Y, \emptys_{X \times Y}).$$
		(iv) By case (ii) we have that
		$$\ \ \ \ \ \ \ \ \ \ \ \ \ \ \ \ \ \ \ \ \ \ \ \ \ (\emptys_X, X) \times (\emptys_Y, Y) := ({\emptys_X} \times  \emptys_Y,  (X \times Y) \cup (X \times Y)) =_{\C E(X \times Y)} (\emptys_{X \times Y}, X \times Y). \ \ \ \ \ \ \ \ \ \ \ \ \ \ \qedhere$$
\end{proof}

\begin{proposition}\label{prp: complemented4}
	Let $\B A, \B A_1, \B A_2 \in \C E^{\mathsmaller{\Disj}}(X)$ and $\B B, \B B_1, \B B_2, \B C \in \C E^{\mathsmaller{\Disj}}(Y)$.\\[1mm]
	\normalfont (i)
	\itshape $\B A \times (\B B \cup \B C) =_{\C E^{\mathsmaller{\Disj}}(X \times Y)} (\B A \times \B B) \cup (\B A \times \B C)$.\\[1mm]
	\normalfont (ii)
	\itshape $\B A \times (\B B \cap \B C) =_{\C E^{\mathsmaller{\Disj}}(X \times Y)} (\B A \times \B B) \cap (\B A \times \B C)$.\\[1mm]
	\normalfont (iii)
	\itshape $\B A \times (\B B - \B C) =_{\C E^{\mathsmaller{\Disj}}(X \times Y)} (\B A \times \B B) - (\B A \times \B C)$.\\[1mm]
		\normalfont (iv)
	\itshape $(\B A_1 \times \B B_1) \cap (\B A_2 \times \B B_2) =_{\C E^{\mathsmaller{\Disj}}(X \times Y)}
	(\B A_1 \cap \B A_2) \times (\B B_1 \cap \B B_2)$.
\end{proposition}

\begin{proof}
	We prove only (i). We have that
	\begin{align*}
		\B A \times (\B B \cup \B C) & := (A^1, A^0) \times (B^1 \cup C^1, B^0 \cap C^0)\\
		& := \big(A^1 \times (B^1 \cup C^1), (A^0 \times Y) \cup [X \times (B^0 \cap C^0)]\big)\\
		& =_{\C E^{\mathsmaller{\Disj}}(X \times Y)}  \big((A^1 \times B^1) \cup (A^1 \times C^1), [(A^0 \times Y) \cup (X \times B^0)] \ \cap\\
		& \ \ \ \ \ \ \ \ \ \ \ \ \ \ \ \  \cap  [(A^0 \times Y) \cup (X \times C^0)]\big)\\
		& := (\B A \times \B B) \times (\B A \times \B C).\qedhere
	\end{align*}
\end{proof}

The notion of a field of complemented subsets is defined by the same clauses as the notion of a field of subsets, but with respect to the operations on complemented subsets.

\begin{definition}\label{def: cfield}
If $F(\B A)$ is an extensional property on $\C E^{\Disj}(X)$, then 
$\B {\C F} := \big\{\B A \in \C E^{\Disj}(X) \mid F(\B A)\big\}$
is a field of complemented subsets, if $F(X, \emptys_X)$, and if $F(\B A)$ and $F(\B B)$, then $F(-\B A)$ and $F(\B A \cup \B B)$.
%A field of $c$-sets of type $(II)$ is defined similarly.	
\end{definition}

Clearly, if $\B {\C F}$ is a field of complemented subsets and $\B A, \B B \in \B {\C F}$, then $1_{\B A}, 0_{\B A}$ and $\B A \cap \B B \in \B {\C F}$. The following remark is immediate to show.

\begin{remark}\label{rem: cfield1}
If $\B {\C F}$ is a field of complemented subsets of $\C X$, then $\B {\C F}$ is a swap algebra of type $(\ti)$.
%$($of type $(II))$.
	
\end{remark}

Actually, what we defined above is a field of complemented subsets of type (I).
A notion of a field of complemented subsets of type (II) is defined similarly.
If $(\B \lambda(i))_{i \in I}$ is a family of complemented subsets indexed by some set $I$, the following distributivity property of complemented subsets was presented by Bishop in~\cite{Bi67}, p.~68:
\[ 0_{\B A} \cup \bigg(\B A \cap \bigcup_{i \in I} \B \lambda (i)\bigg) =_{\mathsmaller{\C E^{\Disj}(X)}}
0_{\B A} \cup \bigg[\bigcup_{i \in I} \big(\B A \cap \B \lambda (i)\big)\bigg],\]
or 
\[ 0_{\B A} \cup \bigg(\B A \cap \bigcup \B {\C F}\bigg) =_{\mathsmaller{\C E^{\Disj}(X)}}
0_{\B A} \cup \bigg[\bigcup (\B A \cap \B {\C F})\bigg],\]
where 
$\B A \cap \B {\C F} := \big\{\B B \in \C E^{\Disj}(X) \mid \exists_{\B F \in \B {\C F}}\big(\B B=_{\C E^{\Disj}(X)}  \B A \cap \B F\big)\big\}$.
Bishop's distributivity property is the constructive counterpart to the classical distributivity property
$$(D_I) \ \ \ \ \ \ \  \B A \cap \bigcup_{i \in I} \B \lambda (i) =_{\mathsmaller{\C E^{\Disj}(X)}}
\bigcup_{i \in I} \big(\B A \cap \B \lambda (i).$$
Distributivity $(D_I)$ is crucial to the classical proof that the standard relative topology on a subset of a topological space is indeed a topology. Constructively though, it cannot be accepted, in general. The constructive counterpart to the definition of the relative topology within $\cs$-topological spaces is presented in section~\ref{sec: cstop}. See~\cite{MWP23} for the proof of the following fact.

\begin{proposition}\label{prp: first2}
	\normalfont (i) 
	\itshape
	If $X := \Real$, the distributivity property $(D_\Nat)$ implies $\LPO$. \\[1mm]
	\normalfont (ii) 
	\itshape
	$(\INT)$ If $\B A \in \C E^{\Eisj}(X)$ is total, then $(D_I)$ holds, for every $X \in \SetExtIneq$. 
\end{proposition}

\section{Complemented points}
\label{sec: cp}

Standard topology is point-set topology, in which points play a crucial role for example, in the sudy of separation axioms of topological spaces. In the framework of complemented subsets points are certain complemented subsets that we call complemented points. In this way we can generalise the standard equivalence 
%The notion of point
$x \in A \TOT \{x\} \subseteq A$, where $\{x\} := \{y \in X \mid y =_X x\}$, to the inclusion relation between complemented points and complemented subsets.

\begin{definition}\label{def: cpoint} Let $\C X := (X, =_X, \neq_X) \in \SetExtIneq$. 
	A complemented point
	%or c-point of $\C X$, 
	is a complemented subset $\B x_P := (\{x\}, P)$, where $x \in X$. We also write $\B x_P := (x, P)$.
If $\B A := (A^1, A^0) \in \C E^{\Disj}(X)$, where $A^1 := X_{P^1}$ and $A^0 := X_{P^0}$, for some extensional properties $P^1, P^0$
	on $X$, we define the following elementhood and non-elementhood coditions
	$$(x, P) \cin \B A :\TOT (x, P) \subseteq \B A \TOT P^1(x) \wedge A^0 \subseteq P,$$
	$$(x, P) \ \ncin \B A :\TOT  (x, P) \cin (-\B A) \TOT  P^0(x) \wedge A^1 \subseteq P.$$
	We call 
	$(x, P)$ total, if it is a total complemented subset i.e., $\{x\} \cup P =_{\C E(X)} X$. Let $\Point(\C X)$ be the set of points of $\C X$, $\Point(x)$ the set of complemented points over $x$,
	%$x \in X$, 
	and $\Point(\B A)$ the set of points of $\B A$ i.e., 
	$$\Point(\C X) := \big\{(X_{Q^1}, X_{Q^0}) \in \C E^{\Disj}(X) \mid \exists_{x \in X}\big(X_{Q^1} =_{\C E(X)} \{x\}\big)\big\},$$
	$$\Point(x) := \big\{(X_{Q^1}, X_{Q^0}) \in \C E^{\Disj}(X) \mid X_{Q^1} =_{\C E(X)} \{x\}\big\},$$
	$$\Point(\B A) := \big\{(X_{Q^1}, X_{Q^0}) \in \C E^{\Disj}(X) \mid \exists_{x \in X}\big(X_{Q^1} =_{\C E(X)} \{x\} \wedge P^1(x) \wedge X_{P^0} \subseteq X_{Q^0}\big)\big\}.$$
	If $(x, P) \in \Point(\C X)$, then $(P, x)$  
	%a complemented point
	 %$(x, P)$ 
	 is a complemented copoint of $\C X$. Let $-\Point(\C X), -\Point(c)$, and  $-\Point(\B A)$ be the corresponding sets. We call $(x, P)$ a potential point of $\B A$, if $A^0 \subseteq P$. If $\B {\C F}$ is an extensional subset of $\C E^{\Disj}(X)$, then $(x, P)$ is a potential point of $\B {\C F}$, if it is a potential point of every 
	% complemented subset in 
	 element of $\B {\C F}$.
\end{definition}

\begin{definition}\label{def: notation1}
	If $\phi(x, P)$ is a formula in $\BST$, we introduce the folowing notational conventions:
	$$\forall_{(x,P) \ccin \C X}\phi(\B x) := \forall_{(x,P) \in \Point(\C X)}\big(\phi(x,P)\big), \ \ \ \ 
	\exists_{(x,P) \ccin \C X}\phi(\B x) := \exists_{(x,P) \in \Point(\C X)}\big(\phi(x,P)\big),$$
	$$\forall_{(x,P) \ccin \B A}\phi(\B x) := \forall_{(x,P) \in \Point(\B A)}\big(\phi(x,P)\big), \ \ \ \ 
\exists_{(x,P) \ccin \B A}\phi(\B x) := \exists_{(x,P) \in \Point(\B A)}\big(\phi(x,P)\big).$$
	\end{definition}

Although we can prove $(x, P) \cin \B A \cap \B B \TOT (x, P) \cin \B A \wedge (x, P) \cin \B B$, we cannot show always that $(x, P) \cin \B A \cup \B B \TOT (x, P) \cin \B A \vee (x, P) \cin \B B$. If we had used the operations between complemented subsets of type (II), then the former equivalence also fails. 

\begin{proposition}\label{prp: cpoint1}
	Let $\C X := (X, =_X, \neq_X) \in \SetExtIneq$, $x \in X$, $\B A := (A^1, A^0) \in \C E^{\Disj}(X)$ and 
	$\B B := (B^1, B^0) \in \C E^{\Disj}(X)$, where $A^1 := X_{P^1}, A^0 := X_{P^0}$ and
	$B^1 := X_{Q^1}, B^0 := X_{Q^0}$, for some extensional properties $P^1, P^0, Q^1$ and $Q^0$
	on $X$. Let also $\B {\C F}$ be an extensional subset of $\C E^{\Disj}(X)$. \\[1mm]
		\normalfont (i) 
	\itshape $\B A \subseteq \B B \TOT \forall_{(x,P) \ccin \C X}\big((x,P) \cin \B A \To (x,P) \cin \B B\big)$.\\[1mm]
		\normalfont (ii) 
	\itshape $(x, P) \cin \B A \vee (x, P) \cin \B B \To (x, P) \cin \B A \cup \B B$.\\[1mm]
		\normalfont (iii) 
	\itshape $(x, P) \cin \B A \cup \B B \TOT (x, P) \cin (A^1, A^0 \cap B^0) \vee (x, P) \cin (B^1, A^0 \cap B^0)$.\\[1mm]
		\normalfont (iv) 
	\itshape If $(x, P)$ is a potential point of $\B A$ and $\B B$, then $(x, P) \cin \B A \cup \B B \To (x, P) \cin \B A \vee (x, P) \cin \B B$.\\[1mm]
		\normalfont (v) 
	\itshape $(x, P) \cin \B A \cap \B B \TOT (x, P) \cin \B A \wedge (x, P) \cin \B B$.\\[1mm]
		\normalfont (vi) 
	\itshape $(x, P) \ \ncin \B A \cup \B B \TOT (x, P) \ \ncin \B A \wedge (x, P) \ \ncin \B B$.\\[1mm]
		\normalfont (vii) 
	\itshape $(x, P) \ \ncin \B A \vee (x, P) \ \ncin \B B \To (x, P) \ \ncin \B A \cap \B B$.\\[1mm]
		\normalfont (viii) 
	\itshape $(x, P) \nncin \B A \cap \B B \TOT (x, P) \cin (A^0, A^1 \cap B^1) \vee (x, P) \cin (B^0, A^1 \cap B^1)$.\\[1mm]
	\normalfont (ix) 
	\itshape If $(x, P)$ is a potential point of $-\B A$ and $-\B B$, then
	$(x, P) \ \ncin \B A \cap \B B \To (x, P) \ \ncin \B A \vee (x, P) \ \ncin \B B$.\\[1mm]
		\normalfont (x) 
	\itshape In general, the following hold:
	$$(x, P) \cin \bigcup \B {\C F} \TOT \exists_{X_{P^1} \in \C F^1}\bigg((x, P) \cin 
\bigg(X_{P^1}, \bigcap \C F^0\bigg)\bigg),$$
$$(x, P) \cin \bigcap \B {\C F} \TOT \forall_{(X_{P^1}, X_{P^0}) \in\B {\C F}}\big((x, P) \cin 
(X_{P^1}, X_{P^0})\big),$$ 
$$(x, P) \ \ncin \bigcup \B {\C F} \TOT \forall_{(X_{P^1}, X_{P^0}) \in\B {\C F}}\big((x, P) \ \ncin (X_{P^1}, X_{P^0})\big),$$
$$(x, P) \ \ncin \bigcap \B {\C F} \TOT \exists_{ X_{P^0} \in \C F^0}\bigg((x, P)  \cin \bigg(X_{P^0}, \bigcap \C F^1\bigg)\bigg).$$
	\normalfont (xi) 
\itshape If $(x, P)$ is a potential point of $\B {\C F}$, 
%$(X_{P^1}, X_{P^0})$, for every $(X_{P^1}, X_{P^0}) \in\B {\C F}$, 
the following hold:
$$(x, P) \cin \bigcup \B {\C F} \TOT \exists_{(X_{P^1}, X_{P^0}) \in\B {\C F}}\big((x, P) \cin 
(X_{P^1}, X_{P^0})\big),$$
$$(x, P) \ \ncin \bigcap \B {\C F} \TOT \exists_{(X_{P^1}, X_{P^0}) \in\B {\C F}}\big((x, P) \ \ncin (X_{P^1}, X_{P^0})\big).$$
%	$(x, P) \ \ncin \bigcup \B {\C F} \TOT \forall_{X_{P^1} \in \C F^1}\big(P^1(x)\big)$.\\[1mm]
	\normalfont (xii) 
\itshape If $\B B \in \B {\C F}$, then $\B B \subseteq  \bigcup \B {\C F} $.
\end{proposition}

\begin{proof}
(i) Suppose first that $A^1 \subseteq B^1$ and $B^0 \subseteq A^0$, and let $(x, P) \subseteq \B A$ i.e., 
$P^1(x)$ and $A^0 \subseteq P$. As $P^1(x) \To Q^1(x)$ and $B^0 \subseteq A^0 \subseteq P$, we conclude that $(x, P) \cin \B B$. For the converse inclusion let $x \in A^1$. As $(x, A^0) \cin \B A$, we get $(x, A^0) \cin \B B$ i.e., $x \in B^1$ and $B^0 \subseteq A^0$. As $x \in A^1$ is arbitrary, we conclude that $A^1 \subseteq B^1$, hence $\B A \subseteq \B B$. \\
(ii) If $x \in A^1$ and $A^0 \subseteq P$, then $x \in A^1 \cup B^1$ and $A^0 \cap B^0 \subseteq A^0 \subseteq P$ i.e., $(x, P) \cin \B A \cup \B B$. If $x \in B^1$ and $B^0 \subseteq P$, we work similarly.\\
(iii) We have that $(x, P) \cin (A^1 \cup B^1, A^0 \cap B^0) :\TOT x \in A^1 \cup B^1 \wedge A^0 \cap B^0 \subseteq P$. If $x \in A^1$, then $(x, P) \cin (A^1, A^0 \cap B^0)$, while if $x \in B^1$, then $(x, P) \cin (B^1, A^0 \cap B^0)$. Conversely, $(x, P) \cin (A^1, A^0 \cap B^0)$, then $x \in A^1 \subseteq A^1 \cup B^1$, and as $A^0 \cap B^0 \subseteq P$, we get $(x, P) \cin \B A \cup \B B$. If $(x, P) \cin (B^1, A^0 \cap B^0)$, we work similarly.\\
(iv) We have that $(x, P) \cin (A^1 \cup B^1, A^0 \cap B^0) :\TOT x \in A^1 \cup B^1 \wedge A^0 \cap B^0 \subseteq P$. If $x \in A^1$, and as $A^0 \subseteq P$, we get $(x, P) \cin \B A$. Similarly, if $x \in B^1$, we get $(x, P) \cin \B B$.\\
(v) We have that $(x, P) \cin (A^1 \cap B^1, A^0 \cup B^0) :\TOT x \in A^1 \cap B^1 \wedge A^0 \cup B^0 \subseteq P$, hence $x \in A^1 \wedge A^0 \subseteq P$ and $x \in B^1 \wedge B^0 \subseteq P$. Conversely, if $x \in A^1 \wedge A^0 \subseteq P$ and $x \in B^1 \wedge B^0 \subseteq P$, then $ x \in A^1 \cap B^1 \wedge A^0 \cup B^0 \subseteq P$.\\
(vi) By the use of (v) in the third equivalence that follows we have that
\begin{align*}
(x, P) \ \ncin \B A \cup \B B & :\TOT (x, P) \cin -(\B A \cup \B B)\\
& \TOT (x, P) \cin (-\B A) \cap (-\B B)\\
& \TOT (x, P) \cin (-\B A) \wedge (x, P) \cin (-\B B)\\
& \TOT: (x, P) \ \ncin \B A \wedge (x, P) \ \ncin \B B.
\end{align*}
(vii)-(ix) We work as in the proof of (ii)-(iv).\\
(x)-(xi) These are obvious generalisations of the previous finite cases and their proofs are similar.\\
(xii) It follows easily from the definition of $\bigcup \B {\C F}$.
\end{proof}

By Proposition~\ref{prp: cpoint1}(i) we get
%have that 
$\B A =_{\C E^{\Disj}(X)} \B B \TOT \forall_{(x,P) \ccin \C X}\big(\B x \cin \B A \TOT \B x \cin \B B\big).$
Complemented points help us also to define an
inequality\footnote{The inequality between complemented subsets is treated differently in~\cite{MWP23}.} on $\C E^{\Disj}(X)$ applying strong negation on the previous equality condition of complemented subsets and considering the ``strong negation'' of $\B x \cin \B A$ to be $\B x \nncin \B A$, and vice versa.

\begin{definition}\label{def: csineq}
If $\B A, \B B \in  \C E^{\Disj}(X)$, let the inequality relation
$$\B A \neq_{\C E^{\Disj}(X)} \B B :\TOT \exists_{(x,P) \ccin \C X}\big[(\B x \cin \B A \wedge \B x \nncin \B B) \vee 
(\B x \cin \B B \wedge \B x \nncin \B A)\big].$$
	\end{definition}

The following iinequality is immediate to show.

\begin{corollary}\label{cor: csineq1}
If  $\C X \in \SetExtIneq$ and $X$ is inhabited, then $0_{\B X} \neq_{\C E^{\Disj}(X)} 1_{\B X}$. Moroever, if $(x, P) \neq_{\C E^{\Disj}(X)} (y, Q)$, then $x \neq_X y$.
\end{corollary}

%
%\begin{proof}
%(i) 	
%	
%\end{proof}

In analogy to
% the case of
subsets, the disjointness of two complemented subsets is derived from strong negation of the corresponding overlapping relation, where strong negation of the elementhood $(x, P) \cin \B A$ is understood to be the non-elementhood $(x, P) \nncin \B A$, and vice versa.

\begin{definition}\label{def: cineq}
	If $\C X := (X, =_X, \neq_X) \in \SetIneq$ and $\B A, \B B \in \C E^{\Disj}(X)$, let 
	$$\B A \between \B B :\TOT \exists_{(x, P) \ccin \C X}\big((x, P) \cin \B A \wedge (x, P) \cin \B B\big),$$
	$$\B A \Disj \B B :\TOT \forall_{(x, P) \ccin \C X}\big((x, P) \nncin \B A \vee (x, P) \nncin \B B\big).$$

\end{definition}

The above overlapping relation is equivalent to the following condition
$$\exists_{(x,P) \ccin \B A}\exists_{(y,Q) \ccin \B B}\big((x,P) =_{\C E^{\Disj}(X)} (y,Q)\big).$$
Moreover, it is easy to show that $(x,P) \ \Disj \ (y,Q) \To x \neq_X y$.
%If $\B A \subseteq \B B$ and $\B B$ has no complemented point that is not in $\B A$, then $\B B \subseteq \B A$, where we use a strong formulation, in order to avoid the aforementioned negation.

\begin{proposition}\label{prp: cpoint2}
		Let $\C X := (X, =_X, \neq_X), \C Y := (Y, =_Y, \neq_Y) \in \SetExtIneq$, $(x, P) \in \Point(\C X)$,
		$(y, Q) \in \Point(\C Y)$, $\B A := (A^1, A^0) \in \C E^{\Disj}(X)$, and $\B C := (C^1, C^0) \in \C E^{\Disj}(Y)$.\\[1mm]
%	$(x, P) \cin \bigcap \B {\C F} \TOT \forall_{X_{P^1} \in \C F^1}\big(P^1(x)\big)$.\\[1mm]
	\normalfont (i) 
	\itshape $\B A =_{\C E^{\Disj(X)}} \bigcup \Point(\B A)$.\\[1mm]
%	\normalfont (ii) 
%	\itshape $(x, \{x\}^{\neq_X}) \subseteq (x, P) \subseteq (x, \emptys_X)$.\\[1mm]
	\normalfont (ii) 
	\itshape $\Point(x)$ is closed under $($arbitrary$)$ unions and intersections.\\[1mm]
	\normalfont (iii) 
	\itshape If $\B A$ is inhabited and $\B A \subseteq (x, P)$, then $\B A \in \Point(x)$.\\[1mm]
%		\normalfont (v) 
%	\itshape $(x, \{x\}^{\neq_X}) \cin \B A \TOT x \in A^1$ and $(x, \{x\}^{\neq_X}) \nncin \B A \TOT x \in A^0$.\\[1mm]
%		\normalfont (vi) 
%	\itshape $(x, \{x\}^{\neq_X}) \cin \B A \cup \B B \To (x, \{x\}^{\neq_X}) \cin \B A \vee (x, \{x\}^{\neq_X}) \cin \B B$.\\[1mm]
%		\normalfont (vii) 
%	\itshape $(x, \{x\}^{\neq_X}) \nncin \B A \cap \B  B \To (x, \{x\}^{\neq_X}) \nncin \B A \vee (x, \{x\}^{\neq_X}) \nncin \B B$.\\[1mm]
		\normalfont (iv) 
	\itshape $(x, \emptys_X) \cin \B A \TOT x \in A^1 \wedge A^0 =_{\C E(X)} \emptys_X$.\\[1mm]
		\normalfont (v) 
	\itshape If $(x, P) \cin \B A$ and $(y, Q) \cin \B C$, then $(x, P) \times (y, Q) \cin \B A\times \B C$.\\[1mm]
		\normalfont (vi) 
	\itshape $(\INT)$ $(x, P) \cin (X, \emptys_X)$ and $(x, P) \nncin (\emptys_X, X)$.\\[1mm]
		\normalfont (vii) 
	\itshape If $(x, P) \cin \B A$ and $\B A \subseteq \B B$, then $(x,P) \cin \B B$.\\[1mm]
	\normalfont (viii) 
	\itshape If $(x, P) \nncin \B B$ and $\B A \subseteq \B B$, then $(x,P) \nncin \B A$.
%	\normalfont (xiii)
%	\itshape  $0_{\B X} \Disj 1_{\B X}$.
\end{proposition}

\begin{proof}
(i) As $(x, A^0) \cin \B A$, for every $x \in A^1$, we get
$$\bigcup \Point(\B A) :=  \bigg(\bigcup \Point^1(\B A), \bigcap \Point^0(\B A)\bigg)=_{\C E^{\Disj}(X)} \B A.$$
(ii) We use the definitions
$(x, P) \cup (x, Q) := (x, P \cap Q)$ and $(x, P) \cap (x, Q) := (x, P \cup Q)$.\\
%(v) $(x, \{x\}^{\neq_X}) \cin (A^1, A^0) :\TOT x \in A^1 \wedge A^0 \subseteq \{x\}^{\neq_X} \TOT x \in A^1$.  The second equivalence follows similarly, while 
%(vi) and (vii) follow immediately from (v), and (viii) follows trivially. To show 
(v) The hypotheses $x \in A^1 \wedge A^0 \subseteq P$ and $y \in C^1 \wedge C^0 \subseteq Q$ imply that 
$$\big((x,y), (P \times Y) \cup (X \times Q)\big) \cin \big(A^1 \times C^1, (A^0 \times Y) \cup (X \times C^0)\big),$$
as $(x, y) \in A^1 \times C^1$, $(A^0 \times Y) \subseteq  (P \times Y)$, and $(X \times C^0) \subseteq (X \times Q)$. 
The rest cases follow easily.
%cases are straightforward to show.
%proofs of (x) and (xi) are straightforward, and (xii) follows immediately from (xi).
\end{proof}

%Notice that (ix) does not hold in general, 
If the product of two complemented subsets is 
defined according to the second type of operations on them (see~\cite{PW22, MWP23}), then
Proposition~\ref{prp: cpoint2}(v) doesn't hold in general. 
The following propositions are straightforward to show.

\begin{proposition}\label{prp: finvpoint1}
	Let $\C X, \C Y, \fXY$ in $(\SetExtIneq, \StrExtFun)$ with a left inverse $g \colon Y \to X$.\\[1mm]
	\normalfont (i)
	\itshape $f$ inverses points i.e., $f^{-1}(y, Q) := (g(y), f^{-1}(Q)) $ is in $\Point(\C X)$, for every $(y, Q) \in \Point(\C Y)$.\\[1mm]
		\normalfont (ii)
	\itshape If $\B B \in \C E^{\Disj}(Y)$, and $(y,Q) \cin \B B$, then $f^{-1}(y, Q) \cin f^{-1}(\B B)$.\\[1mm]
	\normalfont (iii)
	\itshape If $\B B \in \C E^{\Disj}(Y)$, and $(y,Q) \nncin \B B$, then $f^{-1}(y, Q) \nncin f^{-1}(\B B)$.
	
\end{proposition}

\begin{proposition}\label{prp: finvpoint2}
	Let $\C X, \C Y, \fXY$ and $g \colon Y \to X$ in $(\SetExtIneq, \StrExtFun)$. The following are equivalent:\\[1mm]
	\normalfont (i)
	\itshape $g$ is a left inequality-adjoint to $f$.\\[1mm] 
	\normalfont (ii)
	\itshape $f$ inverses points with $g$ as a modulus of inversion i.e., $f^{-1}(y, Q) := (g(y), f^{-1}(Q)) \in \Point(\C X)$, for every $(y, Q) \in \Point(\C Y)$, and $g$ inverses points with $f$ as a modulus of inversion i.e., $g^{-1}(x, P) := (f(x), g^{-1}(P)) \in \Point(\C Y)$, for every $(x, P) \in \Point(\C X)$.
	
\end{proposition}

\begin{proposition}\label{prp: fimagepoint}
Let $\C X, \C Y$ and $\fXY$ in $\SetExtIneq$. The following are equivalent:\\[1mm]
\normalfont (i)
\itshape There is $g \colon Y \to X$ in $\StrExtFun$, such that $g \circ f$ is an injection.\\[1mm]
\normalfont (ii)
\itshape $f$ is an injection.\\[1mm]
\normalfont (iii)
\itshape $f$ preserves points i.e., $f(x, P) := (f(x), f(P))$ is in $\Point(\C Y)$, for every $(x,P) \in \Point(\C X)$.
	
\end{proposition}

\begin{remark}
	\normalfont
If $f$ preserves points, $\B A \in \C E^{\Disj}(X)$ and $(x, P) \cin \B A$, then $f(x, P) \cin f(\B A)$.
Contrary to the case of classical subsets, where $f(x) \notin f(A) \To x \notin A$, we have that if $f(x, P) \nncin f(\B A)$, then $x \in f^{-1}(f(A^0))$ and $f^{-1}(f(A^1)) \subseteq P$, which does not imply, in general, that $x \in A^0$ and $A^1 \subseteq P$ i..e, $(x,P) \nncin \B A$, unless $f$ is an embedding (see Proposition~\ref{prp: extempty1}(xiv)), hence unless $f$ is a strong injection.
\end{remark}
%
%\begin{proof}
%(i) $\To$ (ii) 
%	
%	
%\end{proof}

%
%
%\section{Canonical complemented points}
%\label{sec: ccp}

The behaviour of complemented points e.g., with respect to union of complemented subsets causes several problems, especially when one wants to pass from the union of complemented subsets to local properties, as in the case of the local characterisation of a base in standard topology. It turns out that we need to concentrate on fewer, more well-behaved complemented points.
For that, we introduce and study the so-called canonical complemented points, which have a simple form and suit perfectly to the formulation of equivalent local version of several concepts in $\cs$-topology.

\begin{definition}[Canonical complemented points]\label{def: canonicalpoint}
	Let $\C X := (X, =_X, \neq_X) \in \SetExtIneq$. A complemented point of the form $(x, \{x\}^{\neq_X})$, where $x \in X$, is called canonical. Its complement is a canonical copoint. We also denote the canonical point by $\B x$, and the corresponding copoint by $-\B x$. 	
	Their sets are denoted by $\cPoint(X)$ and $-\cPoint(\C X)$, respectively. If $\B A \in \C E^{\Disj}(X)$, the sets of canonical points and copoints of $\B A$ are denoted by $\cPoint(\B A)$ and $-\cPoint(\B A)$, respectively.
	
\end{definition}

It is imeediate to see that for canonical complemented points, and in connection to Corollary~\ref{cor: csineq1}, we have that $\B x \neq_{\C E^{\Disj}(X)} \B y \TOT x \neq_X y$. Next propoeition shows how much better the behaviour of canonical complemented points is in comparison to arbitrary complemented points. Furthermore, case (viii) is used in the proof of Proposition~\ref{prp: weak}.

\begin{proposition}\label{prp: ccpoint1}
	Let $\C X := (X, =_X, \neq_X), \C Y := (Y, =_Y, \neq_Y) 
	\in \SetExtIneq$, $x \in X$, 
%	$(x, P) \in \Point(\C X)$, 	$(y, Q) \in \Point(\C Y)$, 
	$\B A \in \C E^{\Disj}(X)$, $(\B G_i)_{i \in I}$ a family of complemented subsets of $\C X$ indexed by some set $(I, =_I)$, and $f \colon X \to Y$ strongly extensional.\\[1mm]
	%	$(x, P) \cin \bigcap \B {\C F} \TOT \forall_{X_{P^1} \in \C F^1}\big(P^1(x)\big)$.\\[1mm]
		\normalfont (i) 
	\itshape $\B x \cin \B A \TOT x \in A^1$ and $\B x \nncin \B A \TOT x \in A^0$.\\[1mm]
	\normalfont (ii) 
	\itshape $\B x \subseteq (x, P) \subseteq (x, \emptys_X)$, for every $(x, P) \in \Point(\C X)$.\\[1mm]
	\normalfont (iii) 
	\itshape $\B A \subseteq \bigcup \cPoint(\B A)$.\\[1mm]
%	\normalfont (iii) 
%	\itshape $\Point(x)$ is closed under $($arbitrary$)$ unions and intersections.\\[1mm]
		\normalfont (iv) 
	\itshape $\B x \cin \B A \cup \B B \To \B x \cin \B A \vee \B x \cin \B B$.\\[1mm]
	\normalfont (v) 
	\itshape $\B x \nncin \B A \cap \B  B \To \B x \nncin \B A \vee \B x \nncin \B B$.\\[1mm]
	\normalfont (vi) 
	\itshape $\B x \cin \bigcup_{i \in I}\B G_i \TOT \exists_{i \in I}(\B x \cin \B G_i)$.\\[1mm]
		\normalfont (vii) 
	\itshape $\B x \nncin \bigcap_{i \in I}\B G_i \TOT \exists_{i \in I}(\B x \nncin \B G_i)$. \\[1mm]
		\normalfont (viii) 
	\itshape If $\B H \in \C E^{\Disj}(Y)$ and $x \in X$, then $\B x \cin f^{-1}(\B H) \TOT \B {f(x)} \cin \B H$. 
	
\end{proposition}

\begin{proof}
	(i) $(x, \{x\}^{\neq_X}) \cin (A^1, A^0) :\TOT x \in A^1 \wedge A^0 \subseteq \{x\}^{\neq_X} \TOT x \in A^1$. Case (ii) is trivial, and (iii) follows from the observation $A^0 \subseteq \bigcap_{x \in A^1}\{x\}^{\neq_X} = (A^1)^{\neq_X}$. Cases (iv) and  follow immediately from (i), while (vi) and (vii) are trivial generalisations of (iv) and (v), respectively. Condition $\B x \cin f^{-1}(\B H)$ in (viii) is the conjunction $x \in f^{-1}(H^1) $ and $f^{-1}(H^0) \subseteq \{x\}^{\neq_X}$, while  $\B {f(x)} \cin \B H$ is the conjunction $f(x) \in H^1$ and $H^0 \subseteq \{f(x)\}^{\neq_X}$. Clearly, $x \in f^{-1}(H^1) \TOT f(x) \in H^1$. If $y \in H^0$, then $y \neq_Y f(x)$, as $f(x) \in H^1$ i.e., in this proof we didn't use the hypothesis $f^{-1}(H^0) \subseteq \{x\}^{\neq_X}$. If $x{'} \in f^{-1}(H^0)$ i.e., there is some $y \in H^0$ with $f(x{'}) =_Y y$, then by the hypothesis $H^0 \subseteq \{f(x)\}^{\neq_X}$ and the extensionality of $\neq_Y$ we get $f(x{'}) \neq_Y f(x)$, and, as $f$ is strongly extensional, we conclude that $x{'} \neq_X x$.
\end{proof}

\begin{definition}\label{def: notation2}
	If $\phi(x, \{x^{\neq_X}\})$ is a formula in $\BST$, let the notational conventions:
	$$\forall_{\B x \ccin \C X}\phi(\B x) := \forall_{(x,\{x^{\neq_X}\}) \in \cPoint(\C X)}\big(\phi(x,\{x^{\neq_X}\})\big), \ \ \ \ 
	\exists_{\B x \ccin \C X}\phi(\B x) := \exists_{(x,\{x^{\neq_X}\}) \in \cPoint(\C X)}\big(\phi(x,\{x^{\neq_X}\})\big),$$
	$$\forall_{\B x \ccin \B A}\phi(\B x) := \forall_{(x,\{x^{\neq_X}\}) \in \cPoint(\B A)}\big(\phi(x,\{x^{\neq_X}\})\big), \ \ \ \ 
	\exists_{\B x \ccin \B A}\phi(\B x) := \exists_{(x,\{x^{\neq_X}\}) \in \cPoint(\B A)}\big(\phi(x,\{x^{\neq_X}\})\big).$$
\end{definition}

Understanding $\B x \nncin \B A$ as a sort of strong negation of $\B x \cin \B A$, then the canonical complemented points characterise the inclusion relation between complemeneted subsets, if we interpret the implication in Proposition~\ref{prp: cpoint1}(i) in the strong way. 
Next fact follows immediately from Proposition~\ref{prp: ccpoint1}(i).

\begin{proposition}\label{prp: ccpoint2}
	If $\C X := (X, =_X, \neq_X) \in \SetExtIneq$ and $\B A, \B B \in \C E^{\Disj}(X)$,  then
	$$\B A\subseteq \B B \TOT \forall_{\B x \ccin \C X}\big[(\B x \cin \B A \To \B x \cin \B B) \wedge (\B x \nncin \B B \To \B x \nncin \B A)\big).$$
\end{proposition}

\section{The canonical $\cs$-topology induced by a metric}
\label{sec: csmetric}

Metric spaces are important examples of a topological space in classical topology, and crucial case-studies to all approaches to constructive topology.
Many of the concepts of $\cs$-topological spaces introduced in later sections need to be tested on metric spaces first. In this section we combine classical topological notions, within the constructive complemented subsets-point of view, with the constructive theory of metric spaces.  

\begin{remark}\label{rem: order1}
\normalfont 
Two facts from the basic constructive theory of real numbers that are provable within $\MIN$ are the following: the condition $a \geq b$ is equivalent to the formula $\neg{(a < b)}$ (Lemma 2.18 in~\cite{BB85}, p.~26), and  the constructive version of trichotomy of reals, according to which if $a < b$, then for every $c \in \Real$ either $c < b$ or $c > a$ (Corollary 2.17 in~\cite{BB85}, p.~26).
\end{remark}

First we define the canonical $\cs$-topology induced by some metric $d$ on a set $(X, =_X)$. 
In accordance to the canonical inequality $\neq_{(X,d)}$ (see Definition~\ref{def: canonical}), the empty subset of a metric space is 
defined by
%the following extensional subset of $X$:
$$\emptys_{(X,d)} := \{y \in X \mid d(y, y) > 0\}.$$
If $A \subseteq X$, its $\neq_{(X,d)}$-complement is defined by
$$A^{\neq_{(X,d)}} := \{x \in X \mid \forall_{a \in A}(d(x, a) > 0)\}.$$
 Only a small part of the proof of the next proposition requires $\INT$. The use of the extensional empty subset $\emptys_{(X,d)}$ allows us to provide several proofs concerning it, but not all, within $\MIN$. \textit{Throughout this section $(X, =_X, d)$, or simply $X$, is a metric space with a metric $d$, and $\epsilon \in \Real$ with $\epsilon > 0$}.

\begin{definition}[Complemented ball]\label{def: ball} 
	If $x \in X$,
	%and $\epsilon > 0$, 
	 the $\epsilon$-complemented ball around $x$ is defined by
	$$\B B(x, \epsilon) := \big([d_{x} < \epsilon], [d_{x} \geq \epsilon]\big),$$
	where 
	$$[d_{x} < \epsilon] := \{x{'} \in X \mid d(x, x{'}) < \epsilon\}, \ \ \ \ [d_{x} \geq \epsilon] := \{x{''} \in X \mid d(x, x{''}) \geq \epsilon\}.$$
\end{definition}

%The pair 
\begin{remark}\label{rem: ball1}
	\normalfont (i)
	\itshape $\B B(x, \epsilon)$ is a complemented subset of $(X, =_X, \neq_{(X,d)})$ with $\B x \cin \B B(x, \epsilon)$.	\\[1mm]
		\normalfont (ii)
	\itshape $[d_{x} \geq \epsilon] =_{\C E(X)} [d_{x} < \epsilon]^{\neq_{(X,d)}}$, and $\B B(x, \epsilon)$ is a $1$-tight complemented subset.\\[1mm]
		\normalfont (iii)
	\itshape If $\epsilon_1, \epsilon_2 > 0$, then we have that 
	$\epsilon_1 \leq \epsilon_2 \To \B B(x, \epsilon_1) \subseteq \B B(x, \epsilon_2)$.
	
\end{remark}

\begin{proof}
(i) If $x{'} \in [d_{x} < \epsilon]$ and $x{''} \in [d_{x} \geq \epsilon]$, then by the triangle inequality we have that $(x, x{''}) \leq d(x, x{'}) + d(x{'}, x{''})$, hence $d(x{'}, x{''}) \geq d(x, x{''}) - d(x, x{'}) > 0$.
 Clearly, $\B x \cin \B B(x, \epsilon)$ if and only if $x \in [d_{x} < \epsilon]$, which is trivially the case. \\
 (ii) The inclusion $[d_{x} \geq \epsilon] \subseteq [d_{x} < \epsilon]^{\neq_{(X,d)}}$ follows by case (i). For the converse inclusion, let $y \in [d_{x} < \epsilon]^{\neq_{(X,d)}}$ with $d(x, y) < \epsilon$ i.e., $y \in [d_x < \epsilon]$. By hypothesis on $y$ we get $d(y,y) > 0$, which is impossible, hence, $d(x,y)\geq \epsilon$. $1$-tightness of $\B B(x, \epsilon)$ follows from the implication $\neg{(d(x, x{'}) < \epsilon)} \To d(x, x{'}) \geq \epsilon$.\\
 (iii) If $\epsilon_1 \leq \epsilon_2$, then $[d_{x} < \epsilon_1] \subseteq [d_{x} < \epsilon_2]$ and $[d_{x} \geq \epsilon_2] \subseteq [d_{x} \geq \epsilon_1]$.
\end{proof}

\begin{definition}[The canonical $\cs$-topology induced by a metric]\label{def: canonicalcs}
If $\B G \in \C E^{\Disj}(X)$, let the formula
$$T_d(\B G) := \forall_{\B x \ccin \B G}\exists_{\epsilon > 0}\big(\B B(x, \epsilon) \subseteq \B G\big) :\TOT 
\forall_{\B x \ccin \B G}\exists_{\epsilon > 0}\big([d_{x} < \epsilon] \subseteq G^1 \wedge G^0 \subseteq [d_{x} \geq \epsilon]\big).$$
If $T_d(\B G)$, we call $\B G$ a $\B {\C T}_{d}$-open complemented subset of $X$. Let their set $\B {\C T}_{d}$ be the canonical $\cs$-topology on $X$ induced by its metric $d$. If $\B G$ is a $\B {\C T}_{d}$-open complemented subset of $X$, a modulus of openness $\op_{\B G}$ for $\B G$ is a function $\op_{\B G} \colon G^1 \to (0, +\infty)$, such that
$$\forall_{\B x \ccin \B G}\big(\B B(x, \op_{\B G}(x)) \subseteq \B G\big).$$ 
If $\B G, \B H \in \B {\C T}_d$, such that $\B G =_{\C E^{\Disj}(X)} \B H$, we require that $\op_{\B G} =_{\X F\big(G^1, (0, +\infty)\big)} \op_{\B H}$.
\end{definition}

The  extensionality of the relations $<, \leq$ in $\Real$ and the function-property of the assignment routine $d_x \colon X \sto \Real$, where $d_x(x{'}) := d(x, x{'})$, for every $x{'} \in X$, imply that $T_d(\B G)$ is an extensional property on $\C E^{\Disj}(X)$.
By Proposition~\ref{prp: ccpoint2}(iii) quantification on the canonical points of $\B G$ in the formula $T_d(\B G)$ and in the definition of the modulus $\op_{\B G}$ can be replaced by quantification on $G^1$.  The introduction of a modulus of openness is an instance of the  so-called ``modular'' approach within the constructive theory of metric space, according to which, functions that are induced by the axiom of choice are part of the witnessing data that accompany a fully constructive formulation of the related concept.

\begin{proposition}\label{prp: metric}
The set $\B {\C T}_{d}$-open complemented subsets of $X$ contains $(X, \emptys_X), (\emptys_X, X)$, it is closed under finite unions and arbitrary unions of set-indexed familes of $\B {\C T}_{d}$-open complemented subsets of $X$. 
\end{proposition}

\begin{proof}
	 To show $T_d(X, \emptys_{(X,d)})$,
%	where 	$$\emptys_{(X,d)} := \{y \in X \mid d(y, y) > 0\},$$
	let $\B x \cin (X, \emptys_X)$ and take e.g., $\epsilon := 1$. Clearly, we have that $[d_{x} < 1] \subseteq X$. The inclusion $\emptys_{(X,d)} \subseteq [d_{x} \geq 1]$ follows trivially in $\INT$.
	The proof though, of
	%of the inclusion
	 $T_d(\emptys_{(X,d)}, X)$ can be given within $\MIN$ as follows. 
	%	To show $T_d(\emptys_{(X,d)}, X)$, 
	Let $\B x \cin (\emptys_{(X,d)}, X)$ i.e., $x \in \emptys_X$ and $X \subseteq \{x\}^{\neq_X}$. If $\epsilon_x := \frac{d(x,x)}{2} > 0$, we show that 
	$[d_{x} < \epsilon_x] \subseteq \emptys_X$ and $X \subseteq [d_{x} \geq \epsilon_x]$. For the first inclusion, if $x{'} \in X$ with $d(x, x{'}) < \frac{d(x,x)}{2}$, then 
	$2 d(x{'},x) < d(x,x) \TOT d(x,x) - 2 d(x{'}, x) > 0.$ As 
	$$d(x,x) \leq d(x,x{'}) + d(x{'}, x) \leq d(x,x{'}) + d(x{'},x{'}) + d(x{'},x),$$
	we get $d(x,x) - 2d(x, x{'}) \leq d(x{'}, x{'})$, hence $0 < d(x,x) - 2 d(x{'}, x) \leq d(x{'}, x{'})$ i.e., $x{'} \in \emptys_{(X, d)}$. For the second inclusion, if $x{''} \in X$, the required inequality $d(x{''}, x) \geq \frac{d(x,x)}{2}$ follows immediately by the triangle inequality. 	
	Next we show that if $T_d(\B G)$ and $T_d(\B H)$, then $T_d(\B G \cap \B H)$. If $\B x \cin \B G \cap \B H$, then $\B x \cin \B G$ and $\B x \cin \B H$. Let $\epsilon_1 > 0$ with $\B B(x, \epsilon_1) \subseteq \B G$, and let $\epsilon_2 > 0$ with $\B B(x, \epsilon_2) \subseteq \B H$. If 
	$$\epsilon := \min\{\epsilon_1, \epsilon_2\} =: \epsilon_1 \wedge \epsilon_2,$$
	 then $\epsilon > 0$ too (see~\cite{BV06}, p.~57). By Remark~\ref{rem: ball1}(iii) and since
 $\epsilon \leq \epsilon_1$ and $\epsilon \leq \epsilon_2$, we have that
	$\B B(x, \epsilon) \subseteq \B B(x, \epsilon_1) \subseteq \B G$ and $\B B(x, \epsilon) \subseteq \B B(x, \epsilon_2) \subseteq \B H$, and hence by Corollary~\ref{cor: csoperations1}(i) 
	$\B B(x, \epsilon) \subseteq \B G \cap \B H$. 
	If $T_d(\B G_i)$, for every $i \in I$, and 
	$\B x \cin \bigcup_{i \in I}\B G_i $, then by Proposition~\ref{prp: ccpoint2}(vi) 
	%we have that 
	$\exists_{i \in I}(\B x \cin \B G_i)$. If $\epsilon > 0$ with $\B B(x, \epsilon) \subseteq \B A_i$, then by Corollary~\ref{cor: csoperations1}(ii) we get $\B B(x, \epsilon) \subseteq \bigcup_{i \in I}\B G_i$.	
\end{proof}

By inspection of the previous proof, we get the following moduli of openness. Notice that for the modulus of openness for $(X, \ \emptys_X)$ any constant function to $(0, +\infty)$ would do.

\begin{remark}\label{rem: modmetric1}
\normalfont (i)
\itshape  Let $\op_{(X, \ \emptys_X)} \colon X \to (0, +\infty)$ be the constant function $1$.\\[1mm]
\normalfont (ii)
\itshape Let $\op_{( \ \emptys_X, X)} \colon \emptys_X \to (0, + \infty)$ be defined by the rule $x \mapsto \frac{d(x,x)}{2}$, for every $x$ $\in$ $\emptys_{(X,d)}$.\\[1mm]
\normalfont (iii)
\itshape If $\op_{\B G}$ and $\op_{\B H}$ are  is a moduli of openness for $\B G$ and $\B H$ , respectively, let $\op_{\B G \cap \B H} := \op_{\B G} \wedge \op_{\B H}$ be a modulus of openness for $\B G \cap \B H$, where $(\op_{\B G} \wedge \op_{\B H})(x) := \op_{\B G}(x) \wedge \op_{\B H}(x)$, for every $x \in G^1 \cap H^1$.  
 \end{remark}

Notice that the assignment routine $x \mapsto d(x,x)$, where $x \in \emptys_{(X,d)}$, takes values in (or is a term of set) $(0, +\infty)$, as by hypothesis $x$ is a variable of set $\emptys_{(X,d)}$. What the axiom of metric space $\forall_{x \in X}\big(d(x,x) =_{\Real} 0\big)$ guarantees is that there will be no constant of set $\emptys_{(X,d)}$, but the assignment routine $x \mapsto d(x,x)$ is only a rule that transforms a variable of set $\emptys_{(X,d)}$ to a term of set $(0, +\infty)$, independently from the fact that domain of this rule cannot be inhabited by some constant. 
 To determine a modulus of openness for the union $\bigcup_{i \in I}\B G_i$ with $\op_{\B G_i}$ a modulus of openness for $\B G_i$, for every $i \in I$, we need to add more information to the union. One can add, for example, a projection function $\pr \colon \bigcup_{i \in I}G^1_i \to I$, such that $x \in G_{\pr(x)}^1$, for every  $x \in \bigcup_{i \in I}G^1_i $. In this way the original union $\bigcup_{i \in I}G^1_i $ is equal to the disjoint union $\bigcup_{x \in \bigcup_{i \in I}G_i^1}G_{\pr(x)}$. And then one can define $\op_{\bigcup_{i \in I}G^1_i }(x) := \op_{G_{\pr(x)}}(x)$, for every $x \in \bigcup_{i \in I}G^1_i$.
Another approach will be presented in Remark~\ref{rem: modunionmetric}.

\begin{definition}\label{def: dtop}
Let $\B {\C T}_{d} := \{\B G \in \C E^{\Disj}(X) \mid T_d(\B G)\}$ be the canonical $\cs$-topology  on $X$ induced by
 %the metric 
 $d$, and  
 $$\B B(X) := \{\B G \in \C E^{\Disj}(X) \mid \exists_{x{'} \in X}\exists_{\epsilon{'} > 0}\big(\B G =_{\C E^{\Disj}(X)} \B B(x{'}, \epsilon{'})\}$$
  the set of complemented balls of $X$. If $\B G$ is an arbitrary open complemented subset in $\B {\C T}_d$, we call a function $\beta_{\B G} \colon G^1 \to \B B(X)$ satisfying $\B y \cin \beta_{\B G}(y)$ and $\beta_{\B B(x, \epsilon)}(y) \subseteq \B B(x, \epsilon)$, for every $y \in G^1$, 
   a base-module for $\B G$, or a $\B B(X)$-module for $\B G$. We call such a base-module covering, if
   $$\B G =_{\C E^{\Disj}(X)} \bigcup_{y \in G^1}\beta_{\B G}(y).$$ 
    A  family 
  % of functions
    $(\beta_{\B G})_{\B G \in \B {\C T}_d}$ with $\beta_{\B G}$ a $\B B(X)$-module for $\B G$, for every $\B G \in \B {\C T}_d$, is $\B B(X)$-module for $\B {\C T}_d$. If $\op_{\B G}$ is a module of openness for $\B G$, the canonical $\B B(X)$-module for $\B G$ induced by $\op_{\B G}$ is defined by $\beta_{\B G}(x) := \B B(x, \op_{\B G}(x))$, for every $x \in G^1$.
\end{definition}

Clearly, a base-modulus $\beta_{\B G}$ for $\B G$ defines a module of openness $\op_{\B G}$ for $\B G$ by taking the radius of the corresponding complemented ball i.e., the two concepts are equivalent for metric spaces. 

\begin{remark}\label{rem: ball1}
\normalfont (i)
\itshape The complemented ball $\B B(x, \epsilon)$ is a $\B {\C T}_{d}$-open complemented subset. \\[1mm]
\normalfont (ii)
\itshape The canonical $\B B(X)$-module 
%$\beta_{\B B(x, \epsilon)} \colon [d_x < \epsilon] \to \B B(X)$ 
for $\B B(x, \epsilon)$ 
%, which is induced by the canonical modulus of openness for $\B B(x, \epsilon)$ and it 
is covering.
\end{remark}

\begin{proof}
	(i) If $y \in [d_x < \epsilon]$, let $\op_{\B B(x, \epsilon)}(y) := \epsilon - d(x, y) > 0$. For $\B B(y, \op_{\B B(x, \epsilon)}(y))$
	 %:= \big([d_y < \epsilon_y], [d_y \geq \epsilon_y]\big)$.
	  we have that 
	$$[d_y < \op_{\B B(x, \epsilon)}(y)] \subseteq [d_x < \epsilon] \ \ \& \ \ [d_x \geq \epsilon] \subseteq [d_y \geq \op_{\B B(x, \epsilon)}(y)].$$
	For the first inclusion we have that, if $d(y, z) < \op_{\B B(x, \epsilon)}(y)$, then $d(x, z) \leq d(x, y) + d(y, z) < d(x, y) +  \epsilon - d(x, y) < \epsilon$, and for the second inclusion, if $d(x, z) \geq \epsilon$, then  $\epsilon - d(x, y) \leq d(x, z) - d(x, y) \leq d(y, z)$.\\
	(ii) Let 
	%the \textit{canonical} $\B B(X)$-module for $\B B(x, \epsilon)$ be
	 the assignment routine $y \mapsto \beta_{\B B(x, \epsilon)}(y) := \op_{\B B(x, \epsilon)}(y)$, which is a function, as
	% the metric 
	 $d$ is a function. Clearly, $\B y \cin  \B B(y, \op_{\B B(x, \epsilon)}(y))$, and 
	 by case (i) 
%	we have that 
	$ \B B(y, \op_{\B B(x, \epsilon)}(y)) \subseteq \B B(x, \epsilon)$.
	%(iii) 
	By definition of the canonical base-module for $ \B B(x, \epsilon)$ we have that
	$$\bigcup_{y \in [d_x < \epsilon]}\beta_{\B B(x, \epsilon)}(y) := \bigcup_{y \in [d_x < \epsilon]}\B B(y, \op_{\B B(x, \epsilon)}(y)) := \bigg(\bigcup_{y \in [d_x < \epsilon]}[d_y < \op_{\B B(x, \epsilon)}(y)], \bigcap_{y \in [d_x < \epsilon]}[d_y \geq \op_{\B B(x, \epsilon)}(y)]\bigg).$$
	If $y \in [d_x < \epsilon]$, then 
	%clearly 
	$[d_x < \epsilon] =_{\C E(X)} \bigcup_{y \in [d_x < \epsilon]}[d_y < \op_{\B B(x, \epsilon)}(y)]$, and as by (i) $[d_y < \op_{\B B(x, \epsilon)}(y)] \subseteq [d_x < \epsilon]$ and $[d_x \geq \epsilon] \subseteq [d_y \geq \op_{\B B(x, \epsilon)}(y)]$, we get
	%also have that 
	$[d_x \geq \epsilon] \subseteq \bigcap_{y \in [d_x < \epsilon]}[d_y \geq \op_{\B B(x, \epsilon)}(y)].$
	Next we show the converse inclusion
	$$\bigcap_{y \in[d_x < \epsilon]}[d_y \geq \op_{\B B(x, \epsilon)}(y)] \subseteq [d_x \geq \epsilon].$$
	Let $z \in \bigcap_{y \in[d_x < \epsilon]}[d_y \geq \op_{\B B(x, \epsilon)}(y)]$ and suppose that $d(x,z) < \epsilon$ i.e., $z \in [d_x < \epsilon]$. By the hypothesis on $z$ we 
	%then 
	get $z \in [d_z \geq \op_{\B B(x, \epsilon)}(z)]$ i.e., $d(z, z) \geq \op_{\B B(x, \epsilon)}(z) > 0$, which is absurd, hence $d(x, z) \geq \epsilon$.
\end{proof}

\begin{remark}\label{rem: expl1}
	\normalfont
If $\B G$ is an element of $\B {\C T}_d$ and $\beta_{\B G} \colon G^1 \to \B B(X)$ is a base-module for $\B G$, then in order to show that the family of complemented balls 
 $\big(\beta_{\B G}(y))_{y \in G^1}$ covers $\B G$ i.e., 
$$\B G =_{\C E^{\Disj}(X)} \bigcup_{y \in G^1}\beta_{\B G}(y) := \bigg(\bigcup_{y \in G^1}[d_y < \op_{\B G}(y)], \bigcap_{y \in G^1}[d_y \geq \op_{\B G}(y)]\bigg),$$
we only need to show, as all other inclusions follow easily, the so-called  here \textit{covering-inclusion} for $\beta_{\B G}$:
$$(\CI_{\B G}) \ \ \ \ \ \ \ \ \ \ \ \ \ \ \ \ \ \ \ \ \ \ \ \ \ \ \ \ \ \ \ \ \ \ \ \ \ \ \ \ \ \ \ \ \ \ \ \  \bigcap_{y \in G^1}[d_y \geq \op_{\B G}(y)] \subseteq G^0. \ \ \ \ \ \ \ \ \ \ \ \ \ \ \ \ \ \ \ \ \ \ \ \ \ \ \ \ \ \ \ \ \ \ \ \ \ \ \ \ \ \ \ \ \ \ \ \ \ \ \ \ \ \ \ \ \  \ \ \ \ \ \ \ \ \ \ \ \ \ \ \ \ \ $$
%as all other inclusions follow easily. 
%We call this inclusion the, and 
If $(\CI_{\B G})$ holds, then we say that $\beta_{\B G}$ is a \textit{covering} base-module for $\B G$. 
\end{remark}

\begin{remark}\label{rem: expl2}
	\normalfont
Even if $\beta_{\B G}$ and $\beta_{\B H}$ are  covering base-modules for $\B G$ and $\B H$, respectively, we cannot show, in general, that these modules induce a covering base-module for $\B G \cap \B H$ by the method generated in the proof of Proposition~\ref{prp: metric}. Suppose that  
$$\B H =_{\C E^{\Disj}(X)} \bigcup_{z \in H^1}\beta_{\B H}(z) := \bigg(\bigcup_{z \in H^1}[d_z < \op_{\B H}(z)], \bigcap_{z \in H^1}[d_z \geq \op_{\B H}(z)]\bigg).$$
In order to show the crucial inequality 
$$\bigcap_{w \in G^1 \cap H^1}[d_w \geq \op_{\B G}(w) \wedge \op_{\B H}(w)] \subseteq G^0 \cup H^0,$$
it suffices that $\B G \cap \B H$ is $1$-tight i.e., $\neg{(x \in G^1 \cap H^1)} \To x \in G^0 \cup H^0$. To show this, suppose $k \in \bigcap_{w \in G^1 \cap H^1}[d_w \geq \op_{\B G}(w) \wedge \op_{\B H}(w)]$ and $k \in x \in G^1 \cap H^1$. Then $k \in [d_w \geq \op_{\B G}(w) \wedge \op_{\B H}(w)] \TOT d(k,k) \geq \op_{\B G}(k) \wedge \op_{\B H}(k) > 0$, which is absurd, hence $k \in G^0 \cup H^0$.
\end{remark}

As we show next, the covering inclusion holds for many 
%open complemented subsets in 
elements of $\B {\C T}_d$.
For the case of intersection of two complemented balls in Proposition~\ref{prp: basemodule1}(iii) we need case (ii) of the following lemma, while for the union of two complemented balls we need case (iv) of the following lemma.

\begin{lemma}\label{lem: neglemma}
If $x, x{'}, y, y{'}\in \Real$, the following hold:\\[1mm]
\normalfont (i)
\itshape $\neg{(x < 0 \wedge y < 0)} \To (x \geq 0 \vee y \geq 0)$.\\[1mm]
\normalfont (ii)
\itshape $\neg{(x < y \wedge x{'} < y{'})} \To (x \geq y \vee x{'} \geq y{'})$.\\[1mm]
\normalfont (iii)
\itshape $\neg{(x < 0 \vee y < 0)} \To (x \geq 0 \wedge y \geq 0)$.\\[1mm]
\normalfont (iv)
\itshape $\neg{(x < y \vee x{'} < y{'})} \To (x \geq y \wedge x{'} \geq y{'})$.
\end{lemma}

\begin{proof}
(i) For all properties on the order of reals used in this proof we refer to~\cite{BV06}, p. 57, Ex.~3. We have that $x < 0 \wedge y < 0 \TOT x \vee y := \max\{x, y\} < 0$, hence $\neg{(x < 0 \wedge y < 0)} \To \neg{(x \vee y  < 0)}$. By Remark~\ref{rem: order1} we get $x \vee y \geq 0$, which in turn implies $x \geq 0$ or $y \geq 0$.\\
(ii) The hypothesis is equivalent to $\neg{(x - y < 0 \wedge x{'} -  y{'} < 0)}$, and we use case (i).\\
(iii) We suppose $x < 0$, and we get $x < 0 \vee y < 0$, hence $\bot$. We get $x \geq 0$ by Remark~\ref{rem: order1}. We work similarly, if we suppose $y < 0$.\\
(iv) The hypothesis is equivalent to $\neg{(x - y < 0 \vee x{'} -  y{'} < 0)}$, and we use case (iii).
\end{proof}

\begin{proposition}\label{prp: basemodule1}
Let $\C I := (I, =_I, \neq_I) \in \SetExtIneq$, such that $i \neq_I j$ is discrete $(\Ineq_6)$, and 
let 
$\B G_i$ 
be $\B {\C T}_d$-open, every $i \in I$, such that $\B G_i \cap \B G_j \subseteq (\emptys_X, X)$, for every $i, j \in I$ with $i \neq_I j$.\\[1mm]
\normalfont (i)
\itshape The canonical $\B B(X)$-modules for $(X, \emptys_X)$ and $(\emptys_X, X)$ are covering.\\[1mm]
%\normalfont (ii)
%\itshape If the $\B B(X)$-modules for $\B G$ and $\B H$ are covering, then it is also covering for $\B G \cup \B H$.\\[1mm]
\normalfont (ii)
\itshape  $($Myhill's unique choice principle$)$ If for every $i \in I$ the canonical $\B B(X)$-module for $\B G_i$ is covering, then it is 
%also 
covering for $\bigcup_{i \in I}\B G_i$.\\[1mm]
\normalfont (iii)
\itshape The canonical $\B B(X)$-module for $\B B(x, \epsilon) \cap \B B(y, \delta)$ is covering.\\[1mm]
\normalfont (iv)
\itshape The canonical modulus of opennes for $\B B(x, \epsilon) \cup \B B(y, \delta)$ is defined by the rule
$$\op_{\B B(x, \epsilon) \cup \B B(y, \delta)}(z) := [\epsilon -d(x,z)] \vee [\delta - d(y,z)] =: \eta_{\zeta},$$
for every $z \in [d_x < \epsilon] \cup [d_y < \delta]$, and the canonical $\B B(X)$-module for $\B B(x, \epsilon) \cup \B B(y, \delta)$ is covering.\\[1mm]
\normalfont (v)
\itshape A canonical complemented copoint $(x^{\neq_{(X,d)}}, x)$ is $\B {\C T}_d$-open with a covering $\B B(X)$-module for it.
\end{proposition}

\begin{proof}
(i) $(\MIN)$ We need to show the covering-inclusion $\bigcap_{x \in X}[d_x \geq 1] \subseteq \emptys_X$. Let $y \in X$ such that $d(x, y) \geq 1$, for every $x \in X$. Consequently, $d(y, y) \geq 1 > 0$, hence $y \in \emptys_X$. For the case of $(\emptys_X, X)$, the proof of the covering-inclusion $\bigcap_{y \in \ {\emptys_X}}[d_y \geq \frac{d(y,y)}{2}] \subseteq X$ is shown in the proof of Proposition~\ref{prp: metric}. For the following left inclusion
$$\bigcup_{y \in \ {\emptys_X}}\bigg[d_y < \frac{d(y,y)}{2}\bigg] \subseteq {\emptys_X} \ \ \& \ \ \emptys_X \subseteq \bigcup_{y \in \ {\emptys_X}}\bigg[d_y < \frac{d(y,y)}{2}\bigg],$$ 
we use the fact shown in the proof of Proposition~\ref{prp: metric} that $[d_y < \frac{d(y,y)}{2}] \subseteq \emptys_X$, for every $y \in \emptys_X$. For the second inclusion, each element $y$ of $\emptys_X$ is in $[d_y < \frac{d(y,y)}{2}]$, and hence in their union.\\
(ii) For every $i \in I$ let the covering-inclusion $\bigcap_{y_i \in G_i^1}[d_{y_i} \geq \op_{\B G_i}(y_i)] \subseteq G_i^0$. We define $\op_{\bigcup} \colon \bigcup_{i \in I}G_i^1 \to (0, +\infty)$ by $\op_{\bigcup}(x) := \op_{\B G_i}(x)$, where $i$ is the unique (up to equality) element of $I$ sith $x \in G_i^1$. If $x \in G_j^1$ with $G_i^1 \cap G_j^1 \subseteq \emptys_X$, then we get a contradiction i.e., $\neg{(i \neq_I j)}$, hence by $(\Ineq_6)$ $i =_I j$. By our requirement in Definition~\ref{def: canonicalcs} we get $\op_{\B G_i}(x) =_{\Real} \op_{\B G_j}(x)$. The required covering-inclusion is $$\bigcap_{z  \in \bigcup_{i \in I}G_i^1}[d_{z} \geq \op_{\bigcup}(z)] \subseteq \bigcap_{i \in I}G_i^0.$$
If $k \in \bigcap_{z  \in \bigcup_{i \in I}G_i^1}[d_{z} \geq \op_{\bigcup}(z)] $, then for every $z \in G_i^1$ we get $d(z, k) \geq \op_{\B G_i}(z)$, hence $k \in G_i^0$. As $i \in I$ is arbitray, the required covering-inclusion follows.\\
(iii) The covering-inclusion for $\B B(x, \epsilon) \cap \B B(y, \delta) = \big([d_x < \epsilon] \cap [d_y < \delta], [d_x \geq \epsilon] \cup [d_y \geq \delta]$ is
$$\bigcap_{v \in [d_x < \epsilon] \cap [d_y < \delta]}[d_v \geq \op_{\B B(x, \epsilon)}(v) \wedge \op_{\B B(y, \delta)}(v)] \subseteq [d_x \geq \epsilon] \cup [d_y \geq \delta].$$
Let $k \in \bigcap_{v \in [d_x < \epsilon] \cap [d_y < \delta]}[d_v \geq \op_{\B B(x, \epsilon)}(v) \wedge \op_{\B B(y, \delta)}(v)]$, and suppose that $k \in [d_x < \epsilon] \cap [d_y < \delta] \TOT d(x, k) < \epsilon \wedge d(y, k) < \delta$. Then $k \in [d_k \geq \op_{\B B(x, \epsilon)}(k) \wedge \op_{\B B(y, \delta)}(k)]$ i.e., $d(k, k) \geq \op_{\B B(x, \epsilon)}(k) \wedge \op_{\B B(y, \delta)}(k) > 0$, which is impossible. Hence, $\neg{\big(d(x, k) < \epsilon \wedge d(y, k) < \delta\big)}$. By Lemma~\ref{lem: neglemma}(ii) we get the required disjunction $d(x, k) \geq \epsilon$ or $d(y, k) \geq \delta$.\\
(iv) As $z \in [d_x < \epsilon] \cup [d_y < \delta]$, we have that $\epsilon - d(x,z) > 0$ or $\delta - d(y,z) > 0$, hence $\eta_z > 0$. First we show that $[d_z < \eta_z] \subseteq [d_x < \epsilon] \cup [d_y < \delta]$. Let $w \in X$ with $d(z, w) < \eta_z$. By the property of real numbers $a \vee b > c \TOT a > c \vee b > c$ (see~\cite{BV06}, p.~57) we get
$$d(z, w) < \epsilon - d(x,z) \  \vee \  d(z, w) < \delta - d(y,z).$$
In the first case we have that $d(x, w) \leq d(x, z) + d(z,w)  < \epsilon$, and hence $w \in [d_x < \epsilon]$, while in the second we have that $d(y, w) \leq d(y, z) + d(z,w)  < \delta$, and hence $w \in [d_y < \delta]$. Notice that for this inclusion we didn't use the infromation $z \in [d_x < \epsilon]$ or $z \in [d_y < \delta]$. Next we show that $[d_x < \epsilon] \cap [d_y < \delta] \subseteq [d_z \geq \eta_z]$.
Let $u \in [d_x < \epsilon] \cap [d_y < \delta]$, and we suppose that $d(z,u) < \eta_z$.  By the same argument $d(z, u) < \epsilon - d(x,z)$ or $d(z, u) < \delta - d(y,z)$, which implies $d(x, u) < \epsilon$ or $d(y, u) < \delta$. In both cases we get a contradiction, hence $d(z, u) \geq \eta_z$. To show that the 
canonical $\B B(X)$-module for $\B B(x, \epsilon) \cup \B B(y, \delta)$ is covering, it suffices to show the inclusion 
$$\bigcap_{z \in [d_x < \epsilon] \cup [d_y < \delta]}[d_z \geq \eta_z] \ \subseteq \ [d_x < \epsilon] \cap [d_y < \delta].$$
Let $k \in \bigcap_{z \in [d_x < \epsilon] \cup [d_y < \delta]}[d_z \geq \eta_z] $. We show that 
$$\neg{\big[d(x,k) < \epsilon \vee d(y,k) < \delta\big]}.$$
If $d(x,k) < \epsilon \vee d(y,k) < \delta$, then by our hypothesis in $k$ we get $d(k,k) \geq \eta_k > 0$, which is a contradiction. Hence by Lemma~\ref{lem: neglemma}(iv) we get the required conjunction $d(x, k) \geq \epsilon$ and $d(y, k) \geq \delta$.\\
(v) Let the modulus of opennes $\op_{(x^{\neq_{(X,d)}}, x)} \colon x^{\neq_{(X,d)}} \to (0, +\infty)$, defined by the rule $y \mapsto \frac{d(x,y)}{2}$. We show that if $z \in X$ with $d(z, y) <  \frac{d(x,y)}{2}$, then $d(z, x) > 0$. As $0 < d(x, y) \leq d(x,z) + d(y,z)$, we get
$$d(x,z) \geq d(x,y) - d(y,z) > d(x, y) - \frac{d(x,y)}{2} = \frac{d(x,y)}{2} > 0.$$
The corresponding covering inclusion for this copoint is 
$$\bigcap_{y \in x^{\neq_{(X,d)}}} \bigg[d_y \geq \frac{d(x,y)}{2}\bigg] \subseteq \{x\}.$$
Let $k \in \bigcap_{y \in x^{\neq_{(X,d)}}} \big[d_y \geq \frac{d(x,y)}{2}\big]$. If  $d(k, x) > 0$, then by our hypothesis on $k$ we have that $k \in \big[d_k \geq \frac{d(x,k)}{2}\big] \TOT d(k, k) \geq \frac{d(x,k)}{2} > 0$, which is absurd. Hence, $d(x, k) \leq 0 \TOT d(x, k) = 0$, and hence $x =_{X} k$.
\end{proof}

Notice that in Proposition~\ref{prp: basemodule1}(iv) the modulus of openness $\op_{\B B(x, \epsilon) \cup \B B(y, \delta)}$ is defined by the rule 
$$\op_{\B B(x, \epsilon) \cup \B B(y, \delta)}(z) := \op^{\cup}_{\B B(x, \epsilon)}(z) \vee \op^{\cup}_{\B B(y, \delta)}(z),$$
where $\op^{\cup}_{\B B(x, \epsilon)} \colon X \to \Real$ is defined by the rule $w \mapsto \epsilon - d(x,w)$, and $\op^{\cup}_{\B B(y, \delta)} \colon X \to \Real$ is defined by the rule $w \mapsto \delta - d(y,w)$. I.e., $\op^{\cup}_{\B B(x, \epsilon)}$ and $\op^{\cup}_{\B B(y, \delta)}$ are the extensions of $\op_{\B B(x, \epsilon)}$ and $\op_{\B B(y, \delta)}$ to $X$, respectively. Of course, it suffices to consider their extensions to $[d_x < \epsilon] \cup [d_y < \delta]$. This description of the modulus of openness for the union of two complemented balls motivates the following approach to the modulus of openness for the union of two open complemented subsets of a metric space.

\begin{remark}\label{rem: modunionmetric}
Let $\B G$ and $\B H$ in $\B {\C T}_d$ with moduli of openness $\op_{\B G}$ and $\op_{\B H}$, respectively, such that:\\[1mm]
%\]atisfying the following conditions:\\[1mm]
\normalfont (a) 
\itshape There exist extensions $\op_{\B G}^{\cup} \colon G^1 \cup H^1 \to \Real$ and $\op_{\B H}^{\cup} \colon G^1 \cup H^1 \to \Real$ of $\op_{\B G}$ and $\op_{\B H}$, respectively.\\[1mm]
\normalfont (b) 
\itshape If $z \in G^1 \cup H^1$, then $[d_z < \op_{\B G}^{\cup}(z) \vee \op_{\B H}^{\cup}(z)] \subseteq G^1 \cup H^1$.\\[1mm]
\normalfont (c) 
\itshape $\B G \cup \B H$ is $1$-tight i.e., $\neg{(k \in G^1 \cup H^1)} \To k \in G^0 \cap H^0$.\\[1mm]
Then $\op_{\B G \cup \B H} := \op_{\B G}^{\cup} \vee \op_{\B H}^{\cup}$ is a modulus of openness for $\B G \cup \B H$, 
%and its 
with a covering 
%and its
% induced 
%canonical 
$\B B(X)$-module 
%is covering.
\end{remark}
	
\begin{proof}
The proof is similar, actually simpler, to the proof of Proposition~\ref{prp: basemodule1}(iv).	
\end{proof}

Similarly, if $(\B G_i)_{i \in I}$ is an $I$-family in $\B {\C T}_d$, a modulus of openness for their union is defined
$$\op_{\bigcup_{i \in I}\B G_i } := \bigvee_{i \in I}\op_{i}^{\bigcup},$$ 
with a covering $\B B(X)$-module, if the given moduli of openness $(\op_i)_{i \in I}$ satisfy the following conditions:\\[1mm]
(a) There exist extensions $\op_{i}^{\bigcup} \colon \bigcup_{i \in I}G^1_i \to \Real$ of $\op_i$, for every $i \in I$.\\[1mm]
(b) If $z \in \bigcup_{i \in I}G^1_i$, then
$$\bigvee_{i \in I}\op_{i}^{\bigcup}(z) := \sup \big\{\op_{i}^{\bigcup}(z) \mid i \in I\big\} \in \Real \ \ \ \& \ \ \ \bigg[d_z < \bigvee_{i \in I}\op_{i}^{\bigcup}(z)\bigg] \subseteq 
\bigcup_{i \in I}G^1_i.$$
(d) $\bigcup_{i \in I}\B G_i $ is $1$-tight i.e., $\neg{\big(k \in \bigcup_{i \in I}G^1_i\big)} \To k \in \bigcap_{i \in I}G^0_i$.

\section{Pointwise and uniformly continuous functions between metric spaces}
\label{sec: csmetriccont}

First we present continuity of functions between metric spaces in the standard constructive way, and then we explain how the moduli of pointwise and uniform continuity are related to the moduli of openness.

\begin{definition}\label{def: mscont}
If $(X, d)$ and $(Y, e)$ are metric spaces, a function $\fXY$ is pointwise contninuous at $x_0 \in X$ with modulus of pointwise continuity $\omega_{f, x_0} \colon (0 + \infty) \to (0 + \infty)$, if
$$\forall_{\epsilon > 0}\forall_{x \in X}\big(d(x, x_0) < \omega_{f, x_0}(\epsilon) \To e(f(x), f(x_0)) < \epsilon\big).$$
If there is a function 
$$\omega_f \colon X \to \X F\big((0, +\infty), (0, +\infty)\big), \ \ \ x \mapsto \omega_{f, x},$$
 such that $\omega_{f, x}$ is a modulus of pointwise continuity of $f$ at $x$, for every $x \in X$, then
 $f$ is called pointwise continuous on $X$ with a modulus of pointwise continuity $\omega_f$. 
Function $f$ is $($uniformly$)$ continuous with a modulus of $($uniform$)$ continuity $\Omega_f \colon (0 + \infty) \to (0 + \infty)$, if 
$$\forall_{\epsilon > 0}\forall_{x, x{'} \in X}\big(d(x, x{'}) < \Omega_f(\epsilon) \To e(f(x), f(x{'}) < \epsilon\big).$$
We may denote a pointwise continuous function with a modulus of pointwise continuity by the pair $(f, \omega_f)$, and a uniformy continuous function with a modulus of uniform continuity by the pair $(f, \Omega_f)$. Of course, equality between such pairs is reduced to the standard equality of functions. 
\end{definition}

\begin{remark}\label{rem: secont}
	\normalfont
%If $(X, d)$ and $(Y, e)$ are metric spaces, 
A pointwise continuous function $(\fXY, \omega_f)$ between metric spaces is strongly extensional.
\end{remark}

\begin{proof}
Let $e(f(x), f(x{'})) > 0$, for some $x, x{'} \in X$. By pointwise continuity of $f$ at $x$ we have that  
$\forall_{\epsilon > 0}\big(d(x, x{'}) < \omega_{x}(\epsilon) \To e(f(x), f(x{'})) < \epsilon\big)$. If $\epsilon := e(f(x), f(x{'})) > 0$, then by the contraposition of the previous implication and the equivalence $\neg {(a < b)} \TOT a \geq b$, where $a, b \in \Real$, we get 
$e(f(x), f(x{'})) \geq \epsilon \To d(x, x{'}) \geq \omega_{x}(\epsilon) > 0.$
As the hypothesis holds trivially, we get 
%the required the conclusion 
$d(x, x{'}) >0$.
% is the required inequality $x \neq_{(X, d)} x{'}$.
\end{proof}

The following facts are straightforward to show.

\begin{proposition}\label{prp: compcont}
Let $(X, d), (Y, e), (Z, \rho)$ be metric spaces, $\fXY$ and $g \colon Y \to Z$.\\[1mm]
\normalfont (i) 
\itshape The identity $\id_X \colon X \to X$ is $($pointwise$)$ uniformly continuous with modulus of $($pointwise$)$ uniform continuity $(\omega_{\id_X})$ $\Omega_{\id_X}$, given by the rule $(\omega_{\id_X, x} := \id_{[0, +\infty)}$, for every $x \in X)$ $\Omega_{\id_X} :=  \id_{[0, +\infty)}$.\\[1mm]
\normalfont (ii) 
\itshape If $\omega_f$ and $\omega_g$ are moduli of pointwise continuity for $f$ and $g$, respectively, then the map $\omega_{g \circ f}$, defined by the rule
$\omega_{g \circ f, x} := \omega_{f, x} \circ \omega_{g, f(x)},$
for every $x \in X$, is  a modulus of pointwise continuity for $g \circ f$.\\[1mm]
\normalfont (iii) 
\itshape If $\Omega_f,$ and $\Omega_g$ are moduli of uniform continuity for $f$ and $g$, respectively, then the map $\Omega_{g \circ f}$, defined by the rule
$\Omega_{g \circ f} := \Omega_{f} \circ \Omega_{g},$
is  a modulus of uniform continuity for $g \circ f$.
\end{proposition}

The following categories of metric spaces are subcategories of $(\SetExtIneq, \StrExtFun)$.
\begin{definition}\label{def: metriccats}
Let $(\Metr, \pContMod)$ be the category of metric spaces and pointwise continuous functions with a modulus of pointwise continuity, and let its subcategory $(\Metr, \uContMod)$ of metric spaces and uniformly continuous functions with a modulus of uniform continuity.	
\end{definition}

Next we explain how the moduli of pointwise continuity are related to moduli of openness.

\begin{proposition}\label{prp: pcont1}
Let $(X, d)$ and $(Y, e)$ be metric spaces and $\fXY$. If $f$ is pointwise continuous on $X$ with a modulus of pointwise continuity $\omega_{f}$, then $f$ inverses open complemented subsets and their moduli of openness i.e., $f^{-1}(\B H) \in \B {\C T}_d$, for every $\B H \in \B {\C T}_e$, and if $\op_{\B H}$ is a modulus of openness for $\B H$, then the assignment routine $\op_{f^{-1}(\B H)} \colon  f^{-1}(H^1) \sto (0, +\infty)$, defined by the rule
$$\op_{f^{-1}(\B H)}(x^1) := \omega_{f,x^1}\big(\op_{\B H}(f(x^1))\big),$$
for every $x^1 \in f^{-1}(H^1) $, is a modulus of openness for $f^{-1}(\B H)$, and the following conditions hold:\\[1mm]
\normalfont (*)
\itshape If $x^1 \in f^{-1}(H^1)$, then $\op_{f^{-1}(\B H)}(x^1) =_{\Real} \op_{f^{-1}(\B B(f(x^1), \op_{\B H}(f(x^1))))}(x^1)$.\\[1mm]
\normalfont (i)
\itshape 
$\op_{\id_X^{-1}(\B G)} =_{\X F(G^1, (0, +\infty)])} \op_{\B G}$, for every $\B G \in \B {\C T}_{d}$.\\[1mm]
\normalfont (ii)
\itshape $\op_{(g \circ f)^{-1}(\B K)} =_{\X F(f^{-1}(g^{-1}(K^1)), (0, +\infty)])} \op_{f^{-1}(g^{-1}(\B K))}$, where $g \colon (Y, e) \to (Z, \rho)$ is pointwise continuous with modulus of pointwise continuity $\omega_g$, and $\B K$ is an arbitrary open complemented subset in $\B {\C T}_{\rho}$. 
\end{proposition}

\begin{proof}
%Let $f$ be pointwise continuous on $X$ with a modulus of pointwise continuity $\omega_{f}$.
 By Remark~\ref{rem: secont} $f$ is strongly extensional and the inverse image of a complemented subset under $f$ is well-defined. If $\B H \in \B {\C T}_e$, and if $x^1 \in f^{-1}(H^1) \TOT f(x^1) \in H^1$, then
if 
$$\epsilon := \op_{\B H}(f(x^1)) > 0,$$
 we have that $[e_{f(x^1)} < \epsilon] \subseteq H^1$ and $H^0 \subseteq [e_{f(x^1)} \geq \epsilon]$. We show that $\B B_d(x^1, \omega_{f,x^1}(\epsilon)) \subseteq f^{-1}(\B H)$ i.e., 
$$[d_{x^1} < \omega_{f,x^1}(\epsilon)] \subseteq f^{-1}(H^1) \ \ \& \ \  f^{-1}(H^0) \subseteq 
[d_{x^1} \geq \omega_{f,x^1}(\epsilon)].$$
For the first inclusion, let $x \in X$ with $d(x^1, x)  <  \omega_{f,x^1}(\epsilon)$. By pointwise continuity of $f$ at $x^1$ we get $e(f(x^1), f(x)) < \op_{\B H}(f(x^1))$ hence $f(x) \in H^1 \TOT x \in f^{-1}(H^1)$. For the second inclusion, let $x^0 \in f^{-1}(H^0) \TOT f(x^0) \in H^0$. 
To show $d(x^1, x^0) \geq \omega_{f,x^1}(\epsilon)$, suppose that $d(x^1, x^0) < \omega_{f,x^1}(\epsilon)$. By pointwise continuity of $f$ at $x^1$ we get  $e(f(x^1), f(x^0)) < \epsilon$, hence $f(x^0) \in H^1$, which contradicts the fact that $H^1 \Disj H^0$. Hence  $d(x^1, x^0) \geq \omega_x(\epsilon)$. The 
%above defined 
assignment routine $\op_{f^{-1}(\B H)} \colon  f^{-1}(H^1) \sto (0, +\infty)$
%$$\op_{f^{-1}(\B H)}(x^1) := \omega_{f,x^1}\big(\op_{\B H}(f(x^1))\big),$$
is a function, as $\omega_f$, $\op_{\B H}$ and $f$ are functions. Clearly, $\op_{f^{-1}(\B H)}$ is a modulus of openness for $f^{-1}(\B H)$. \\
(*) If $x^1 \in f^{-1}(H^1)$, then by the definition we have that
\begin{align*}
\op_{f^{-1}(\B B(f(x^1), \op_{\B H}(f(x))))}(x^1) & := \omega_{f,x^1}\big(\op_{\B B(f(x^1), \op_{\B H}(f(x^1)))}(f(x^1))\big)\\
& := \omega_{f,x^1}\big(\op_{\B H}(f(x^1) - e(f(x^1), f(x^1))))\big)\\
& =_{\Real} \omega_{f,x^1}\big(\op_{\B H}(f(x^1))\big)\\
& =: \op_{f^{-1}(\B H)}(x^1).
\end{align*}
(i) If $x^1 \in G^1$, then by Proposition~\ref{prp: compcont}(i) for every $x^1 \in G^1$ we have that
$$\op_{\id_X^{-1}(\B G)}(x^1) := \omega_{\id_X, x^1}(\op_{\B G}(\id_X(x^1))) := \op_{\B G}(x^1).$$
(ii) If $x^1 \in f^{-1}(g^{-1}(K^1))$, then by Proposition~\ref{prp: compcont}(ii) we have that
\begin{align*}
\op_{(g \circ f)^{-1}(\B K)}(x^1) & := \omega_{g \circ f, x^1}\big(\op_{\B K}(g(f(x^1)))\big)\\
& := \big(\omega_{f, x^1} \circ \omega_{g, f(x^1)}\big)\big(\op_{\B K}(g(f(x^1)))\big)\\
& := \omega_{f, x^1}\bigg(\omega_{g, f(x^1)}\big(\op_{\B K}(g(f(x^1)))\big)\bigg)\\
& =:   \omega_{f, x^1}\big(\op_{g^{-1}(\B K)}(f(x^1))\big)\\
& =:  \op_{f^{-1}(g^{-1}(\B K))}(x^1). \qedhere
\end{align*}
%If $y_0 \in Y$, we define $\op_{f^{-1}(\B B(y_0, \epsilon))} \colon \{x \in X \mid e(f(x), y_0) < \epsilon\} \to [0, + \infty)$ by the rule
%$$x_0 \mapsto \op_{f^{-1}(\B B(y_0, \epsilon))}(x_0) := \omega_{x_0}(\epsilon - e(y_0, f(x_0))) =: \omega_{x_0}\big(\op_{\B B(y_0, \epsilon)}(f(x_0))\big)$$
%%i.e., 
%%$$\op_{f^{-1}(\B B(y_0, \epsilon))} := \omega_{x_0} \circ \op_{\B B(y_0, \epsilon)} \circ f.$$
%We show that $\B B(x_0, \op_{f^{-1}(\B B(y_0, \epsilon))}(x_0)) \subseteq f^{-1}(\B B(y_0, \epsilon))$ i.e.,
%$$[d_{x_0} < \omega_{x_0}(\epsilon - e(y_0, f(x_0)))] \subseteq \{x \in X \mid e(f(x), y_0) < \epsilon\}$$
%and 
%$$\{x{'} \in X \mid e(f(x{'}), y_0) \geq \epsilon\} \subseteq [d_{x_0} \geq \omega_{x_0}(\epsilon - e(y_0, f(x_0)))].$$
%For the first inclusion, let $z \in X$ with $d(x_0, z) < \omega_{x_0}(\epsilon - e(y_0, f(x_0)))$. By continuity of $f$ at $x_0$ we get 
%\begin{align*}
%e(f(z), y_0) & \leq e(f(z), f(x_0))  + e(f(x_0), y_0)\\
%& < \epsilon - e(y_0, f(x_0)) + e(f(x_0), y_0)\\
%& = \epsilon.
%\end{align*}
%For the second inlcusion, let $x{'} \in X$ with $e(f(x{'}), y_0) \geq \epsilon$, and we show that $d(x_0, x{'}) \geq \omega_{x_0}(\epsilon - e(y_0, f(x_0)))$. Suppose that $d(x_0, x{'}) < \omega_{x_0}(\epsilon - e(y_0, f(x_0)))$, hence by continuity at $x_0$ we get $e(f(x_0), f(x{'})) < \epsilon - e(y_0, f(x_0))$. Working as above, we get $e(f(x{'}), y_0) \leq e(f(x{'}), f(x_0))  + e(f(x_0), y_0) < \epsilon$, which conradicts our hypothesis on $x{'}$, and $d(x_0, x{'}) \geq \omega_{x_0}(\epsilon - e(y_0, f(x_0)))$ follows by Remark~\ref{rem: order1}.\\
\end{proof}

The previous proof shows that the use of $<$ instead of $\leq$ in the definition of (pointwise) continuity of $\fXY$ is crucial. If $\B H$ is a complemented ball in $(Y, e)$, then we get
$$\op_{f^{-1}(\B B(y_0, \epsilon))}(x_0) := \omega_{x_0}(\epsilon - e(y_0, f(x_0))) =: \omega_{f,x_0}\big(\op_{\B B(y_0, \epsilon)}(f(x_0))\big),$$
for every $x_0 \in \{x \in X \mid e(f(x), y_0) < \epsilon\}$. If $\B G$ is $\B {\C T}_d$-open, let the %following 
sets of moduli of openness
$$\Mod(\B G) := \{\phi \in \X F\big(G^1, (0, +\infty)\big) \mid \phi \ \mbox{is a modulus of openness for} \ \B G)\},$$
$$\Mod^{\ast}(\B G) := \{\phi \in \Mod(\B G) \mid \phi \ \mbox{satisfies condition (*) in Proposition~\ref{prp: pcont1}}\}.$$
The disjoint union of the family $\B G \mapsto \Mod(\B G)$, defined in~\cite{Pe20}, section 3.2, is the set
$$\bigcup_{\B G \in \B {\C T}_d}\Mod(\B G)$$
that has elements pairs $(\B G, \phi)$, where $\B G \in \B {\C T}_d$ and $\phi \in \Mod(\B G)$. For the definition of equality on this set see~\cite{Pe20}. Due to condition $(*)$ above, the modulus of pointwise continuity $\omega_f$ of $f$ induces a function 
$$\op_f^* \colon \bigcup_{\B H \in \B {\C T}_e}\Mod(\B H) \to \bigcup_{\B G \in \B {\C T}_d}\Mod^*(\B G).$$
Next follows the converse to Proposition~\ref{prp: pcont1}.

\begin{proposition}\label{prp: pcont2}
	Let $(X, d)$ and $(Y, e)$ be metric spaces and let $\fXY$ be strongly extensional. If $f$ inverses open complemented subsets together with their moduli of openness i.e., if there is a function
	$$\op_f^* \colon \bigcup_{\B H \in \B {\C T}_e}\Mod(\B H) \to \bigcup_{\B G \in \B {\C T}_d}\Mod^*(\B G),$$
$$(\B H, \op_{\B H}) \mapsto (f^{-1}(\B H), \op_{f^{-1}(\B H)}),$$
then $f$ is pointwise continuous on $X$ with a modulus of pointwise continuity $\omega_{f}$, defined by 
$$\omega_{f, x}(\epsilon)  := \op_{f^{-1}(\B B(f(x), \epsilon))}(x), \ \ \ \ x \in X, \ \epsilon >0.$$
\end{proposition}

\begin{proof}
	If 
	%$f$ is strongly extensional and inverses open complemented subsets together with their moduli of openness, and 
	$x_0 \in X$ and $\epsilon > 0$, then by hypothesis on $f$ we have that 
	$$f^{-1}\big(\B B(f(x_0), \epsilon)\big) := \big(f^{-1}\big([e_{f(x_0)} < \epsilon]\big), f^{-1}\big([e_{f(x_0)} \geq \epsilon]\big)\big)$$
	is $\B {\C T}_d$-open. The real number $\omega_{f, x_0}(\epsilon)  := \op_{f^{-1}(\B B(f(x_0), \epsilon))}(x_0)$ is well-defined, as by Remark~\ref{rem: ball1}(i) the complemented ball $\B B(f(x_0), \epsilon)$ has a modulus of openness and by our hypothesis its inverse image under $f$ also has one\footnote{In principle $\Mod(\B H)$ maybe empty. What the rule $\op_f$ guarantees is that, given a modulus of opennes for $\B H$, a modulus of openness for $f^{-1}(\B H)$ is generated by $\op_f$.}. In this case we have that 
	$\B B(x_0, \omega_{f, x_0}(\epsilon)) \subseteq f^{-1}\big(\B B(f(x_0), \epsilon)\big)$ i.e., 
	$$[d_{x_0} < \omega_{f,x_0}(\epsilon)] \subseteq f^{-1}\big([e_{f(x_0)} < \epsilon]\big) \ \ \& \ \ 
	f^{-1}\big([e_{f(x_0)} \geq \epsilon]\big) \subseteq [d_{x_0} \geq \omega_{f, x_0}(\epsilon)].$$
	Hence, if $x \in X$ with $d(x_0, x) < \omega_{f, x_0}(\epsilon)$, then $x \in f^{-1}\big([e_{f(x_0)} < \epsilon]\big)$ i.e., $e(f(x_0), f(x)) < \epsilon$.
\end{proof}

The previous proof shows that it suffices that $f$ inverses the moduli of openness for the complemented balls i.e., there is a modulus of openness $\op_{f^{-1}(\B B(y_0, \epsilon))}$ for  $f^{-1}(\B B(y_0, \epsilon))$, for every $\epsilon > 0$ and $y_0 \in Y$. 
Clearly, if $f$ inverses complemented balls, then the proof of the converse does not require the inversion of moduli. Next we show that the two constructions 
described in Propositions~\ref{prp: pcont1} and~\ref{prp: pcont2}, respectively, are inverse to each other. For case (ii) below, it is essential to begin with a function $\op_f^*$ that sends a modulus of opennes for $\B H$ to a modulus of openness for $f^{-1}(\B H)$ that satisfies condition (*) in Proposition~\ref{prp: pcont1}. Actually, condition (*) was motivated by the proof of Proposition~\ref{prp: pcont3}(ii). As these are exactly the functions generated by a modulus $\omega_f$, there is no loss of generality.

\begin{proposition}\label{prp: pcont3}
\normalfont (i)
\itshape Let the constructions $(f, \omega_f) \mapsto (f, \op_f^*)$ and $(f, \op_f*) \mapsto (f, \omega_f{'})$, described in Propositions~\ref{prp: pcont1} and~\ref{prp: pcont2}, respectively. Then $\omega_f{'} = \omega_f$, where this equality is within the set $\X F\big(X, \X F((0, +\infty), (0, +\infty))\big)$.\\[1mm]
\normalfont (ii)
\itshape Let the constructions $((f, \op_f^*) \mapsto (f, \omega_f)$ and $(f, \omega_f) \mapsto (f, \op_f{'})$, described in Propositions~\ref{prp: pcont2} and~\ref{prp: pcont1}, respectively, where 
$\op_f^* \colon \bigcup_{\B H \in \B {\C T}_e}\Mod(\B H) \to \bigcup_{\B G \in \B {\C T}_d}\Mod^*(\B G)$. 
Then $\op_f{'} =\op^*_f$, where this equality is within the set $\X F\big(\bigcup_{\B H \in \B {\C T}_e}\Mod(\B H), \bigcup_{\B G \in \B {\C T}_d}\Mod^*(\B G)\big)$.
\end{proposition}

\begin{proof}
(i) If $x \in X$ and $\epsilon >0$, we have that
\begin{align*}
\omega_{f, x}{'}(\epsilon) & := \op_{f^{-1}(\B B(x, \epsilon))}(x)\\
& := \omega_{f,x}\big(\op_{\B B(x, \epsilon}(f(x)))\big)\\
& :=\omega_{f,x}\big(\epsilon - e(f(x), f(x))\big)\\
& =_{\Real} \omega_{f,x}(\epsilon).
\end{align*}
(ii) If $\B H \in \B {\C T}_e$ and $x \in f^{-1}(H^1)$, then by condition (*) we have that
\begin{align*}
\op{'}_{f^{-1}(\B H)}(x) & := \omega_{f,x}\big(\op_{\B H}(f(x))\big)\\
& := \op_{f^{-1}(\B B(f(x), \op_{\B H}(f(x))))}(x)\\
& =_{\Real} \op_{f^{-1}(\B H)}(x).  \ \ \ \ \ \qedhere 
\end{align*}
\end{proof}

\begin{remark}\label{rem: pcont4}
Let $(X, d)$ and $(Y, e)$ be metric spaces and $(\fXY, \omega_f)$ pointwise continuous on $X$.\\[1mm]
\normalfont (i) 
\itshape $\omega_{f, x}(1) =_{\Real} 1$, for every $x \in X$.\\[1mm]
\normalfont (i)
\itshape $\omega_{f, x}\big(\op_{\B K}(x) \wedge \op_{\B L}((f(x))\big) =_{\Real} \omega_{f, x}\big(\op_{\B K}(f(x))\big) \wedge \omega_{f, x}\big(\op_{\B L}(f(x))\big)$, for every $\B K, \B L \in \B {\C T}_e$ and every $x \in K^1 \cap L^1$.
\end{remark}

\begin{proof}
(i) If $x \in X$, then $1 =: \op_{(X, \ \emptys_X)}(x)  =_{\Real} \op_{f^{-1}(Y, \ \emptys_Y)}(x) := \omega_{f,x}\big(\op_{(Y, \ \emptys_Y)}(f(x))\big) := \omega_{f,x}(1)$.\\
(ii) 
%If $\B K, \B L \in \B {\C T}_e$ and $x \in K^1 \cap L^1$, then 
Since by definition equal open complemented subsets have equal moduli of openness, we get
\begin{align*}
\omega_{f, x}\big(\op_{\B K}(x) \wedge \op_{\B L}((f(x))\big) & =: \omega_{f, x}\big(\op_{\B K \cap \B L}(f(x))\big)\\
& =: \op_{f^{-1}(\B K \cap \B L)}(x)\\
& =_{\Real}   \op_{f^{-1}(\B K)}(x) \wedge  \op_{f^{-1}(\B L)}(x)\\
& =:  \omega_{f, x}\big(\op_{\B K}(f(x))\big) \wedge \omega_{f, x}\big(\op_{\B L}(f(x))\big). \qedhere
\end{align*}
\end{proof}

The uniform case of Proposition~\ref{prp: pcont1} is shown similarly using Proposition~\ref{prp: compcont}(iii).

\begin{proposition}\label{prp: ucont1}
	Let $(X, d)$ and $(Y, e)$ be metric spaces and $\fXY$. If $f$ is uniformly continuous on $X$ with a modulus of uniform continuity $\Omega_{f}$, then $f$ inverses open complemented subsets and their moduli of openness, where if $\op_{\B H}$ is a modulus of openness for $\B H$, then the function
	$$\op_{f^{-1}(\B H)} := \Omega_{f} \circ \op_{\B H} \circ f_{|f^{-1}(H^1)},$$
	is a modulus of opennness for $f^{-1}(\B H)$, such the following conditions hold:\\[1mm]
	\normalfont (*)
	\itshape If $x^1 \in f^{-1}(H^1)$, then $op_{f^{-1}(\B H)}(x^1) =_{\Real} \op_{f^{-1}(\B B(f(x^1), \op_{\B H}(f(x))))}(x^1)$.\\[1mm]
	\normalfont (i)
	\itshape $\op_{\id_X^{-1}(\B G)} =_{\X F(G^1, (0, +\infty)])} \op_{\B G}$, for every $\B G \in \B {\C T}_{d}$.\\[1mm]
	\normalfont (ii)
	\itshape $\op_{(g \circ f)^{-1}(\B K)} =_{\X F(f^{-1}(g^{-1}(K^1)), (0, +\infty)])} \op_{f^{-1}(g^{-1}(\B K))}$, where $g \colon (Y, e) \to (Z, \rho)$ is uniformly continuous with modulus of uniform continuity $\Omega_g$, and $\B K$ is an arbitrary open complemented subset in $\B {\C T}_{\rho}$. 
\end{proposition}

%As  for pointwise continuous functions, d
Due to condition $(*)$ above, the modulus of uniform continuity $\Omega_f$ of $f$ induces a function 
$$\Op_f^* \colon \bigcup_{\B H \in \B {\C T}_e}\Mod(\B H) \to \bigcup_{\B G \in \B {\C T}_d}\Mod^*(\B G).$$
 If $(f, \Omega_f)$ is uniformly continuous, then for every $y \in Y$, $\epsilon > 0$ and $x \in f^{-1}(\B B(x, \epsilon))$, we have that 
$$\op_{f^{-1}(\B B(y, \epsilon))}(x) := \Omega_f\big(\op_{\B B(y, \epsilon)}(f(x)) := \Omega_f\big(\epsilon - e(y, f(x))\big).$$
Hence, for every $x \in X$ we get
$$\op_{f^{-1}(\B B(f(x), \epsilon))}(x) =_{\Real} \Omega_f(\epsilon).$$
This observation motivates the formulation of the converse to Proposition~\ref{prp: ucont1} that follows.
We need $X$ to be an inhabited metric space, in order to define the induced modulus of uniform continuity.

\begin{proposition}\label{prp: ucont2}
	Let $(X, d)$ and $(Y, e)$ be metric spaces, $x_0 \in X$, and $\fXY$ strongly extensional. If $f$ inverses open complemented subsets 
	%together with 
	and their moduli of openness i.e., if there is a function
	$$\Op_f^* \colon \bigcup_{\B H \in \B {\C T}_e}\Mod(\B H) \to \bigcup_{\B G \in \B {\C T}_d}\Mod^*(\B G),$$
	$$(\B H, \op_{\B H}) \mapsto (f^{-1}(\B H), \op_{f^{-1}(\B H)}),$$
	such that for every $\epsilon > 0$ it satisfies
	$$\forall_{x, x{'} \in X}\big(\op_{f^{-1}(\B B(f(x), \epsilon))}(x) =_{\Real} \op_{f^{-1}(\B B(f(x{'}), \epsilon))}(x{'}) \big),$$
	then $f$ is uniformly continuous on $X$ with a modulus of uniform continuity $\Omega_{f}$, defined by 
	$$\Omega_{f}(\epsilon)  := \op_{f^{-1}(\B B(f(x_0), \epsilon))}(x_0).$$
\end{proposition}

\begin{proof}
	We fix $\epsilon > 0$. By the crucial equality condition satisfied by $\Op_f$, let $x, x{'} \in X$ such that
	$$d(x,x{'}) < \Omega_f(\epsilon) := \op_{f^{-1}(\B B(f(x_0), \epsilon))}(x_0) =_{\Real} \op_{f^{-1}(\B B(f(x), \epsilon))}(x).$$
	Working as in the proof of Proposition~\ref{prp: pcont2}, we get $x{'} \in \B B_d(x, \Omega_f(\epsilon)) \subseteq f^{-1}\big(\B B_e(f(x), \epsilon)\big)$.
	%, hence $e(f(x), f(x{'})) < \epsilon$.
	\end{proof}

%The proof that  is similar to the proof of 
Similarly to the proof of Proposition~\ref{prp: pcont3}, the two constructions above are inverse to each other.

\begin{proposition}\label{prp: ucont3}
	\normalfont (i)
	\itshape Let the constructions $(f, \Omega_f) \mapsto (f, \Op_f^*)$ and $(f, \Op_f^*) \mapsto (f, \Omega_f{'})$, described in Propositions~\ref{prp: ucont1} and~\ref{prp: ucont2}, respectively. Then $\Omega_f{'} =_{\mathsmaller{\X F((0, +\infty), (0, +\infty))}} \Omega_f$.\\[1mm]
	\normalfont (ii)
	\itshape Let the constructions $((f, \Op_f^*) \mapsto (f, \Omega_f)$ and $(f, \Omega_f) \mapsto (f, \Op_f{'})$, described in Propositions~\ref{prp: ucont2} and~\ref{prp: ucont1}, respectively, where 
	$\Op_f^* \colon \bigcup_{\B H \in \B {\C T}_e}\Mod(\B H) \to \bigcup_{\B G \in \B {\C T}_d}\Mod^*(\B G)$. 
	Then $\Op_f{'} =\Op^*_f$, where this equality is within the set $\X F\big(\bigcup_{\B H \in \B {\C T}_e}\Mod(\B H), \bigcup_{\B G \in \B {\C T}_d}\Mod^*(\B G)\big)$.
\end{proposition}

%Different behaviour of pointwise continuous and uniformly continuous functions with respect to the inversion of moduli in specific cases. Clearly, uniformly continuous functions have a better behaviour...
In~\cite{BB85} one can find many examples of uniformly continuous functions between metric spaces, such as the constant functions, the distance from a point $x_0 \in X$, $x \mapsto d(x, x_0)$, and the distance from a located subset $A$ of $X$, $x \mapsto d(x, A)$, for every $x \in X$. It is straightforward tio calculate the uniform inversion of the corresponding moduli of openness for them.

\section{General $\cs$-topologies}
\label{sec: cstop}

Next we define the notions of a $\cs$-topology and a $\cs$-topological space. This is the first step in the use of complemented subsets in constructive topology. The immediate generalisation of the canonical $\cs$-topology induced by some metric pertains to $\cs$-spaces with a given base, which is presented in section~\ref{sec: csbtop}.

\begin{definition}\label{def: ctop}
Let $\C X : = (X, =_X, \neq_X) \in \SetExtIneq$. If $T(\B G)$ is an extensonal property on $\C E^{\Disj}(X)$ and $\B {\C T}$ 
%, or a $c$-topology, 
is the extensional subset 
$\B {\C T} := \{\B G \in \C E^{\Disj}(X) \mid T(\B G)\}$
of $\C E^{\Disj}(X)$, then $\B {\C T}$ is a topology of open complemented subsets, or a $\cs$-topology, if the following conditions are sastified:\\[1mm]
$(\cTop_1)$ $(X, \emptys_X) \in \B {\C T}$ and $(\emptys_X, X) \in \B {\C T}$.\\[1mm]
$(\cTop_2)$ If $\B G, \B H \in \B {\C T}$, then 
%$\B G \cup \B H,$
$\B G \cap \B H \in \B {\C T}$.\\[1mm]
$(\cTop_3)$ If $I$ is a given set and $(\B G_i)_{i \in I}$ is an $I$-family of elements of $\B {\C T}$, then
$\bigcup_{i \in I}\B G_i \in \B {\C T}$.\\[2mm]
We call the pair $(X, \B {\C T})$ a $\cs$-topological space.
\end{definition}
An alternative formulation of $(\cTop_3)$ is $(\cTop_3{'})$, according to which, if 
$F(\B G)$ 
%X_{P^1}, X_{P^0})$ 
is any extensional property on $\C E^{\Disj}(X)$,
% and 
%$\B {\C F} := \{(X_{P^1}, X_{P^0}) \in \C E^{\Disj}(X) \mid F(X_{P^1}, X_{P^0})\},$
such that $F(\B G) \To T(\B G)$, for every $\B G \in \C E^{\Disj}(X)$, then 
$$\bigcup \B {\C F} := \bigg(\bigcup \C F^1, \bigcap \C F^0\bigg) \in \B {\C T},$$
where $\B {\C F}, \C F^1$ and $\C F^0$ are introduced in Definition~\ref{def: union}.
%$$\C F^1 := \big\{ X_P \in \C E(X) \mid \exists_{X_Q \in \C E(X)}\big(F(X_P, X_Q)\big)\big\},$$
%$$\C F^0 := \big\{ X_Q \in \C E(X) \mid \exists_{X_P \in \C E(X)}\big(F(X_P, X_Q)\big)\big\}.$$
%We define
%$$\bigcup \B {\C F} := \bigg(\bigcup \C F^1, \bigcap \C F^0\bigg), \ \ \ \ \ \bigcap \B {\C F} := \bigg(\bigcap \C F^1, \bigcup \C F^0\bigg).$$

\begin{example}\label{ex: top1}
	\normalfont
The trivial topology of open complemented subsets is the set $\{(X, \emptys_X), (\emptys_X, X)\}$, and the discrete topology is $\C E^{\Disj}(X)$ itself.
	
\end{example}

\begin{example}[The Sierpinski $\cs$-space]\label{ex: top2}
	\normalfont 
Let the booleans $\X 2 := \{0, 1\}$ equipped with the obvious equality and inequality. The Sierpinski $\cs$-topology on $\X 2$ is the set
$\B {\Si} := \{(\X 2, \emptys_{\X 2}), (\emptys_{\X 2}, \X 2), (\{0\}, \{1\})\}$.
	
\end{example}

\begin{definition}\label{def: ccl}
A complemented subset $\B F$ of a $\cs$-topological space $(X, \B {\C T})$ is closed, 
if $- \B F \in \B {\C T}$.  We denote by $- \B {\C T}$ the set of closed complemented subsets of $(X, \B {\C T})$ i.e., $- \B {\C T} := \{ \B F \in \C E^{\Eisj}(X) \mid T(- \B F)\}$.  A
complemented subset $\B G$ of $\C X$ is called clopen, if $\B G$ is open and closed. Let $\Clop(\C X) := \{\B G \in \C E^{\Eisj}(X) \mid T(\B G) \wedge T(-\B G)\}$ be the set of clopen subsets of $\C X$.
\end{definition}

The classical duality between open and
closed subsets of a topology of open subsets is recovered in the ``complemented'' framework.
If $(X, T)$ is a standard topological space, then the pairs $(G, H)$, where $G \in T$ and $H$ closed with $G \Disj H$, form, using classical logic, a topology of open complemented subsets. The pairs of the form $(G, G^c)$, where $G^c$ is the standard (logical) complement of $G$ in $X$, also form classically such a topology. If $\B {\C T}$ is a topology of open coplemented subsets though, its first components ${\C T}^1$ form a standard topology, the closed sets of which need not be in ${\C T}^0$. 
Next we find the swap-analogue to the classical result that the clopen sets of a topological space form a Boolean algebra. 
%A swap algebra is defined in~\cite{MWP23} and it is the abstract version of the structure of $\C E^{\Disj}(X)$.....

\begin{corollary}\label{cor: clopswap}
	If $(X, \B {\C T})$ is a $\cs$-topological space, its clopen comlemented subsets $\Clop(X, \B {\C T})$ is a swap algebra of type $(\ti)$. The total elements $\TotClop(X, \B {\C T})$ of $\Clop(X, \B {\C T})$ form a Boolean algebra.
\end{corollary}

\begin{proof}
Clearly, $\Clop(X, \B {\C T})$ is a field of complemented sets of type $(\ti)$, hence by Remark~\ref{rem: cfield1} $\Clop(X, \B {\C T})$ is a swap algebra of type $(\ti)$.
\end{proof}

Strong extensionality of functions is crucial to our topological framework, as in Bishop-Cheng Measure Theory. The category $(\csTop, \csCont)$ of $\cs$-topological
spaces, and $\cs$-continuous functions, defined next, is a subcategory of $(\SetExtIneq, \StrExtFun)$.

\begin{definition}\label{def: ccont}
	If $(X, \B {\C T})$ to $(Y, \B {\C S})$ are $\cs$-topological spaces, a $\cs$-continuous function from $(X, \B {\C T})$ to $(Y, \B {\C S})$ is a \textit{strongly extensional} function $f \colon X \to Y$ such that
	$f^{-1}(\B H) \in \B {\C T}$, for every $\B H \in \B {\C S}$.  Let $\csC(\C X, \C Y)$ be the set of $\cs$-continuous functions from $\C X$ to $\C Y$, and let $(\csTop, \csCont)$ be the category of $\cs$-spaces and $\cs$-continuous functions.
\end{definition}

%EXAMPLES...
The $\subseteq$-relations
% found 
in the next proposition are between extensional subsets of $\C E^{\Disj}(X)$. The proof of the second inclusion in (ii) is due to Proposition~\ref{prp: finvproperties}(ii).

\begin{proposition}\label{prp: finvcont}
	\normalfont (i) 
	\itshape
	The inverse image under $f$ of a field of complemented sets in $\C Y$ is a field of complemented sets in $\C X$.\\[1mm]
	\normalfont (ii) 
	\itshape $f^{-1}\big(\Clop(\C Y)\big) \subseteq \Clop(\C X)$ and $f^{-1}\big(\TotClop(\C Y)\big) \subseteq \TotClop(\C X)$.
%	\normalfont (iii) 
%	\itshape
%	
\end{proposition}

%\begin{proof}
%	The proof of the second inclusion in (ii) is due to Proposition~\ref{prp: finvproperties}(ii).	? or straightforward..
%\end{proof}

\begin{definition}\label{def: homeo}
	A homeomorphism between $\cs$-spaces $(X, \B {\C T})$ and $(Y, \B {\C S})$, or a $\cs$-homeomorphism, is a pair $(f \colon X \to Y, g \colon Y \to X)$ of inverse $\cs$-continuous functions. A function $\fXY$ is $\cs$-open, if $f(\B G) \in \B {\C S}$, for every $\B G \in \B {\C T}$, and it is $\cs$-closed, if $f(\B F) \in - \B {\C S}$, for every $\B G \in  - \B {\C T}$.
	\end{definition}

\begin{definition}\label{def: relative}
	If $(X, \B {\C T})$ is a $\cs$-topological space and $\B A := (A, \emptys_X)$, the relative $\cs$-topology of $\C T$ on $\B A$ is $\C T_{\B A} := \{ 0_{\B A} \cup (\B A \cap \B G) \mid \B G \in \C T\}$.
\end{definition}

If $\B A := (X, \emptys_X)$, the only total complemenetd subset of this form, then $\C T_{\B A}
= \{\B A \cap \B G \mid \B G \in \C T\}$ has the standard form of the relative topology in the $1$-diamensional framework. 

\begin{proposition}\label{prp: relcstop}
$(\INT)$	If $\B A := (A, \emptys_X)$, then $\C T_{\B A}$ 
%	relative $\cs$-topology of $\C T$ on $\B A$
	 is a $\cs$-topology on $\B A$.
\end{proposition}

\begin{proof}
	First we notice that\footnote{If instead of $(A, \emptys_X)$, we consider an arbitrary complemented subset $(A^1, A^0)$ in the definition of the relative topology, then $0_{\B A} \cup (\B A \cap \B G) := (A^1 \cap G^1, A^0 \cup (A^1 \cap G^0))$. Hence $0_{\B A} \cup (\B A \cap (\ \emptys_X, X)) =_{\C E^{\Disj}(X)} (\ \emptys_X, A^0 \cup A^1)$. Consequently, if $A^0$ is inhabited, we get a pair unequal to $(\ \emptys_X, A^1)$. This is the reason we restrict the definition of relative topology to complemented subsets of $X$ with $\emptys_X$ as their $0$-component.}
	$$0_{\B A} \cup (\B A \cap \B G) := (A \cap G^1, {\emptys_X} \cup (A \cap G^0)) =_{\C E^{\Disj}(X)} (A \cap G^1, A \cap G^0).$$
	Next we show that $\B A \in \C T_{\B A}$, as
	$0_{\B A} \cup (\B A \cap \B (X, \emptys_X)) =_{\C E^{\Disj}(X)} (A \cap X, A \cap {\emptys_X}) 
	=_{\C E^{\Disj}(X)} \B A$.
	To show that $- \B A := ({\emptys_A}, A) \in \C T_{\B A}$, we observe that ${\emptys_A} =_{\C E(X)} {\emptys_X}$: if the extensional property $A(x)$ defines $A$, then 
	$$\emptys_A := \{a \in A \mid a \neq_A a\} =_{\C E(X)} \{x \in X \mid A(x) \wedge x \neq_X x\} \subseteq {\emptys_X}.$$
	The converse inclusion follows from Ex Falso (if $x \neq_X x$, then $A(x)$). Consequently, 
	$0_{\B A} \cup (\B A \cap \B (\ \emptys_X, X)) =_{\C E^{\Disj}(X)} (\ \emptys_X, A) 
	=_{\C E^{\Disj}(X)} (\ \emptys_A, A) =: -\B A$. If $\B G, \B H \in \B {\C T}$, then 
	\begin{align*}
	[0_{\B A} \cup (\B A \cap \B G)] \cap [0_{\B A} \cup (\B A \cap \B H)] & =_{\mathsmaller{\C E^{\Disj}(X)}}(A \cap G^1, A \cap G^0) \cap (A \cap H^1, A \cap H^0)\\
	& =_{\mathsmaller{\C E^{\Disj}(X)}}  ((A \cap (G^1 \cap H^1)), (A \cap G^0) \cup (A \cap H^0))\\
	& =_{\mathsmaller{\C E^{\Disj}(X)}}  ((A \cap (G^1 \cap H^1)), (A \cap (G^0 \cap H^0))\\
	&=_{\mathsmaller{\C E^{\Disj}(X)}} 0_{\B A} \cup (\B A \cap (\B G \cap \B H)).		
	\end{align*}
	If $(\B G_i)_{i \in I}$ is a family of open complemented subsets in $\B {\C T}$ indexed by some set $I$, then by Bishop's distributivity property we have that
	\begin{align*}
0_{\B A} \cup \bigg(\B A \cap \bigcup_{i \in I} \B G(i)\bigg) & =_{\mathsmaller{\C E^{\Disj}(X)}}
0_{\B A} \cup \bigg[\bigcup_{i \in I} \big(\B A \cap \B G(i)\big)\bigg]\\
& := 0_{\B A} \cup \bigg[\bigcup_{i \in I} \big(A \cap G(i)^1, G(i)^0\big)\bigg]\\
& := 0_{\B A} \cup \bigg(\bigcup_{i \in I} A \cap G(i)^1, \bigcap_{i \in I}G(i)^0\bigg)\\
& =_{\mathsmaller{\C E^{\Disj}(X)}} \bigg(A \cap \bigg(\bigcup_{i \in I} A \cap G(i)^1\bigg), A \cap \bigcap_{i \in I}G(i)^0\bigg)\\
& =_{\mathsmaller{\C E^{\Disj}(X)}} \bigg(\bigcup_{i \in I} \big(A \cap G(i)^1\big), \bigcap_{i \in I}\big(A \cap G(i)^0\big)\bigg)\\
& =: \bigcup_{i \in I}\big(A \cap G(i)^1, A \cap G(i)^0\big)\\
& =_{\mathsmaller{\C E^{\Disj}(X)}}  \bigcup_{i \in I}[0_{\B A} \cup (A \cap \B G(i))]. \ \ \ \ \qedhere
\end{align*}
%	in order to show $(\cTop_3)$ for the relative $\cs$-topology.
\end{proof}

The classical definition of quotient topology is translated smoothly in the context of $\cs$-spaces. The proof that the $\cs$-quotient topology is actually a $\cs$-topology is straightfowrad and it is based on the properties of the inverse image of a strongly extensional function (Proposition~\ref{prp: finvproperties}).

\begin{definition}\label{def: csquotient}
		If $(X, \B {\C T})$ is in $\csTop$, $\C Y \in \SetExtIneq$ and $\phi \colon X \to Y$ is strongly extensional, the quotient $\cs$-topology in $\C Y$ determined by $\phi$ is given by
		$$\B {\C T}_{\phi} := \{\B H \in \C E^{\Disj}(Y) \mid \phi^{-1}(\B H) \in \B {\C T}\}.$$ 
\end{definition}
The study of the quotient $\cs$-topology is along the line of classical quotient topology of open subsets. Based on the properties of the direct image of a strong injection (Proposition~\ref{prp: fimageproperties}), the dual construction gives rise to the $\cs$-co-quotient topology on $\C X \in \SetExtIneq$ determined by some surjection and strong injection $\theta \colon X \to Y$, where $(\C Y, \B {\C S}) \in \csTop$ and $\B {\C T}_{\theta} := \{\B G \in \C E^{\Disj}(X) \mid \theta(\B G) \in \B {\C S}\}$.

\section{$\cs$-Topologies induced by bases}
\label{sec: csbtop}

The natural generalisation of a $\cs$-topology $\B {\C T}_d$ induced by some metric $d$ on $X$ is the notion of a $\cs$-topology $\B {\C T}_{\B {\C B}}$ induced by some base $\B B$ on $X$. In the case of metric spaces the equality $x =_X x{'}$ is equivalent to the induced by the metric equality $d(x, x{'}) =_{\Real} 0$, the (strong) negation of which is $d(x, x{'}) \neq_{\Real} 0$, which amounts to the inequality $d(x, x{'})> 0$. In the case of $\cs$-topologies induced by a base $\B {\C B}$ we start from a set with an inequality,  which is necessary to the definition of complemented subsets, and then we define the canonical equality and inequality induced by $\B {\C B}$. The relation of these induced relations to the original ones is presented next. 

\textit{Throughout this section $\C X := (X, =_X, \neq_X), \C Y := (Y, =_Y, \neq_Y) \in \SetExtIneq$, $\C B(\B B)$ is an extensional property on $\C E^{\Disj}(X)$,  $\C C(\B C)$ is an extensional property on $\C E^{\Disj}(Y)$, $\B {\C B} := \{\B B \in \C E^{\Disj}(X) \mid \C B(\B B)\}$, and $\B {\C C} := \{\B C \in \C E^{\Disj}(Y) \mid \C C(\B C)\}$.}

\begin{definition}\label{def: eqineqB}
%Let $\C X := (X, =_X, \neq_X) \in \SetExtIneq$, and $\C B(\B G)$ an extensional property on $\C E^{\Disj}(X)$. If $\B {\C B} := \{\B B \in \C E^{\Disj}(X) \mid \C B(\B B)\}$, 
The 
%canonical 
equality and inequality relations on $X$ induced by $\B {\C B}$ are defined, respectively, by
$$x =_{\B {\C B}} x{'} :\TOT \forall_{\B B \in \B {\C B}}\big(\B x \cin \B B \TOT x{'} \cin \B B\big),$$ 
$$x \neq_{\B {\C B}} x{'} :\TOT \exists_{\B B \in \B {\C B}}\big[\big(\B x \cin \B B \wedge \B x{'} \nncin \B B\big) \vee \big(\B x{'} \cin \B B \wedge \B x \nncin \B B\big)\big].$$
If $\B B \in \B {\C B}$ witnesses the inequality $x \neq_{\B {\C B}} x{'}$, we write $\B B \colon x \neq_{\B {\C B}} x{'}$, or $x \neq_{\B B} x{'}$. We say that $\B {\C B}$ separates the points of $X$, and in this case we call the inequality $x \neq_{\B {\C B}} x{'}$ $0$-tight if,
% , we have that
$$x =_{\B {\C B}} x{'} \To x =_X x{'},$$
for every $x, x{'} \in X$. We say that $\B {\C B}$ co-separates the points of $X$, if, for every $x, x{'} \in X$, 
%we have that
$$x \neq_X x{'} \To x \neq_{\B {\C B}} x{'}.$$
We call $\B {\C B}$ cotransitive, if the induced inequality $\neq_{\B {\C B}}$ is cotransitive. Let also $\emptys_{\B {\C B}} := \{x \in X \mid x \neq_{\B {\C B}} x \}$.
\end{definition}

The extensionality of $\neq_{\B {\C B}}$ follows from the extensionality of $B^1, B^0$ and the extensionality of the relations $\B x \cin \B B$ and $\B x \nncin \B B$.

\begin{remark}\label{rem: B1}
	\normalfont (i) 
	\itshape For every $x, x{'} \in X$ we have that $x =_X x{'} \To x =_{\B {\C B}} x{'}$.\\[1mm]
	\normalfont (ii) 
	\itshape For every $x, x{'} \in X$ we have that $x \neq_{\B {\C B}} x{'} \To x =_X x{'}$.\\[1mm]
	\normalfont (iii) 
	\itshape If ${\B {\C B}}$ separates the points of $X$, the equality $=_X$ is equivalent to the induced equality  $=_{\B {\C B}}$
	on $X$.\\[1mm]
	\normalfont (iv) 
	\itshape If every element $\B B$ of $\B {\C B}$ is $0$-tight, then $\neq_{\B {\C B}}$ is tight with respect to $=_{\B {\C B}}$ on $X$.\\[1mm]
	\normalfont (v) 
	\itshape If every element $\B B$ of $\B {\C B}$ is total, then $\B {\C B}$ is cotransitive.\\[1mm]
%	\normalfont (vi) 
%	\itshape If $\neq_X$ is equivalent to $\neq_{\B {\C B}}$, then $\neq_X$ is an apartness relation.\\[1mm]
		\normalfont (vi) 
	\itshape  If $\neq_{\B {\C B}}$ is tight, then $\neq_{\B {\C B}}$ is $0$-tight .\\[1mm]
	\normalfont (vii) 
	\itshape  If $\neq_{\B {\C B}}$ is $0$-tight and every element $\B B$ of $\B {\C B}$ is $0$-tight, then 
	$\neq_{\B {\C B}}$ is tight.\\[1mm]
\normalfont (viii) 
\itshape $(\MIN)$ $\emptys_{\B {\C B}} \subseteq \emptys_X$ and $(\INT)$ $\emptys_X \subseteq \emptys_{\B {\C B}}$.\\[1mm]
\normalfont (ix) 
\itshape
 $(\MIN)$ If $\B {\C B}$ co-separates the points of $X$, then $\emptys_X \subseteq \emptys_{\B {\C B}}$.
\end{remark}

\begin{proof}
Case (i) follows from the extensionality of the relation $x \cin \B B$. Case (ii) follows from the trivial fact that if $\B B \colon x \neq_{\C {\B B}} x{'}$, then  $(x \in ^1 \wedge x{'} \in B^0) \vee (x{'} \in B^1\wedge x \in B^0)$, hence in both cases $x \neq_X x{'}$.
Cases (iii) follows immediately by case (i) and the definition of separation of points by $\B {\C B}$.
For the proof of case (iv) see the argument in the proof of case (viii). 
Case (v) is immediate to show.  
For the proof of (vi), suppose that $x =_{\B {\C B}} x{'}$. In order to show $x =_X x{'}$, it suffices to show $\neg(x \neq_{\B {\C B}} x{'})$, which follows trivially by our hypothesis $x =_{\B {\C B}} x{'}$.
For the proof of (vii), suppose that $x =_{\B {\C B}} x{'} \To x =_X x{'}$, and let $\neg(x \neq_{\B {\C B}} x{'})$. We show that $x =_X x{'}$. Let $\B B \in \B {\C B}$, such that $x \in B^1$. By $0$-tightness of $\B B$, to show that $x{'} \in B^1$ too, it suffices to show $\neg{(x{'} \in B^0)}$, which follows immediately by our hypothesis $\neg(x \neq_{\B {\C B}} x{'})$. If $x{'} \in B^1$, then we show similarly that $x \in B^1$ too. The first inclusion of case (viii) follows immediately from case (ii), while for the secong we use Ex Falso. The inclusion of case (ix) follows immediately from the implication $x \neq_X x \To x \neq_{\B {\C B}} x$.
\end{proof}

As already mentioned after Definition~\ref{def: canonical}, $=_{\B {\C B}}$ and $\neq_{\B {\C B}}$ are  %the above notions of equality and inequality induced by $\B {\C B}$ are
 the set-theoretic version of the notions of equality and inequality induced by an extensional $F \subseteq \X F(X)$ (see~\cite{Pe24} for a 
 %comprehensive 
 study of these concepts)
$$x =_{F} x{'} :\TOT \forall_{f \in F}\big(f(x) =_{\Real} f(x{'})\big),$$
$$x \neq_{F} x{'} :\TOT \exists_{f \in F}\big(f(x) \neq_{\Real} f(x{'})\big).$$
\textit{A complemented subset can be seen as the set-theoretic version of a strongly extensional function}. If $(B, C)$ is a pair of complemented subsets of $X$, we may call it strongly extensional, if for every $x, x{'} \in X$ we have that $x \neq_{(B, C)} x{'} \To x \neq_X x{'}$, where 
$$x \neq_{(B, C)} x{'} \TOT (x \in B \wedge x{'} \in C) \vee (x{'} \in B \wedge x \in C).$$
Strong extensionality of $(B, C)$ implies %that 
$B \Disj C$, and conversely every complemented subset is strongly extensional. Hence,  Remark~\ref{rem: B1}(ii) is the set-theoretic counterpart to the fact that if every function in $F$ is strongly extensional, then $x \neq_F x{'} \To x \neq_X x{'}$ (Remark 3.2(vii) in~\cite{Pe24}). Since 
%Due to the fact that 
$\neq_{\Real}$ is an apartness relation, the inequality $\neq_{F}$ is always cotransitive, hence an apartness relation, while cases (vi) and (vii) of Remark~\ref{rem: B1} are the set-theoretic version of the fact that for $\neq_{F}$ is tight $(\Ineq_3)$ if and only $F$ separates the points of $X$.
In metric spaces the set of complemented balls $\B B(X)$ has the role of $\B {\C B}$.

\begin{remark}\label{rem: metriceqineq}
If $(X, =_X, d)$ is a metric space, and $x, x{'} \in X$, the following hold:\\[1mm]
\normalfont (i)
\itshape $x =_X x{'} \TOT x =_{\B B(X)} x{'}$.\\[1mm]
\normalfont (ii)
\itshape $x \neq_{(X,d)} x{'} \TOT x \neq_{\B B(X)} x{'}$.\\[1mm]
\normalfont (iii)
\itshape $\B B(X)$ is a cotransitive set.
\end{remark}

\begin{proof}
For the proof of the implication $x =_X x{'} \To x =_{\B B(X)} x{'}$ of the case (i), we use the extensionality of complemented balls, while for the converse implication we suppose that $x =_{\B B(X)} x{'}$ and 
%we suppose that 
$\epsilon := d(x, x{'}) > 0$. Consequently, $x \in [d_x < \frac{\epsilon}{2}]$ and $x{'} \in [d_x \geq \frac{\epsilon}{2}]$. The implication $x \neq_{(X,d)} x{'} \To x \neq_{\B B(X)} x{'}$ of case (ii) follows by exactly the previous argument, and for the converse implication, let $y \in X$ and $\epsilon > 0$, such that $d(x, y) < \epsilon$ and $d(x{'}, y) \geq \epsilon$. As $) < \epsilon \leq d(x{'}, y)  \leq d(x, x{'})  + d(x, y)$, we get $0 < \epsilon - d(x,y) \leq d(x, x{'})$. Case (iii) follows from the fact that $x \neq_{(X,d)} x{'}$ is cotransitive.
\end{proof}

%As in the case of metric spaces, the empty set ...

\begin{definition}[Base for a $\cs$-topology]\label{def: csbase}
A base for a $\cs$-topology is a structure $(\B {\C B}, \beta_X, \beta_{\ \emptys_X}, \beta_{\B {\C B}})$, where $\beta_X \colon X \to \B {\C B}$, $\beta_{\ \emptys_X} \colon \emptys_X \to \B {\C B}$, and 
$$ \beta_{\B {\C B}} := \big(\beta_{\B B, \B C}\big)_{\B B, \B C \in  \B {\C B}}, \ \ \ \ 
\beta_{\B B, \B C} \colon B^1 \cap C^1 \to \B {\C B},$$
such that the following conditions hold:\\[1mm]
$(\Base_1)$ $\B x \cin \beta_X(x)$, for every $x \in X$.\\[1mm]
$(\Base_2)$ $\B x \cin \beta_{\ \emptys_X}(x)$ and $\beta_{\ \emptys_X}(x) \subseteq (\emptys_X, X)$, for every $x \in \emptys_X$.\\[1mm]
$(\Base_3)$ $\B x \cin \beta_{\B B, \B C}(x)$ and $\beta_{\B B, \B C}(x) \subseteq \B B \cap \B C$, for every $x \in B^1 \cap C^1$.\\[1mm]
We call the base-moduli $\beta_X$ and $\beta_{\ {\emptys_X}}$ covering, if 
$$\bigcup_{x \in X}\beta_X(x) =_{\C E^{\Disj}(X)} (X, \emptys_X), \ \ \ \ \ \bigcup_{x \in \ \emptys_X}\beta_{\ \emptys_X}(x) =_{\C E^{\Disj}(X)} (\emptys_X, X),$$
respectively. Let also $\beta_X(x) := (B_x^1, B_x^0)$, for every $x \in X$, and $\beta_{\ \emptys_X}(x) := (E_x^1, E_x^0)$, for every $x \in \emptys_X$. We call the base-modulus $\beta_{\B {\C B}}$ covering, if for every $\B B, \B C \in \B {\C B}$, we have that 
$$\B B \cap \B C =_{\C E^{\Disj}(X)} \bigcup_{z^1 \in B^1 \cap C^1}\beta_{\B B, \B C}(z^1).$$
Let $\beta_{\B B, \B C}(z^1) := (BC_{z^1}^1, BC_{z^1}^0)$, for every $z^1 \in B^1 \cap C^1$.
\end{definition}

Clearly, the base-moduli $\beta_X, \beta_{\ \emptys_X}$, and $\beta_{\B {\C B}}$ are covering, if the following inclusions hold, respectively,
$$\bigcap_{x \in X}B_x^0 \subseteq {\emptys_X}, \ \ \ \ \ X \subseteq \bigcap_{x \in \ {\emptys_X}}E_x^0, \ \ \ \ \ \ \bigcap_{z^1 \in B^1 \cap C^1}BC_{z^1}^0 \subseteq B^0 \cup C^0.$$

\begin{remark}\label{rem: metricbase}
	If $d$ is a metric on $X$,
	%$(X, =_X, d)$ is a metric space,
	then $(\B B(X), \beta_X, \beta_{\ \emptys_X}, \beta_{\B B(X)})$ is a base for 
	%the $\cs$-topology 
	$\C {\B T}_d$,
	%induced by $d$, 
	where $\beta_X \colon X \to \B B(X)$, $\beta_{\ \emptys_X} \colon \emptys_X \to \B B(X)$, and $\beta_{\B B(x, \epsilon), \B B(y, \delta)} \colon [d_x < \epsilon] \cap [d_y < \delta] \to \B B(X)$ are defined, respectively, by
	$$\beta_X(x) := \B B(x, 1), \ \ \ \ x \in X,$$
	$$\beta_{\ \emptys_X}(x) := \B B\bigg(x, \frac{d(x,x)}{2}\bigg), \ \ \ \ \ x \in \emptys_X,$$
	$$\beta_{\B B(x, \epsilon), \B B(y, \delta)}(z) := \B B(z, \zeta_z), \ \ \ \ \zeta_z := [\epsilon - d(x,z)] \wedge [\delta - d(y,z)], \ \ \ \ z \in [d_x < \epsilon] \cap [d_y < \delta].$$
	Moreover, $\beta_X, \beta_{\ \emptys_X}$ and $\beta_{\B B(X)}$ are covering base-moduli. 
\end{remark}

\begin{proof}
Condition $(\Base_1)$ follows trivially, while the proof of condition $(\Base_2)$ is within the proof of Proposition~\ref{prp: metric}. For the proof of condition $(\Base_3)$, it suffices to show that $[d_z < \zeta_z] \subseteq [d_x < \epsilon] \cap [d_y < \delta]$ and $[d_x \geq \epsilon] \cup [d_y \geq \delta] \subseteq [d_z \geq \zeta_z]$. If $x{'} \in X$ with $d(z, x{'}) < \zeta_z$, then we have that
$$d(x, x{'}) \leq d(x,z) + d(z, x{'}) < d(x, z)  + \zeta_z \leq d(x,z) + \epsilon - d(x, z) = \epsilon,$$
$$d(y, x{'}) \leq d(y,z) + d(z, x{'}) < d(y,z) + \zeta_z \leq d(y, z) + \delta - d(y, z) = \delta.$$
If $x{''}, y{'} \in X$ with $d(x, x{''}) \geq \epsilon$ and $d(y, y{'}) \geq \delta$, then we have that
$$d(x{''}, z) \geq d(x, x{''}) - d(x,z) \geq \epsilon - d(x, z) \geq \zeta_z,$$
$$d(y{'}, z) \geq d(y{'}, y) - d(y,z) \geq \delta - d(y, z) \geq \zeta_z.$$
The covering properties of $\beta_X, \beta_{\ \emptys_X}$ and $\beta_{\B B(X)}$ are shown in Proposition~\ref{prp: basemodule1}(i) and (iii).
\end{proof}

\begin{remark}\label{rem: pcont5}
Let $(X, d)$ and $(Y, e)$ be metric spaces and $(f, \omega_f)$ pointwise continuous on $X$. Then for every $y, y{'} \in Y, \epsilon, \delta > 0$, and for every $u \in f^{-1}\big(\B B(y, \epsilon) \cap \B B(y{'}, \delta)\big)$ we have that 
$$\omega^*_{f, u}\big(\beta_{\B B(y, \epsilon), \B B(y{'}, \delta)}(f(u))\big) =_{\B B(X)} 
\beta_{f^{-1}\big(\B B(y, \epsilon)\big), f^{-1}\big(\B B(y{'}, \delta)\big)}(u).$$
\end{remark}

\begin{proof}
By Remark~\ref{rem: pcont4}(ii) we have that
\begin{align*}
\omega^*_{f, u}\big(\beta_{\B B(y, \epsilon), \B B(y{'}, \delta)}(f(u))\big) & := 
\omega^*_{f, u}\big(\B B\big(f(u), \op_{\B B(y, \epsilon)}(f(u)) \wedge \op_{ \B B(y{'}, \delta)}(f(u))\big)\big)\\
 & :=  \B B\big(u, \omega_{f,u}\big(\op_{\B B(y, \epsilon)}(f(u)) \wedge \op_{ \B B(y{'}, \delta)}(f(u))\big)\big)\\
  & =_{\B B(X)} \B B\big(u, \omega_{f,u}\big(\op_{\B B(y, \epsilon)}(f(u))\big) \wedge \omega_{f,u}\big(\op_{ \B B(y{'}, \delta)}(f(u))\big)\big)\\
  & =: \B B\big(u, \op_{f^{-1}(\B B(y, \epsilon))}(u) \wedge \op_{f^{-1}(\B B(y, \delta))}(u)\big)\\
  & =: \beta_{f^{-1}\big(\B B(y, \epsilon)\big), f^{-1}\big(\B B(y{'}, \delta)\big)}(u). \qedhere
\end{align*}
\end{proof}

The following fact on the base of a relative topology induced by the base for the original space is straightforward to show. The covering conditions are also inherited from $\B {\C B}$ to $\B {\C B_A}$.

\begin{remark}\label{rem: baserel}
If $(\B {\C B}, \beta_X, \beta_{\ \emptys_X}, \beta_{\C {\B B}})$ is a base for $X$, then $(\B {\C B_A}, \beta_A, \beta_{\ \emptys_A}, \beta_{\C {\B B_A}})$ is a base for $A \subseteq X$, where if $\B A := (A, \emptys_A)$, $\beta_X (a) := (B_a^1, B_a^0)$, for every $a \in A$, $\beta_{\ \emptys_X}(u) :=  (B_u^1, B_u^0)$, for every $u \in  \emptys_A \subseteq  \emptys_X$, let
%we have that 
% and if 
$$\B {\C B_A} := \{\B B_A := 0_{\B A} \cup (\B A \cap \B B) =_{\mathsmaller{\C E^{\Disj}(X)}} (A \cap B^1, A \cap B^0) \mid \B B \in \C {\B B}\},$$
$$\beta_A \colon A \to \B {\C B_A}, \ \ \ \beta_A (a) := 0_{\B A} \cup (\B A \cap \beta_X(a)) =_{\mathsmaller{\C E^{\Disj}(X)}} (A \cap B_a^1, A \cap B_a^0),$$
$$\beta_{\ \emptys_A} \colon \emptys_A \to \B {\C B_A}, \ \ \ \beta_{\ \emptys_A} (u) := 0_{\B A} \cup (\B A \cap \beta_{\ \emptys_X}(a)) =_{\mathsmaller{\C E^{\Disj}(X)}} (A \cap B_u^1, A \cap B_u^0),$$
$$\beta_{\B B_A \B C_A} \colon A \cap (B^1 \cap C^1) \to \B {\C B_A}, \ \ \ \beta_{\B B_A \B C_A}(a) := 0_{\B A} \cup (\B A \cap \beta_{\B B \B C}(a)).$$
 \end{remark}

\begin{definition}[The canonical $\cs$-topology induced by a base $\B {\C B}$]\label{def: canonicalcs}
%	If $\B {\C B}$ is a base for a $\cs$-topology, l
	Let the formula
	$$T_{\B {\C B}}(\B G) := \forall_{\B x \ccin \B G}\exists_{\B B \in \B {\C B}}\big(x \cin \B B \wedge \B B \subseteq \B G\big).$$
	If $T_{\B {\C B}}(\B G)$, we call $\B G$ a $\B {\C T}_{\B {\C B}}$-open complemented subset of $X$. Let their set $\B {\C T}_{\B {\C B}}$ be the canonical $\cs$-topology on $X$ induced by the base $\B {\C B}$. We also call $\B {\C T}_{\B {\C B}}$ a $\csb$-topology, and the structure $(X, =_X, \neq_X; \B B, \beta_X, \beta_{\ \emptys_X}, \beta_{\B B})$, or simpler $(X, \B B)$, a $\csb$-topological space.	
	If $\B G$ is a $\B {\C T}_{\B {\C B}}$-open complemented subset of $X$, a modulus of openness $\op_{\B G}$ for $\B G$ is a function $\op_{\B G} \colon G^1 \to \B {\C B}$, such that
	$$\forall_{\B x \ccin \B G}\big(\B x \cin \op_{\B G}(x) \wedge \op_{\B G}(x) \subseteq \B G\big).$$ 
	If $\B G, \B H \in \B {\C T}_{\B {\C B}}$, such that $\B G =_{\C E^{\Disj}(X)} \B H$, we require that $\op_{\B G} =_{\X F\big(G^1, \B {\C B}\big)} \op_{\B H}$. A modulus of openness $\op_{\B G}$ for $\B G$ is called covering, if $\B G =_{\C E^{\Disj}(X)} \bigcup_{x^1 \in G^1}\op_{\B G}(x^1)$.
	\end{definition}

The extensionality of the relations $\B x \cin \B B$ and $\B B \subseteq \B G$ implies the extensionality of the property $T_{\B {\C B}}(\B G)$. The use of moduli of openness for $\B {\C T}_{\B {\C B}}$-open complemented subsets of $X$, together with the moduli $\beta_X, \beta_{\ \emptys_X}$, and $\beta_{\B {\C B}}$ in the definition of a base for a $\cs$-topology reflect the extension of the modular approach to the constructive theory of metric spaces to the constructive theory of $\cs$-topological spaces.

\begin{proposition}\label{prp: cstopb}
$\B {\C T}_{\B {\C B}}$ is a $\cs$-topology. Moreover, the following hold:\\[1mm]
\normalfont (i)
\itshape $\op_{(X, \ {\emptys_X})} := \beta_X$ is a module of openness for $(X, {\emptys_X})$, which is covering if 
%and only if 
$\beta_X$ is covering.\\[1mm]
\normalfont (ii)
\itshape $\op_{(\ {\emptys_X}, X)} := \beta_{\ \emptys_X}$ is a module of openness for $({\emptys_X}, X)$, which is covering if
% and only if 
 $\beta_{\ \emptys_X}$ is covering.\\[1mm]
\normalfont (iii)
\itshape If $\op_{\B G} \colon G^1 \to \B {\C B}$ and $\op_{\B H} \colon H^1 \to \B {\C B}$ are moduli of openness for $\B G$ and $\B H$, respectively, then $\op_{\B G \cap \B H} \colon G^1 \cap H^1 \to \B {\C B}$, defined by the rule  
$$\op_{\B G \cap \B H}(z^1) := \beta_{\op_{\B G}(z^1), \op_{\B H}(z^1)}(z^1), \ \ \ \ \ \ \ z^1 \in G^1 \cap H^1,$$
is a modulus of openness for $\B G \cap \B H$.\\[1mm]
\normalfont (iv)
\itshape If $\B B \in \B {\C B}$, then $\B B \in \B {\C T}_{\B {\C B}}$ and $\op_{\B B} \colon B^1 \to \B {\C B}$, defined by $\op_{\B B}(x^1) := \B B$, for every $x^1 \in B^1$, is a covering modulus of openness for $\B B$.\\[1mm]
\normalfont (v)
\itshape If $\beta_{\B {\C B}}$ is a covering base-modulus, then for every $\B B, \B C \in \B {\C B}$ the modulus of openness $\op_{\B B \cap \B C}$, defined in \normalfont (iii)\itshape, is covering.\\[1mm]
\normalfont (vi)
\itshape $($Myhill's unique choice principle$)$ Let $\C I := (I, =_I, \neq_I) \in \SetExtIneq$, such that $i \neq_I j$ is discrete, and $\B G_i \in \B {\C T}_{\B {\C B}}$, for every $i \in I$, such that $\B G_i \cap \B G_j \subseteq (\emptys_X, X)$, for every $i, j \in I$ with $i \neq_I j$. If $\op_{\B G_i}$ is a covering module of openness for $\B G_i$, then $\op_{\bigcup} \colon \bigcup_{i \in }G_i^1 \to \B {\C B}$, defined by $\op_{\bigcup}(x) := \op_{\B G_i}(x)$, where $i$ is the unique(up to equality) element of $I$ with $x \in G_i$, is a covering modulus of openness for $\bigcup_{i \in I}\B G_i$.
\end{proposition}

\begin{proof}
The proofs are similar to the proofs of Proposition~\ref{prp: metric} and Proposition~\ref{prp: basemodule1}(ii).
\end{proof}

\section{Pointwise-like and uniformly-like continuous functions between $\csb$-spaces}
\label{sec: csbcont}

In this section we translate the notions of pointwise and uniform continuity of functions between metric spaces to the case of $\csb$-spaces. The notions of pointwise-like and uniform-like continuity between
$\csb$-spaces do not form a direct generalisation of pointwise and uniform continuity between metric spaces, because the codomain of a modulus of openness for a complemented open set in a $\csb$-space is its base $\B {\C B}$, while the codomain of a modulus of openness for a complemented open set in a metric space is the interval $[0, + \infty)$. The behaviour of the moduli $\omega_{f,x}$ and $\Omega_f$ in the case of metric spaces cannot be extended to similar functions on basic complemented balls. In the case of $\csb$-spaces we work directly with similar functions on the corresponding sets of basic sets.

\begin{definition}\label{def: csbpcont} Let $(X, \B {\C B})$ and $(Y, \B {\C C})$ be $\csb$-topological spaces, and let $\B {\C T}_{\B {\C B}}$ and $\B {\C T}_{\B {\C C}}$ be their induced $\cs$-topologies, respectively. A continuous function, or a $\csb$-continuous function from $(X, \B {\C B})$ to $(Y, \B {\C C})$ is a pair $(f, \op_f)$, where $\fXY$ is a strongly extensional function, and 
$$\op_f \colon \bigcup_{\B H \in \B {\C T}_{\B {\C C}}}\Mod(\B H) \to \bigcup_{\B G \in \B {\C T}_{\B {\C B}}}\Mod(\B G),$$
$$(\B H, \op_{\B H}) \mapsto \big(f^{-1}(\B H), \op_{f^{-1}(\B H)}\big).$$
 Let $\csbC(X, Y)$ be their set, equipped with the equality of $\X F(X,Y)$. If $(Z, \B {\C D})$ is a $\csb$-space and $(g, \op_g)$ is a $\csb$-continuous function from $(Y, \B {\C C})$ to $(Z, \B {\C D})$, let $(g, \op_g) \circ (f, \op_f) := (g \circ f, \op_{g \circ f})$, where
 $$\op_{g \circ f} \colon \bigcup_{\B K \in \B {\C T}_{\B {\C D}}}\Mod(\B K) \to \bigcup_{\B G \in \B {\C T}_{\B {\C B}}}\Mod(\B G),$$
 $$(\B K, \op_{\B K}) \mapsto \big(f^{-1}(g^{-1}(\B K)), \op_{f^{-1}(g^{-1}(\B K))}\big).$$
Let $(\csbTop, \csbCont)$ be the category of $\csb$-topological spaces and $\csb$-continuous functions. We call a $\csb$-continuous function $(f, \op_f)$ from $(X, \B {\C B})$ to $(Y, \B {\C C})$ pointwise-like continuous on $X$, if there is a modulus of pointwise-like continuity 
 $\omega_f := (\omega_{f,x})_{x \in X}$, where for every $x \in X$ we have that 
 $$\omega_{f, x} \colon \B {\C C}(f(x)) \to \B {\C B}(x),$$
 $$\B {\C C}(f(x)) := \big\{\B C \in \B {\C C} \mid f(x) \in C^1\big\}, \ \ \ \B {\C B}(x) :=  \big\{\B B \in \B {\C B} \mid x \in B^1\big\},$$
such that for every $\B H \in \B {\C T}_{\B {\C C}}$ and every $x \in f^{-1}(H^1)$ we have that
	$$\op_{f^{-1}(\B H)}(x) =_{\B {\C B}(x)} \omega_{f,x}\big(\op_{\B H}(f(x))\big).$$
		\begin{center}
		\begin{tikzpicture}
			
			\node (E) at (0,0) {$f^{-1}(H^1) \ni x$};
			\node[right=of E] (K) {};
				\node[right=of K] (L) {};
			\node[right=of L] (H) {$f(x) \in H^1$};
			\node[below=of E] (B) {$\omega_{f,x}\big(\op_{\B H}(f(x))\big) \ni \B {\C B}(x)$};
			\node[below=of H] (C) {$\op_{\B H}(f(x)) \in \B {\C C}(f(x))$};

			\draw[|->] (E)--(H) node [midway,above] {$f$};
			\draw[|->] (E)--(B) node [midway,left] {$\op_{f^{-1}(\B H)}$};
			%\draw[->,bend right] (E) to node [midway,below] {$\id_Y$} (Y);
			\draw[|->] (H)--(C) node [midway,right] {$\op_{\B H}$};
			\draw[|->] (C)--(B) node [midway,below] {$\omega_{f,x}$};
			
		\end{tikzpicture}
	\end{center}
Let $\csbPointC(X, Y)$ be their set, equipped with the equality of $\X F(X,Y)$.
If $(g, \op_g)$ is a pointwise-like $\csb$-continuous function from $(Y, \B {\C C})$ to $(Z, \B {\C D})$, the modulus of pointwise-like continuity  for $g \circ f$ given by 
$\omega_{g \circ f, x} := \omega_{f,x} \circ \omega_{g, f(x)}$, for every $x \in X$.  
Let $(\csbTop, \csbPointCont)$ be the category of $\csb$-topological spaces and pointwise-like $\csb$-continuous functions.

A $\csb$-continuous function $(f, \op_f)$ from $(X, \B {\C B})$ to $(Y, \B {\C C})$ is uniformly-like continuous, if there is a modulus of uniform-like continuity 
	$$\Omega_f  \colon \B {\C C}(f(X)) \to \B {\C B},$$
	$$\B {\C C}(f(X)) := \big\{\B C \in \B {\C C} \mid \exists_{x \in X}\big(f(x) \in C^1\big)\big\},$$
	such that for every $\B H \in \B {\C C}$ we have that
	$$\op_{f^{-1}(\B H)} =_{\X F\big(f^{-1}(H^1), \B {\C B}\big)} \Omega_{f} \circ \op_{\B H} \circ f_{|f^{-1}(H^1)}.$$
			\begin{center}
		\begin{tikzpicture}
			
			\node (E) at (0,0) {$f^{-1}(H^1)$};
				\node[right=of E] (K) {};
			\node[right=of K] (H) {$H^1$};
			\node[below=of E] (B) {$\B {\C B}$};
				\node[below=of H] (C) {$\B {\C C}(f(X))$.};

			\draw[->] (E)--(H) node [midway,above] {$f_{|f^{-1}(H^1)}$};
			\draw[->] (E)--(B) node [midway,left] {$\op_{f^{-1}(H)^1}$};
			%\draw[->,bend right] (E) to node [midway,below] {$\id_Y$} (Y);
			\draw[->] (H)--(C) node [midway,right] {$\op_{\B H}$};
			\draw[->] (C)--(B) node [midway,below] {$\Omega_f$};
			
		\end{tikzpicture}
	\end{center}
	Let $\csbUnifC(X, Y)$ be their set, equipped with the equality of $\X F(X,Y)$.
	If $(g, \op_g)$ is a uniformly-like $\csb$-continuous function from $(Y, \B {\C C})$ to $(Z, \B {\C D})$, we define the modulus of uniform-like continuity  for the composite function $g \circ f$ by 
	$\Omega_{g \circ f}:= \Omega_{f} \circ \Omega_{g}$.  
		Let $(\csbTop, \csbUnifCont)$ be the category of $\csb$-topological spaces and uniformly-like $\csb$-continuous functions.
\end{definition}

If $\fXY$ is $\cs$-continuous, then if $\B C \colon f(x) \neq_{\B {\C C}} f(x{'})$, then $\op_{\B C}$ is defined on $C^1$, and $\op_{f^{-1}(\B C)}(x) \colon x \neq_{\B {\C B}} x{'}$. 
Clearly, we have the following subcategory-relations:
$$(\csbTop, \csbUnifCont) \leq (\csbTop, \csbPointCont) \leq (\csbTop, \csbCont) \leq (\csTop, \csCont).$$

\begin{proposition}\label{prp: csbpcont1}
Let $(f, \op_f)$ be a $\csb$-continuous function from $(X, \B {\C B})$ to $(Y, \B {\C C})$.\\[1mm]
\normalfont (i) 
\itshape $\omega_{f, x}(\B C) := \op_{f^{-1}(\B C)}(x)$, for every $\B C \in \B {\C C}(f(x))$, and $x \in f^{-1}(C^1)$.\\[1mm]
\normalfont (ii)
\itshape If $\B H \in \B {\C T}_{\B {\C C}}$ such that for every $y \in H^1$ we have that $\op_{\B H}(y) =_{\B {\C C}} \op_{\op_{\B H}(y)}(y)$, then for every $x \in f^{-1}(H^1)$ we have that $\op_{\B f^{-1}(H)}(x) =_{\B {\C B}} \op_{f^{-1}(\op_{\B H}(f(x))}(x)$.\\[1mm]
\normalfont (iii)
\itshape $\beta_X(x) := \omega_{f, x}\big(\beta_Y(f(x))\big)$, for every $x \in X$.\\[1mm]
\normalfont (iv)
\itshape $\beta_{\ \emptys_X}(x) := \omega_{f, x}\big(\beta_{\ \emptys_Y}(f(x))\big)$, for every $x \in \emptys_X$.\\[1mm]
\normalfont (v)
\itshape If $\B K, \B L \in \B {\C T}_{\B {\C C}}$ with moduli of openness $\op_{\B K}$ and $\op_{\B L}$, respectively, then for every $x \in f^{-1}(K^1) \cap f^{-1}(L^1)$ we have that
$$\beta_{\op_{f^{-1}(\B K), f^{-1}(\B L)}}(x) \  =_{\B {\C B}(x)}  \ \omega_{f, x}\big(\beta_{\op_{\B K}, \op_{\B L}}(f(x))\big).$$
\end{proposition}

\begin{proof}
We only show case (v). We have that
\begin{align*}
\beta_{\op_{f^{-1}(\B K), f^{-1}(\B L)}}(x) & =: \op_{f^{-1}(\B K), f^{-1}(\B L)}(x)\\
& =_{\B {\C B}(x)} \op_{f^{-1}(\B K \cap \B L)}(x)\\
& := \omega_{f, x}\big(\op_{\B K \cap \B L}(f(x))\big)\\
& := \omega_{f, x}\big(\beta_{\op_{\B K}, \op_{\B L}}(f(x))\big). \qedhere
\end{align*}
\end{proof}

\section{The product of two $\csb$-spaces and the weak $\csb$-topology}
\label{sec: csbprod}

In this section we define the product of two $\csb$-spaces, and we show how our choices in the description of the corresponding bases affect the type of continuity of the projection-functions.

\begin{definition}\label{def: prod1}
If $\C X := (X, =_X, \neq_X ; \B {\C B}, \beta_X, \beta_{\ \emptys_X}, \beta_{\B {\C B}})$ and $\C Y := (Y, =_Y, \neq_Y ; \B {\C C}, \beta_Y, \beta_{\ \emptys_Y}, \beta_{\B {\C C}})$ are in $\csbTop$, then the product $\csb$-space is the structure
$$\C {X \times Y} := (X \times Y, =_{X \times Y}, \neq_{X \times Y} ; \B {\C B \times \C C}, \beta_{X \times Y}, \beta_{ \ \emptys_{X \times Y}}, \beta_{\B {\C B \times \C C}}),$$
where $=_{X \times Y}$ is given in Definition~\ref{def: canonicalineq}, and the product-base is defined by
$$\B {\C B \times \C C} := \{\B B \times \B C \mid \B B \in \B {\C B}, \C C \in \B {\C C}\}.$$
Moreover, let $\beta_{X \times Y} \colon X \times Y \to \B {\C B \times \C C}$, defined, for every $(x, y) \in X \times Y$, by
$$\beta_{X \times Y}(x, y) := \beta_X(x) \times \beta_Y(y).$$
Let $\beta_{\ \emptys_{X \times Y}} \colon \emptys_{X \times Y} \to \B {\C B \times \C C}$, defined, for every $(u, w) \in \emptys_{X \times Y}$, by
$$\beta_{\ \emptys_{X \times Y}}(u, w) := \beta_{\ {\emptys_X}}(u) \times \beta_{\ \emptys_Y}(w).$$
Finally, let $\beta_{\B B_1 \times \B C_1, \B B_2 \times \B C_2} \colon [(\B B_1 \times \B C_1) \cap(\B B_2 \times \B C_2)]^1\to \B {\C B \times \C C}$, defined, by
$$\beta_{\B B_1 \times \B C_1, \B B_2 \times \B C_2}(x, y) := \beta_{\B B_1, \B B_2}(x) \times \beta_{\B C_1, \B C_2}(y),$$
for every $(x, y)$ in 
$$[(\B B_1 \times \B C_1) \cap(\B B_2 \times \B C_2)]^1 =_{\C E (X \times Y)} [(\B B_1 \cap \B B_2) \times (\B C_1 \times \B C_2)]^1 =_{\C E (X \times Y)} (B ^1_1 \cap B_2^1) \times (C^1_1 \cap C^1_2).$$
\end{definition}

The above definitions of the base-moduli for $\C {X \times Y}$ are well-defined by Propositions~\ref{prp: emptyprod}(iii, iv) and Proposition~\ref{prp: complemented4}(iv). As we have already mentioned in section~\ref{sec: setineq}, $\pr_X, \pr_Y$ are strongly extensional. Next we define a modulus of openness for $\pr_X$ (and for $\pr_Y$ we work similarly),
$$\op_{\pr_X} \colon \bigcup_{\B G \in \B {\C T_{\B {\C B}}}}\Mod(\B G) \to \bigcup_{\B H \in \B {\C T}_{\B {\C B \times \C C}}}\Mod(\B H),$$
$$(\B G, \op_{\B G}) \mapsto \big(\pr_X^{-1}(\B G), \op_{\pr_X^{-1}(\B G)}\big),$$
where $\op_{\B G} \colon G^1 \to \B {\C B}$, such that for every $x \in G^1$, we have that $x \in \op_{\B G}(x) := (B_x^1, B_x^0) \subseteq \B G$, and $\pr_X^{-1}(\B G) := \big(\pr_X^{-1}(G^1), \pr_X^{-1}(G^0)\big)$, where
$\pr_X^{-1}(G^1) =_{\C E(X \times Y)} G^1 \times Y$ and $\pr_X^{-1}(G^0) =_{\C E(X \times Y)} G^0 \times Y$. Let $\op_{\pr_X^{-1}(\B G)} \colon G^1 \times Y \to \B {\C B} \times \B {\C C}$, defined by
$$\op_{\pr_X^{-1}(\B G)}(x, y) := \op_{\B G}(x) \times \beta_Y (y),$$
for every $(x, y) \in G^1 \times Y$. If $ \beta_Y (y) := (C_y^1, C_y^0)$, we have that
$$\op_{\B G}(x) \times \beta_Y (y) := (B_x^1, B_x^0) \times (C_y^1, C_y^0) := \big(B_x^1 \times C_y^1, (B_x^0 \times Y) \cup (X \times C_y^0)\big).$$
Clearly, $(x, y) \in B_x^1 \times C_y^1$. Next we show that $\op_{\B G}(x) \times \beta_Y (y) \subseteq \pr_X^{-1}(\B G)$. By hypothesis $B_x^1 \subseteq G^1$, $G^0 \subseteq B_x^0$, $C_y^1 \subseteq Y$, and $\emptys_Y \subseteq C_y^0$. Hence, $B_x^1 \times C_y^1 \subseteq G^1 \times Y$ and $G^0 \times Y \subseteq B_x^0 \times Y \subseteq (B_x^0 \times Y) \cup (X \times C_y^0)$. Hence, we get the following result.

\begin{proposition}\label{prp: pcontproj}
If we define $\omega_{\pr_X, (x,y)} \colon \B {\C B}(\pr_X(x,y)) \to (\B {\C B} \times \B {\C C})(x,y)$ by
$$\omega_{\pr_X, (x,y)}(\B D_x) := \B D_x \times \beta_Y(y),$$
for every $\B D_x \in  \B {\C B}(x)$, then $\pr_X$ is pointwise-like continuous with modulus of pointwise-like continuity $\omega_{\pr_X, (x,y)}$. 
\end{proposition}

\begin{definition}\label{def: uniormbasemod}
Let $\C X \in \csbTop$. The base-modulus $\beta_X$ is called uniform, if $(X, \emptys_X) \in \B {\C B}$, and $\beta_X(x) := (X, \emptys_X)$, for every $x \in X$. Similarly, the base-modulus 
$\beta_{\ \emptys_X}$ is called uniform, if $(\emptys_X, X) \in \B {\C B}$, and 
$\beta_{\ \emptys_X}(u) := (\emptys_X, X)$,  for every $u \in \emptys_X$.
\end{definition}

Next result follows immediately from the discussion above.

\begin{proposition}\label{prp: ucontproj} If $\beta_Y$ is a a uniform base-modulus for $\C Y$, 
	let $\Omega_{\pr_X} \colon \B {\C B} \to (\B {\C B} \times \B {\C C})$, defined by
	$$\Omega_{\pr_X}(\B C) := \B C \times (Y, \emptys_Y),$$
	for every $\B C \in  \B {\C B}$. Then $\pr_X$ is uniformly-like continuous with modulus of uniform-like continuity $\Omega_{\pr_X}$. 
\end{proposition}

As one can add $(X, \emptys_X)$ in any given base, one can use a uniform base-modulus $\beta_X$ for $\C X$. Consequently, the category of these spaces with uniformly-like continuous functions as arrows has products. Next we define the weak $\csb$-topology and we show that it is the least $\csb$-topology that turns all given functions into uniformly-like continuous ones.

\begin{proposition}\label{prp: weak}
If $\C X \in \SetExtIneq$, $\C Y_i  := (Y_i, =_i, \neq_i ; \B {\C C_i}, \beta_{Y_i}, \beta_{\ \emptys_{Y_i}}, \beta_{\B {\C C_i}})\in \csbTop$, for every $i \in I$,  
and if $f_i \colon X \to Y_i$ are strongly extensional functions, then  $\C X := (X, =_X, \neq_X ; \B {\C B}, \beta_X, \beta_{\ \emptys_X}, \beta_{\B {\C B}}) \in \csbTop$, where  
$$\B {\C B} := \bigg\{\bigcap_{k = 1}^n f_{i_k}^{-1}(\B H_{i_k}) \mid n \in \Nat^+, i_1, \ldots, i_n \in I, \B H_{i_1} \in \B {\C T}_{\B {\C B_{i_1}}}, \ldots, \B H_{i_n} \in  \B {\C T}_{\B {\C B_{i_n}}}\bigg\} \cup \{(X, \emptys_X), (\ \emptys_X, X)\},$$
$\beta_X(x) := (X, \emptys_X)$, for every $x \in X$, $\beta_{\ \emptys_X}(u) := (\ \emptys_X, X)$, for every $u \in \emptys_X$, and $\beta_{\B B, \B B{'}}(x) := \B B \cap \B B{'} \in \B {\C B}$, for every $x \in [\B B \cap \B B{'}]^1$. With this $\csb$-topology on $X$ $f_i$ is uniformly-like continuous, for every $i \in I$. 
\end{proposition}
	
\begin{proof}
We only show the uniform-like continuity of $f_i$, where $i \in I$. We define a modulus
$$\op_{f_i} \colon \bigcup_{\B H_i \in \B {\C T}_{\B {\C C_i}}}\Mod(\B H_i) \to \bigcup_{\B G \in \B {\C T}_{\B {\C B}}}\Mod(\B G),$$
$$(\B H_i, \op_{\B H_i}) \mapsto \big(f_i^{-1}(\B H_i), \op_{f_i^{-1}(\B H_i)}\big),$$
where $\op_{f_i^{-1}(\B H_i)} \colon f_i^{-1}(\B H_i) \to \B {\C B}$, and for every $x \in  f_i^{-1}(H_i^1) \TOT f_i(x) \in H_i^1$ let
$$\op_{f_i^{-1}(\B H_i)} (x) := \Omega_{f_i}\big(\op_{\B H_i}(f_i(x))\big) := f_i^{-1}\big(\op_{\B H_i}(f_i(x))\big),$$
where $ \Omega_{f_i} \colon \B {\C C_i} \to \B {\C B}$ is a modulus of uniform-like continuity for $f_i$, given by the rule $\B C \mapsto f_i^{-1}(\B C)$.
As $\B {f(x_i)} \cin \op_{\B H_i}(f_i(x)) \subseteq \B H_i$, by Proposition~\ref{prp: ccpoint1}(viii) we get $\B x_i \cin f_i^{-1}\big(\op_{\B H_i}(f_i(x))\big) \subseteq f_i^{-1}(\B H_i)$.
\end{proof}

One can avoid the above uniform base-moduli and get a proof completely within $(\MIN)$, if there $i_0 \in I$, such that $f_{i_0}$ is also an injection. Then one can define $\beta_X(x) := f_{i_0}^{-1}(Y_{i_0}, \emptys_{Y_{i_0}})$ and $\beta_{\ \emptys_X}(u) := f_{i_0}^{-1}(\ \emptys_{Y_{i_0}}, Y_{i_0})$. Then $\B {\C B}$ is not necessary to contain $(X, \emptys_X)$ and $( \emptys_X, X)$.

%
%
%\section{New $\cs$-topological spaces from given ones}
%\label{sec: csnew}

%
%
%\section{Separation axioms}
%\label{sec: cssep}
%
%
%

\section{Concluding comments}
\label{sec: concl}

In~\cite{Bi73}, p.~28, Bishop described two obstacles in the constructivisation of general topology:
\begin{quote}
The constructivization of general topology is impeded by two obstacles. First, the classical notion of a topological space is not constructively viable. Second, even for metric spaces the classical notion of a continuous function is not constructively viable; the reason is that there is no constructive proof that a (pointwise) continuous function from a compact metric space to R is uniformly continuous. Since uniform continuity for functions on a compact space is the useful concept, pointwise continuity (no longer useful for proving uniform continuity) is left with no useful function to perform. Since uniform continuity cannot be formulated in the context of a general topological space, the latter concept also is left with no useful function to perform.
\end{quote}
A discussion on these two obstacles is already offered in~\cite{Pe23a}, as a motivation to the theory of Bishop spaces, a constructive, function-theoretic alternative to classical topology. Here we addressed Bishop's above formulation of the problem of constructivising general topology in a completely different way. Using his own deep idea of complemented subsets, instead of subsets, in case complementation is crucial to the fundamental concepts of a classical mathematical theory, we offered a new approach to both obstacles. The notion of a topology of open sets is not constructively viable, but, as we showed here, the notion of a topology of open complemented subsets is constructively viable. Apart from the impredicativity in some of the notions employed, such as the extensional powerset and the definition of extensional subsets of it through the scheme of separation, our elaboration of the theory of $\cs$-spaces is completely constructive. Regarding the second obstacle, we proposed the notion of  uniform-like continuity as a notion of continuity of uniform-type for complemented topological spaces. Although our analysis of continuity of functions here is within the ``$2$-dimensional'' framework of complemeted subsets, all concepts and facts on pointwise-like and uniform-like continuity can be translated to the ``$1$-dimensional'' framework of subsets, given the use of moduli of openness in the definition of an open subset with respect to a given base. This development is, in our view, interesting to a classical mathematician too, and shows that constructive mathematics is not always ``behind'' classical mathematics.

This paper is only the starting point of a constructive elaboration of the theory of topologies of open complemented subsets. Next we list some of the problems and tasks that arise naturally from this work.

\begin{enumerate}
	\item To study the relation between the $1$-dimensional subset theory and the $2$-dimensional complemented subset theory. $2$-dimensional dualities that are not possible to trace in the $1$-dimensional framework e.g., empty and coempty complemented subsets, inhabited and coinhabited complemented subsets etc., need to be investigated further, together with possibly new ones. Moreover, the lattice and algebraic properties of the various sets of tight subsets need to be determined.
	
	\item To examine the various categories of sets and functions introduced. The variety of these categories show how rich the constructive framework is, and how many concepts are revealed by this strong, ``negationless'' approach to mathematics.

	\item Many classical results on subsets hold constructively for total complemented subsets, and similarly many classical results on open subsets hold constructively for total, open complemented subsets. It is natural to investigate whether there is some common ``metamathematical'' formulation of this phenomenon. 
	
	\item The theory of $\cs$-topologies seems to be a natural framework for the study of continuous \textit{partial} functions, a topic also interesting form the classical topology-point of view!
	
	\item To examine the relation between the categories $(\csTop, \csCont)$ and $(\csbTop, \csbCont)$, and to investigate the properties of $(\csbTop, \csbPointCont)$ and $(\csbTop, \csbUnifCont)$.
	
	\item To delve into the big question of compactness of $\cs$-topologies. Can the $2$-dimensional framework of complemented subset theory offer new constructive answers to the  notion of a compact space?
	
	\item To include the study of separation axioms in the theory of $\cs$-spaces. As we have seen already, all related notions, such as point, belonging, and not belongin relations, disjointness etc., are positively translated in our $2$-dimensional framework. For example, a Hausdorff $\cs$-space 
	is a $\cs$-space such that for every $\B x, \B y \in \cPoint(X)$ with $x \neq_X y$ there are $\B G_x, \B G_y \in \B {\C T}$, such that $\B x \cin \B G_x, \B y \cin \B G_y$ and $\B G_x \Disj \B G_y$. If the original space is a $\csb$-space, then the Hausdorff-condition for $x, y \in X$ implies also that $x \neq_{\B {\C B}} y$.

	\item To incorporate into the theory of $\cs$-spaces furthe concepts and results from classical topological spaces. Here, although new insights were revealed, we only scratched the surface.
	
	\item To revisit many classical topological results with the distinction between pointwise-like and uniformly-like continuous functions at hand. It is expected that most of the continuous functions defined or shown to exist are uniformly-like continuous, rather than pointwise-like continuous. This was known to Bishop in the case of continuous functions between metric spaces, but many case-studies in our general framework indicate that the concept of uniform-like continuity pervades topological spaces too.

	\end{enumerate}

%
%
%
%If $f$ is uniformly continuous what extra properties does the inverse image and the direct image of $f$ possess?\\
%Functions that inverse or preserve balls?\\
%Inversion and preservation of points?\\
%Can we show  the inclusion $\emptys_{(X,d)} \subseteq [d_{x} \geq 1]$ in MIN?\\
%Can every $\cs$-space be seen as a $\csb$-space? Moreover, is every $\cs$-continuous also $\csb$-continuous i.e., can we always find a modulus of pointwise continuity??? Probably with the axiom of choice??

\end{document}